\documentclass[11pt,preprint]{imsart}

\usepackage[mathscr]{eucal}
\usepackage{graphicx}
\usepackage{sidecap}
\usepackage{latexsym}
\usepackage{psfrag}
\usepackage{epsfig}
\usepackage{enumerate}
\usepackage{amsmath}
\usepackage{amsthm}
\usepackage{amssymb}
\DeclareMathAlphabet{\mathpzc}{OT1}{pzc}{m}{it}

\include{psbox}

\usepackage{amsfonts}
\usepackage{color}
\usepackage[usenames,dvipsnames]{xcolor}
\usepackage[textsize=small]{todonotes}
\usepackage{mathrsfs}
\usepackage{hyperref}

\usepackage[numbers]{natbib}

\begin{document}

\newtheorem{theorem}{\bf Theorem}[section]
\newtheorem{definition}{\bf Definition}[section]
\newtheorem{corollary}{\bf Corollary}[section]
\newtheorem{lemma}{\bf Lemma}[section]
\newtheorem{assumption}{Assumption}
\newtheorem{condition}{\bf Condition}[section]
\newtheorem{proposition}{\bf Proposition}[section]
\newtheorem{definitions}{\bf Definition}[section]
\newtheorem{problem}{\bf Problem}

\newcommand{\skp}{\vspace{\baselineskip}}
\newcommand{\noi}{\noindent}
\newcommand{\osc}{\mbox{osc}}
\newcommand{\lfl}{\lfloor}
\newcommand{\rfl}{\rfloor}

\theoremstyle{remark}
\newtheorem{example}{\bf Example}[section]
\newtheorem{remark}{\bf Remark}[section]

\newcommand{\Amb}{{\mathbb{A}}}
\newcommand{\Bmb}{{\mathbb{B}}}
\newcommand{\Cmb}{{\mathbb{C}}}
\newcommand{\Dmb}{{\mathbb{D}}}
\newcommand{\Emb}{{\mathbb{E}}}
\newcommand{\Fmb}{{\mathbb{F}}}
\newcommand{\Gmb}{{\mathbb{G}}}
\newcommand{\Hmb}{{\mathbb{H}}}
\newcommand{\Imb}{{\mathbb{I}}}
\newcommand{\Jmb}{{\mathbb{J}}}
\newcommand{\Kmb}{{\mathbb{K}}}
\newcommand{\Lmb}{{\mathbb{L}}}
\newcommand{\Mmb}{{\mathbb{M}}}
\newcommand{\Nmb}{{\mathbb{N}}}
\newcommand{\Omb}{{\mathbb{O}}}
\newcommand{\Pmb}{{\mathbb{P}}}
\newcommand{\Qmb}{{\mathbb{Q}}}
\newcommand{\Rmb}{{\mathbb{R}}}
\newcommand{\Smb}{{\mathbb{S}}}
\newcommand{\Tmb}{{\mathbb{T}}}
\newcommand{\Umb}{{\mathbb{U}}}
\newcommand{\Vmb}{{\mathbb{V}}}
\newcommand{\Wmb}{{\mathbb{W}}}
\newcommand{\Xmb}{{\mathbb{X}}}
\newcommand{\Ymb}{{\mathbb{Y}}}
\newcommand{\Zmb}{{\mathbb{Z}}}

\newcommand{\Amc}{{\mathcal{A}}}
\newcommand{\Bmc}{{\mathcal{B}}}
\newcommand{\Cmc}{{\mathcal{C}}}
\newcommand{\Dmc}{{\mathcal{D}}}
\newcommand{\Emc}{{\mathcal{E}}}
\newcommand{\Fmc}{{\mathcal{F}}}
\newcommand{\Gmc}{{\mathcal{G}}}
\newcommand{\Hmc}{{\mathcal{H}}}
\newcommand{\Imc}{{\mathcal{I}}}
\newcommand{\Jmc}{{\mathcal{J}}}
\newcommand{\Kmc}{{\mathcal{K}}}
\newcommand{\lmc}{{\mathcal{l}}}
\newcommand{\Lmc}{{\mathcal{L}}}
\newcommand{\Mmc}{{\mathcal{M}}}
\newcommand{\Nmc}{{\mathcal{N}}}
\newcommand{\Omc}{{\mathcal{O}}}
\newcommand{\Pmc}{{\mathcal{P}}}
\newcommand{\Qmc}{{\mathcal{Q}}}
\newcommand{\Rmc}{{\mathcal{R}}}
\newcommand{\Smc}{{\mathcal{S}}}
\newcommand{\Tmc}{{\mathcal{T}}}
\newcommand{\Umc}{{\mathcal{U}}}
\newcommand{\Vmc}{{\mathcal{V}}}
\newcommand{\Wmc}{{\mathcal{W}}}
\newcommand{\Xmc}{{\mathcal{X}}}
\newcommand{\Ymc}{{\mathcal{Y}}}
\newcommand{\Zmc}{{\mathcal{Z}}}

\newcommand{\one} {{\boldsymbol{1}}}
\newcommand{\dbd} {{\boldsymbol{d}}}
\newcommand{\Dbd} {{\boldsymbol{D}}}
\newcommand{\ebd} {{\boldsymbol{e}}}
\newcommand{\fbd} {{\boldsymbol{f}}}
\newcommand{\Hbd} {{\boldsymbol{H}}}
\newcommand{\Kbd} {{\boldsymbol{K}}}
\newcommand{\Lbd} {{\boldsymbol{L}}}
\newcommand{\lbd} {{\boldsymbol{l}}}
\newcommand{\mbd} {{\boldsymbol{m}}}
\newcommand{\mubd} {{\boldsymbol{\mu}}}
\newcommand{\nbd} {{\boldsymbol{n}}}
\newcommand{\nubd} {{\boldsymbol{\nu}}}
\newcommand{\Nbd} {{\boldsymbol{N}}}
\newcommand{\Nalpha} {{\boldsymbol{N_\alpha}}}
\newcommand{\Nbeta} {{\boldsymbol{N_\beta}}}
\newcommand{\Ngamma} {{\boldsymbol{N_\gamma}}}
\newcommand{\pbd} {{\boldsymbol{p}}}
\newcommand{\Pbd} {{\boldsymbol{P}}}
\newcommand{\phibd} {{\boldsymbol{\phi}}}
\newcommand{\phihatbd} {{\boldsymbol{\hat{\phi}}}}
\newcommand{\psibd} {{\boldsymbol{\psi}}}
\newcommand{\psihatbd} {{\boldsymbol{\hat{\psi}}}}
\newcommand{\rbd} {{\boldsymbol{r}}}
\newcommand{\Rbd} {{\boldsymbol{R}}}
\newcommand{\sbd} {{\boldsymbol{s}}}
\newcommand{\ubd}{{\boldsymbol{u}}}
\newcommand{\Ubd} {{\boldsymbol{U}}}
\newcommand{\Wbd} {{\boldsymbol{W}}}
\newcommand{\Xbd} {{\boldsymbol{X}}}
\newcommand{\ybd} {{\boldsymbol{y}}}
\newcommand{\Ybd} {{\boldsymbol{Y}}}
\newcommand{\zbd} {{\boldsymbol{z}}}

\newcommand{\Amcbar}{{\bar{\Amc}}}
\newcommand{\bbar}{{\bar{b}}}
\newcommand{\Bbar}{{\bar{B}}}
\newcommand{\ebdbar}{{\bar{\ebd}}}
\newcommand{\Gmcbar} {{\bar{\Gmc}}}
\newcommand{\Ibar}{{\bar{I}}}
\newcommand{\Jmcbar} {{\bar{\Jmc}}}
\newcommand{\mubar} {{\bar{\mu}}}
\newcommand{\Mmbbar}{{\bar{\Mmb}}}
\newcommand{\omegabar} {{\bar{\omega}}}
\newcommand{\Omegabar} {{\bar{\Omega}}}
\newcommand{\pbar} {{\bar{p}}}
\newcommand{\Pbar} {{\bar{P}}}
\newcommand{\Phibar} {{\bar{\Phi}}}
\newcommand{\Pmbbar} {{\bar{\Pmb}}}
\newcommand{\Qmbbar} {{\bar{\Qmb}}}
\newcommand{\Rmcbar} {{\bar{\Rmc}}}
\newcommand{\sigmabar} {{\bar{\sigma}}}
\newcommand{\Smcbar} {{\bar{\Smc}}}
\newcommand{\taubar} {{\bar{\tau}}}
\newcommand{\Thetabar} {{\bar{\Theta}}}
\newcommand{\Vbar} {{\bar{V}}}
\newcommand{\Wbar}{{\bar{W}}}
\newcommand{\xbar} {{\bar{x}}}
\newcommand{\Xbar} {{\bar{X}}}
\newcommand{\Xitbar}{{\bar{X}_t^i}}
\newcommand{\Xjtbar}{{\bar{X}_t^j}}
\newcommand{\Xktbar}{{\bar{X}_t^k}}
\newcommand{\Xisbar}{{\bar{X}_s^i}}
\newcommand{\ybar} {{\bar{y}}}
\newcommand{\Ybar} {{\bar{Y}}}
\newcommand{\Zbar} {{\bar{Z}}}

\newcommand{\Ahat}{{\hat{A}}}
\newcommand{\bhat}{{\hat{b}}}
\newcommand{\fhat}{{\hat{f}}}
\newcommand{\ghat}{{\hat{g}}}
\newcommand{\hhat}{{\hat{h}}}
\newcommand{\Jmchat}{{\hat{\Jmc}}}
\newcommand{\lambdahat} {{\hat{\lambda}}}
\newcommand{\lhat}{{\hat{l}}}
\newcommand{\muhat} {{\hat{\mu}}}
\newcommand{\nuhat} {{\hat{\nu}}}
\newcommand{\phihat} {{\hat{\phi}}}
\newcommand{\psihat} {{\hat{\psi}}}
\newcommand{\rhohat} {{\hat{\rho}}}
\newcommand{\Yhat}{{\hat{Y}}}

\newcommand{\atil}{{\tilde{a}}}
\newcommand{\Atil}{{\tilde{A}}}
\newcommand{\btil}{{\tilde{b}}}
\newcommand{\Btil}{{\tilde{B}}}
\newcommand{\Ctil}{{\tilde{C}}}
\newcommand{\ebdtil}{{\tilde{\ebd}}}
\newcommand{\etatil}{{\tilde{\eta}}}
\newcommand{\ftil}{{\tilde{f}}}
\newcommand{\Ftil}{{\tilde{F}}}
\newcommand{\Fmctil}{{\tilde{\Fmc}}}
\newcommand{\Gmctil}{{\tilde{\Gmc}}}
\newcommand{\htil}{{\tilde{h}}}
\newcommand{\Itil}{{\tilde{I}}}
\newcommand{\Jtil}{{\tilde{J}}}
\newcommand{\Jmctil}{{\tilde{\Jmc}}}
\newcommand{\mutil}{{\tilde{\mu}}}
\newcommand{\Ntil}{{\tilde{N}}}
\newcommand{\nubdtil}{{\tilde{\nubd}}}
\newcommand{\omegatil}{{\tilde{\omega}}}
\newcommand{\Omegatil}{{\tilde{\Omega}}}
\newcommand{\Pmbtil}{{\tilde{\Pmb}}}
\newcommand{\qtil}{{\tilde{q}}}
\newcommand{\rbdtil}{{\tilde{\rbd}}}
\newcommand{\sigmatil}{{\tilde{\sigma}}}
\newcommand{\Tmctil}{{\tilde{\Tmc}}}
\newcommand{\thetatil}{{\tilde{\theta}}}
\newcommand{\Vmctil}{{\tilde{\Vmc}}}
\newcommand{\wtil}{{\tilde{w}}}
\newcommand{\Wtil}{{\tilde{W}}}
\newcommand{\xitil}{{\tilde{\xi}}}
\newcommand{\Xittil}{{\tilde{X}^{i}_t}}
\newcommand{\Xistil}{{\tilde{X}^{i}_s}}
\newcommand{\xtil}{{\tilde{x}}}
\newcommand{\Xtil}{{\tilde{X}}}
\newcommand{\ytil}{{\tilde{y}}}
\newcommand{\zetatil}{{\tilde{\zeta}}}
\newcommand{\Zittil}{{\tilde{Z}^{i,N}_t}}
\newcommand{\Zistil}{{\tilde{Z}^{i,N}_s}}
\newcommand{\Ztil}{{\tilde{Z}}}

\newcommand{\Rd}{{\Rmb^d}}
\newcommand{\Xt}{{\Xmb_t}}
\newcommand{\XT}{{\Xmb_T}}
\newcommand{\YT}{{\Ymb_T}}
\newcommand{\mumt}{{\mu^m(t)}}
\newcommand{\mums}{{\mu^m(s)}}
\newcommand{\Gammainf}{{\| \Gamma \|_\infty}}
\newcommand{\intXt}{{\int_\Xt}}
\newcommand{\intXT}{{\int_\XT}}

\newcommand{\EmbP}{{\Emb_{\Pmb^N}}}
\newcommand{\EmbPbar}{{\Emb_{\Pmbbar^N}}}
\newcommand{\EmbPG}{{\Emb_{\Pmb^N,\Gmc}}}

\newcommand{\UiNt}{{U^{i,N}_t}}
\newcommand{\UiNs}{{U^{i,N}_s}}
\newcommand{\Wit}{{W^i_t}}
\newcommand{\Wis}{{W^i_s}}
\newcommand{\Xit}{{X^i_t}}
\newcommand{\Xis}{{X^i_s}}
\newcommand{\Xjt}{{X^j_t}}
\newcommand{\Xjs}{{X^j_s}}
\newcommand{\Xkt}{{X^k_t}}
\newcommand{\Xks}{{X^k_s}}
\newcommand{\XiNt}{{X^{i,N}_t}}
\newcommand{\XiNs}{{X^{i,N}_s}}
\newcommand{\XjNt}{{X^{j,N}_t}}
\newcommand{\XjNs}{{X^{j,N}_s}}
\newcommand{\XkNt}{{X^{k,N}_t}}
\newcommand{\XkNs}{{X^{k,N}_s}}
\newcommand{\YiNt}{{Y^{i,N}_t}}
\newcommand{\YiNs}{{Y^{i,N}_s}}
\newcommand{\Yit}{{Y_t^i}}
\newcommand{\Yis}{{Y_s^i}}
\newcommand{\Yjt}{{Y_t^j}}
\newcommand{\Ykt}{{Y_t^k}}
\newcommand{\ZiNt}{{Z_t^{i,N}}}
\newcommand{\ZjNt}{{Z_t^{j,N}}}
\newcommand{\ZiNs}{{Z_s^{i,N}}}
\newcommand{\ZjNs}{{Z_s^{j,N}}}

\newcounter{constant}
\setcounter{constant}{1}
\newcommand{\const}{{\arabic{constant}}}
\newcommand{\constplus}{{\stepcounter{constant}\const}}
\newcommand{\constminus}{{\addtocounter{constant}{-1}\const}}

\newcommand{\beginsec}{
\setcounter{lemma}{0} \setcounter{theorem}{0}
\setcounter{corollary}{0} \setcounter{definition}{0}
\setcounter{example}{0} \setcounter{proposition}{0}
\setcounter{condition}{0} \setcounter{assumption}{0}
\setcounter{remark}{0} }

\numberwithin{equation}{section}
\numberwithin{lemma}{section}

\begin{frontmatter}
\title{Some Fluctuation Results for Weakly Interacting Multi-type Particle Systems}

 \runtitle{CLT for  Multi-type Particle Systems}

\begin{aug}
\author{ Amarjit Budhiraja and Ruoyu Wu\\ \ \\
}
\end{aug}

\today

\skp

\begin{abstract}
A collection of $N$-diffusing interacting particles where each particle belongs to one of $K$ different populations is considered. 
Evolution equation for a particle from population $k$ depends on the $K$ empirical measures of particle states corresponding to the various populations and the form of this dependence may change from one population to another.
In addition, the drift coefficients in the particle evolution equations may depend on a factor that is common to all particles and which is described through the solution of a stochastic differential equation coupled, through the empirical measures, with the $N$-particle dynamics. 
We are interested in the asymptotic behavior as $N\to \infty$. 
Although the full system is not exchangeable, particles in the same population have an exchangeable distribution. 
Using this structure, one can prove using standard techniques a law of large numbers result and a propagation of chaos property. 
In the current work we study fluctuations about the law of large number limit. 
For the case where the common factor is absent the limit is given in terms of a Gaussian field whereas in the presence of a common factor it is characterized through a mixture of Gaussian distributions. 
We also obtain, as a corollary, new fluctuation results for disjoint sub-families of single type particle systems, i.e.\ when $K=1$. 
Finally, we establish limit theorems for multi-type statistics of such weakly interacting particles, given in terms of multiple Wiener integrals.

\noi {\bf AMS 2000 subject classifications:} 60F05; 60K35; 60H30; 60J70.

\noi {\bf Keywords:} Mean field interaction, common factor, weakly interacting diffusions, propagation of chaos, central limit theorems, fluctuation limits, exchangeability, symmetric statistics, multiple Wiener integrals.
\end{abstract}

\end{frontmatter}


\section{Introduction}\label{section of introduction}\beginsec

For $N \ge 1$, let $Z^{1,N}, \dotsc, Z^{N,N}$ be $\Rmb^d$-valued stochastic processes, representing trajectories of $N$ particles, each of which belongs to one of $K$ types (populations) with the membership map denoted by $\pbd: \{1,\dotsc,N\} \to \{1,\dotsc,K\} \doteq \Kbd$, namely $i$-th particle is type $\alpha$ if $\pbd(i)=\alpha$.
The dynamics is given in terms of a collection of stochastic differential equations (SDE) driven by mutually independent Brownian motions (BM) with each particle's initial condition governed independently by a probability law that depends only on its type.
The $N$ stochastic processes interact with each other through the coefficients of the SDE which, for the $i$-th process, with $\pbd(i)=\alpha$, depend on not only the $i$-th state process and the $\alpha$-th type, but also the empirical measures $\mu^{\gamma,N}_t = \frac{1}{N_\gamma} \sum_{j:\pbd(j)=\gamma} \delta_{Z^{j,N}_t}$, $\gamma \in \{1,\dotsc,K\}$ and a stochastic process that is common to all particle equations (common factor).
Here $N_\gamma$ is the total number of particles that belong to the $\gamma$-th type.
The common factor is an $m$-dimensional stochastic process described once more through an SDE driven by a BM which is independent of the other noise processes.
Such stochastic systems are commonly referred to as {\em weakly interacting diffusion processes} and have a long history.
Classical works that study law of large number (LLN) results and central limit theorems (CLT) include McKean~\cite{McKean1966class,McKean1967propagation}, Braun and Hepp~\cite{BraunHepp1977vlasov}, Dawson~\cite{Dawson1983critical}, Tanaka~\cite{Tanaka1984limit}, Oelschal\"{a}ger~\cite{Oelschlager1984martingale}, Sznitman~\cite{Sznitman1984,Sznitman1991}, Graham and M\'{e}l\'{e}ard~\cite{GrahamMeleard1997}, Shiga and Tanaka~\cite{ShigaTanaka1985}, M\'{e}l\'{e}ard~\cite{Meleard1998}. 
All the above papers consider exchangeable populations, i.e.\ $K=1$, and a setting where the common factor is absent.
Motivated by approximation schemes for Stochastic Partial Differential Equations (SPDE) the papers~\cite{KurtzXiong1999, KurtzXiong2004} considered a setting where the common factor is modeled as a Brownian sheet that drives the dynamics of each particle. 
The paper~\cite{KurtzXiong1999} studied LLN and~\cite{KurtzXiong2004} considered fluctuations about the LLN limit. 
In a setting where particle dynamics are given through jump-diffusions and the common factor is described by another jump-diffusion that is coupled with the particle dynamics, a CLT was recently obtained in~\cite{BudhirajaSaha2014}.
The fluctuation limit theorems in~\cite{KurtzXiong2004,BudhirajaSaha2014}, although allowing for a common factor, are limited to exchangeable populations. 
The goal of the current work is to study fluctuations for multi-type particle systems. 
Since these systems are not exchangeable (there is also no natural way to regard the system as a $K$-vector of $d$-dimensional exchangeable particles), classical techniques for proving CLT, developed in the above papers~\cite{Sznitman1984,ShigaTanaka1985,Meleard1998,KurtzXiong2004}, are not directly applicable . 

Multi-type systems have been proposed as models in social sciences~\cite{ContucciGalloMenconi2008phase}, statistical mechanics~\cite{Collet2014macroscopic}, neurosciences~\cite{BaladronFaugeras2012}, etc.
In particular the last paper~\cite{BaladronFaugeras2012}, considers interacting diffusions of the form studied in the current work and establishes a LLN result and a propagation of chaos property. 
Our results in particular will provide asymptotic results on fluctuations from the LLN behavior obtained in~\cite{BaladronFaugeras2012}. 
Systems with a common factor also arise in many different areas. 
In Mathematical Finance, they have been used to model correlations between default probabilities of multiple firms~\cite{CvitanicMaZhang2012law}; in neuroscience modeling these arise as systematic noise in the external current input to a neuronal ensemble~\cite{FaugerasTouboulCessac2009constructive}; and for particle approximation schemes for SPDE, the common factor corresponds to the underlying driving noise in the SPDE~\cite{KurtzXiong1999, KurtzXiong2004}.  
The goal of this work is to study a family of multi-type weakly interacting diffusions with a common factor.
Our main objective is to establish a suitable CLT where the summands are quite general functionals of the trajectories of the particles with suitable integrability properties. 
Specifically, in the case where there is no common factor, letting $\Nalpha$ denote the set of indices $i$ such that $\pbd(i)=\alpha$ and 
\begin{equation*}
	\xi^N_\alpha (\phi) = \frac{1}{\sqrt{N_\alpha}} \sum_{i \in \Nalpha} \phi (Z^{i,N}),
\end{equation*}
for functions $\phi$ on the path space of the particles, we will establish (see Theorem~\ref{thmchap2multitype:CLT}) the weak convergence of the family $\{\xi^N_\alpha (\phi), \phi \in \Amc_\alpha, \alpha \in \Kbd\}$, in the sense of finite dimensional distributions, to a mean $0$ Gaussian field $\{\xi_\alpha (\phi), \phi \in \Amc_\alpha, \alpha \in \Kbd\}$. 
Here $\Amc_\alpha$ is a family of functions on the path space that are suitably centered and have appropriate integrability properties (see Section~\ref{secchap2multitype:CLT} for definitions.)  In the presence of a common factor the centering term is in general random (a function of the common factor)  and denoting by $\Vmc^N_\alpha(\phi)$ these suitably randomly centered and normalized sums of $\{\phi(Z^{i,N}), i \in \Nalpha\}$ (see~\eqref{eqchap3commonfactor:V alpha phi alpha}), we prove that under suitable conditions, $\{\Vmc^N_\alpha(\phi), \phi \in \Amcbar_\alpha, \alpha \in \Kbd\}$, where $\Amcbar_\alpha$ is once again a collection of functions on the path space with suitable integrability, converges in the sense of finite dimensional distributions to a random field whose distribution is given in terms of a Gaussian mixture. 

CLT established in this work also leads to new fluctuation results for the classical single type setting.  
Consider for example the case with no common factor and suppose one is interested in the joint asymptotic distribution of
 $(\xi^N_1(\phi_1), \xi^N_2(\phi_2))$, where 
\begin{equation*}
	\xi^N_1(\phi_1) = \frac{1}{\sqrt{\lfloor \lambda N \rfloor}} \sum_{i=1}^{\lfloor \lambda N \rfloor} \phi_1(Z^{i,N}), \quad \xi^N_2(\phi_2) = \frac{1}{\sqrt{N - \lfloor \lambda N \rfloor}} \sum_{i=\lfloor \lambda N \rfloor + 1}^N \phi_2(Z^{i,N})
\end{equation*}
and $\lambda \in (0,1)$. Existing results on central limit theorems for $K=1$ (eg.~\cite{ShigaTanaka1985, Sznitman1984, Meleard1998})  
do not give information on the joint limiting behavior of the above random variables.  Indeed, a naive guess that the propagation of
chaos property should imply the asymptotic independence of $\xi^N_1(\phi_1)$ and  $\xi^N_2(\phi_2))$ is in general false.
In Section~\ref{secchap2multitype:application} we will illustrate through a simple example how one can characterize the joint asymptotic distribution of the above pair. 

We are also interested in asymptotic behavior of path  functionals of particles of multiple type. Specifically, in the 
no common factor case, we will study in Section~\ref{secchap2multitype:generalization of CLT} the limiting
distribution of multi-type statistics of the form
\begin{equation} 
	\label{eq:intro1}
	\xi^N(\phi) = \frac{1}{\sqrt{N_1 \dotsm N_K}} \sum_{i_1 \in \boldsymbol{N_1}} \dotsb \sum_{i_K \in \boldsymbol{N_K}} \phi(Z^{i_1,N},\dots,Z^{i_K,N})
\end{equation}
where $\phi$ is a suitably centered function on the path space of $K \times d$ dimensional stochastic processes with appropriate integrability.
In the classical case (cf.~\cite{Dynkin1983}) where the particles have independent and identical dynamics, limit distributions of analogous statistics (with
$\boldsymbol{N_1} = \dotsb = \boldsymbol{N_K}$ ) are given through certain
multiple Wiener integrals (see Section~\ref{secchap2multitype:asymptotics of symmetric statistics}). 
In the setting considered here $\{Z^{i,N}\}$ are neither independent nor identical and we need to suitably extend the classical result for U-statistics to a multi-type setting and apply techniques as in the proof of Theorem~\ref{thmchap2multitype:CLT} to establish weak convergence of~\eqref{eq:intro1} and characterize the limit distributions.

The central idea in our proofs is a change of measure technique based on Girsanov's theorem that, in the case of single type populations with no common factor, goes back to the works~\cite{ShigaTanaka1985, Sznitman1984}.
For this case, the technique reduces the problem to a setting with i.i.d.\ particles and the main challenge is to suitably analyze the asymptotic behavior of the Radon-Nikodym derivative.  
In the multi-type setting (with no common factor), although one can similarly reduce to a problem for independent particles (however not identically distributed), the Radon-Nikodym derivative is given in terms of quantities that involve particles of different types and the classical results on asymptotics for symmetric statistics~\cite{Dynkin1983} that are used in~\cite{ShigaTanaka1985, Sznitman1984} are not directly applicable and one needs to suitably {\em lift} the problem to a higher dimensional space.  
For this, we first treat the easier setting where for all $\alpha, \gamma \in \Kbd$, $N_\alpha = N_\gamma$ so that under the new measure one can view the whole population as a collection of i.i.d.\ $K$-vector  particles with $k$-th coordinate belonging to the $k$-th type. 
We then extend the result to the setting where the number of particles of different  types may not be the same. 
For proofs in this general case, the property that, under the new measure, the particles are independent of each other together with classical results for multiple Wiener integrals, play an important role.
Note that when $N_\alpha = N_\gamma, \forall \, \alpha,\gamma \in \Kbd$, one can view $\{Z^{i,N}\}$ under the original measure as a $K$-vector of exchangeable particles.
For this case one can apply results of~\cite{ShigaTanaka1985,Sznitman1984} to deduce a CLT.
However it seems hard to use this result directly to treat the setting with different numbers of particles in different populations.
One of the key ingredients in the approach taken in~\cite{ShigaTanaka1985, Sznitman1984} for the single type setting is to identify the limit of Radon-Nikodym derivatives in terms of a suitable integral operator.
In the multi-type case (with no common factor) there is an analogous operator, however describing it and  the Hilbert space
on which it acts requires more work.  Roughly speaking the Hilbert space corresponds to the space of $\nuhat$-square integrable
functions on the path space of $K\times 2d$ dimensional continuous stochastic processes, where $\nuhat = \nu_1 \otimes \dotsb \otimes \nu_K$
and $\nu_{\alpha}$ for $\alpha \in \Kbd$ is  the limit law of processes associated with particles of type $\alpha$.

The setting with a common factor (see Sections~\ref{chap3commonfactor:weakly interacting particle systems with a common factor},~\ref{secchap3commonfactor:preparations} and~\ref{secchap3commonfactor:proofs}) presents several additional challenges. The factor process is fully coupled with the 
$N$-particle dynamics in the sense that the coefficients depend not only on its own state but also on the particle empirical measures $\{\mu^{\alpha,N}_t, \alpha \in \Kbd\}$. 
Due to the presence of the common factor, the limit of $\mu^{\alpha,N}_t$ will in general be a random measure.
As a result, the centering in the fluctuation theorem will typically be random as well and one expects the limit law for such fluctuations to be not Gaussian but rather a ``Gaussian mixture".
Our second main result (Theorem~\ref{thmchap3commonfactor:CLT}) establishes such a CLT under appropriate conditions. 
Proof proceeds by first considering a closely related collection of $N$ stochastic processes which, conditionally on a common factor, are independent with distribution that only depends on its type.
Unlike the setting with no common factor, there is no convenient change of measure under which this collection has the same law as that of the original collection of processes. 
We instead consider a change of measure, under which the distribution of the random centering remains invariant and one can view the distribution of a suitably {\em perturbed} form of the original vector of scaled and centered sums in terms of conditionally i.i.d.\ collections associated with the $K$ types of particles.  
Asymptotics of this latter collection can be analyzed in a manner analogous to the no-common-factor case, however in order to deduce the asymptotic properties of the original collection, one needs to carefully estimate the error introduced by the perturbation (cf.\ Section~\ref{secchap3commonfactor:study YN - Y}). 
Once again, to identify the limits of the Radon-Nikodym derivative, integral
operators on suitable $L^2$-path spaces are employed. 
A new aspect here is that we need to first consider limit laws conditional on the common factor (almost surely), which are now characterized in terms of certain random integral operators.
Finally, the description of the limit of the (un-conditional) Radon-Nikodym derivatives and synthesis of the limiting random field requires a measurable construction of multiple Wiener integrals with an additional random parameter (see Section~\ref{secchap3commonfactor:combining contributions from JN1 and JN2}).

Central limit theorems for systems of weakly interacting particles with a common factor have previously been studied in~\cite{KurtzXiong2004, BudhirajaSaha2014}. Both of these papers consider the case $K=1$. The paper~\cite{KurtzXiong2004}
establishes a fluctuation result for centered and scaled empirical measures in the space of suitably modified Schwartz distributions. Such a result does not immediately yield limit theorems for statistics that depend on particle states at multiple time instants (see~\cite{BudhirajaSaha2014} for a discussion of this point). In contrast the approach taken in~\cite{BudhirajaSaha2014}, and also in the current paper, allows to establish limit theorems for quite general square integrable path functionals. Furthermore, the paper~\cite{BudhirajaSaha2014} sketches an argument for recovering the weak convergence of empirical measures in the Schwartz space from the CLT for path functionals. 
Although not pursued here, with additional work and in an analogous manner, for the current setting as well one can establish convergence of suitably centered and normalized empirical measures $\{\mu^{\alpha,N}_t, \alpha \in \Kbd\}$ in an appropriate product Schwartz space.

The paper is organized as follows.
In Section~\ref{secchap2multitype:model} we begin by introducing our model of multi-type weakly interacting diffusions where the common factor is absent.
A basic condition (Condition~\ref{condchap2multitype:cond1}) is stated, under which both SDE for the pre-limit $N$-particle system and for the corresponding limiting nonlinear diffusion process have unique solutions, and a law of large numbers and a propagation of chaos property holds.
These results are taken from~\cite{BaladronFaugeras2012}.
For simplicity we consider here the case where the dependence of the drift coefficients on the empirical measures is linear. A more general nonlinear dependence is treated in Section~\ref{chap3commonfactor:weakly interacting particle systems with a common factor}.
In Section~\ref{secchap2multitype:fluctuations} we present a CLT (Theorem~\ref{thmchap2multitype:CLT}) for the no-common-factor case. 
As noted previously, this CLT gives new asymptotic results for a single type population as well. This point is illustrated through an example in Section~\ref{secchap2multitype:application}.  
We also give a limit theorem (Theorem~\ref{thmchap2multitype:general CLT}) for multi-type statistics of the form as in~\eqref{eq:intro1} in Section~\ref{secchap2multitype:generalization of CLT}. Proofs of Theorems~\ref{thmchap2multitype:CLT} and 
\ref{thmchap2multitype:general CLT} are provided in Section~\ref{secchap2multitype:proofs}.  
But before, in Section~\ref{chap3commonfactor:weakly interacting particle systems with a common factor}, we state our main results for the setting where a common factor is present. 
Specifically, Section~\ref{secchap3commonfactor:well posedness} states a basic condition (Condition~\ref{condchap3commonfactor:cond1}), which will ensure pathwise existence and uniqueness of solutions to both SDE for the $N$-particle system and a related family of SDE describing the limiting nonlinear Markov process.
Main result for the common factor setting is Theorem~\ref{thmchap3commonfactor:CLT} which appears in Section~\ref{secchap3commonfactor:CLT}.
For the sake of the exposition, some of the conditions for this theorem (Conditions~\ref{condchap3commonfactor:cond2} and~\ref{condchap3commonfactor:cond3}) and related notation appear later   in Section~\ref{secchap3commonfactor:preparations}. 
Finally Section~\ref{secchap3commonfactor:proofs} contains proofs of Theorem~\ref{thmchap3commonfactor:CLT} and related results.

The following notation will be used. 
Fix $T < \infty$. 
All stochastic processes will be considered over the time horizon $[0,T]$. 
We will use the notations $\{X_t\}$ and $\{X(t)\}$ interchangeably for stochastic processes.
Denote by $\Pmc(\Smb)$ the space of probability measures on a Polish space $\Smb$, equipped with the topology of weak convergence.
A convenient metric for this topology is the bounded-Lipschitz metric $d_{BL}$ defined as
\begin{equation*}
d_{BL}(\nu_1,\nu_2) = \sup_{f \in \Bmb_1} | \langle f, \nu_1 - \nu_2 \rangle |, \nu_1, \nu_2 \in \Pmc(\Smb),
\end{equation*}
where $\Bmb_1$ is the collection of all Lipschitz functions $f$ that are bounded by $1$ and such that the corresponding Lipschitz constant is also bounded by 1, and $\langle f,\mu \rangle = \int f \, d\mu$ for a signed measure $\mu$ on $\Smb$ and $\mu$-integrable function $f : \Smb \to \Rmb$.
Let $\Bmc(\Smb)$ be the Borel $\sigma$-field on $\Smb$.
Space of functions that are continuous from $[0,T]$ to $\Smb$ will be denoted as $\Cmb_\Smb[0,T]$ and equipped with the uniform topology. 
For $k \in \Nmb$, let $\Cmc_k$ denote $\Cmb_{\Rmb^k}[0,T]$ and let $\|f\|_{*,t} = \sup_{0 \le s \le t} \|f(s)\|$ for $f \in \Cmc_k$, $t \in [0,T]$.
Space of functions that are continuous and bounded from $\Smb$ to $\Rmb$ will be denoted as $\Cmb_b(\Smb)$.
For a bounded function $f$ from $\Smb$ to $\Rmb$, let $\|f\|_\infty = \sup_{x \in \Smb} |f(x)|$.
Probability law of an $\Smb$-valued random variable $\eta$ will be denoted as $\Lmc(\eta)$. 
Expected value under some probability law $\Pmb$ will be denoted as $\Emb_{\Pmb}$.
Convergence of a sequence $\{X_n\}$ of $\Smb$-valued random variables in distribution to $X$ will be written as $X_n \Rightarrow X$. 
For a $\sigma$-finite measure $\nu$ on a Polish space $\Smb$, denote by $L^2_{\Rmb^k}(\Smb,\nu)$ as the Hilbert space of $\nu$-square integrable functions from $\Smb$ to $\Rmb^k$. 
When $k$ and $\Smb$ are not ambiguous, we will merely write $L^2(\nu)$. 
The norm in this Hilbert space will be denoted as $\| \cdot \|_{L^2(\Smb,\nu)}$.
Given a sequence of random variables $X_n$ on probability spaces $(\Omega_n, \Fmc_n, P_n)$, $n \ge 1$, we say $X_n$ converges to $0$ in $L^2(\Omega_n, P_n)$ if $\int X_n^2 \, dP_n \to 0$ as $n \to \infty$.
We will usually denote by $\kappa, \kappa_1, \kappa_2, \dotsc$, the constants that appear in various estimates within a proof. 
The value of these constants may change from one proof to another.
Cardinality of a finite set $A$ will be denoted as $|A|$.
For $x,y \in \Rmb^d$, let $x \cdot y = \sum_{i=1}^d x_i y_i$.


\section{Model} \label{secchap2multitype:model}

Consider an infinite collection of particles where each particle belongs to one of $K$ different populations. 
Letting $\Kbd = \{1,\dotsc,K\}$, define a function $\pbd : \Nmb \to \Kbd$ by $\pbd(i) = \alpha$ if the $i$-th particle belongs to $\alpha$-th population.
For $N \in \Nmb$, let $\Nbd = \{1,\dotsc,N\}$.
For $\alpha \in \Kbd$, let $\Nalpha = \{i \in \Nbd: \pbd(i) = \alpha \}$ and denote by $N_\alpha$ as the number of particles belonging to the $\alpha$-th population, namely $N_\alpha = |\Nalpha|$. 
Assume that $N_\alpha/N \to \lambda_\alpha \in (0,1)$ as $N \to \infty$.
Let $\Nmb_\alpha = \{i \in \Nmb: \pbd(i) = \alpha \}$.

For fixed $N \ge 1$, consider the following system of equations for the $\Rmb^d$-valued continuous stochastic processes $Z^{i,N}$, $i \in \Nbd$, given on a filtered probability space $(\Omega, \Fmc, \Pbd, \{\Fmc_t\})$.
For $i \in \Nalpha$, $\alpha \in \Kbd$,
\begin{equation} \label{eqchap2multitype:ZiNt}
	\ZiNt = Z_0^{i,N} + \int_0^t f_\alpha(s,\ZiNs) \, ds + \int_0^t \sum_{\gamma=1}^{K} \langle b_{\alpha\gamma}(\ZiNs,\cdot), \mu^{\gamma,N}_s \rangle \, ds + W^i_t,
\end{equation}
where $f_\alpha : [0,T] \times \Rmb^d \to \Rmb^d$ and $b_{\alpha\gamma} : \Rmb^d \times \Rmb^d \to \Rmb^d$ are suitable functions, and $\mu^{\gamma,N}_s = \frac{1}{N_\gamma} \sum_{j \in \Ngamma} \delta_{Z^{j,N}_s}$. 
Here $\{W^i, i \in \Nbd\}$ are mutually independent $d$-dimensional $\{\Fmc_t\}$-Brownian motions. 
We assume that $\{Z_0^{i,N}: i \in \Nbd\}$ are $\Fmc_0$-measurable and mutually independent, with $\Lmc(Z_0^{i,N}) = \mu^\alpha_0$ for $i \in \Nalpha, \alpha \in \Kbd$.

Conditions on the various coefficients will be introduced shortly. 
Along with the $N$-particle equation $\eqref{eqchap2multitype:ZiNt}$ we will also consider a related infinite system of equations for $\Rmb^d$-valued continuous stochastic processes $X^i$, $i \in \Nmb$, given (without loss of generality) on $(\Omega, \Fmc, \Pbd, \{\Fmc_t\})$. 
For $\alpha \in \Kbd$ and $i \in \Nmb_\alpha$,
\begin{equation} \label{eqchap2multitype:Xit}
	\Xit = X_0^i + \int_0^t f_\alpha(s,\Xis) \, ds + \int_0^t \sum_{\gamma=1}^{K} \langle b_{\alpha\gamma}(\Xis,\cdot), \mu^\gamma_s \rangle \, ds + W^i_t,
\end{equation}
where $\mu^\alpha_s = \Lmc(X^i_s)$.
We assume that $\{X_0^i : i \in \Nmb\}$ are $\Fmc_0$-measurable and mutually independent, with $\Lmc(X_0^i) = \mu^\alpha_0$ for $\alpha \in \Kbd$ and $i \in \Nmb_\alpha$.

The existence and uniqueness of pathwise solutions of $\eqref{eqchap2multitype:ZiNt}$ and $\eqref{eqchap2multitype:Xit}$ can be shown under following conditions on the coefficients (cf.~\cite{BaladronFaugeras2012}).

\begin{condition} \label{condchap2multitype:cond1}
	$(a)$ For all $\alpha \in \Kbd$, the functions $f_\alpha$ are locally Lipschitz in the second variable, uniformly in $t \in [0,T]$: 
	For every $r > 0$ there exists $L_r \in (0,\infty)$ such that for all $x, y \in \{ z \in \Rmb^d : \| z \| \le r \}$, and $t \in [0,T]$:
	\begin{equation*}
		\|f_\alpha(t,x) - f_\alpha(t,y)\| \le L_r \|x-y\|.
	\end{equation*}
	$(b)$ There exists $L \in (0,\infty)$ such that for all $\alpha \in \Kbd$, $t \in [0,T]$ and $x \in \Rmb^d$: 
	\begin{equation*}
		x \cdot f_\alpha(t,x) \le L(1+\|x\|^2).
	\end{equation*}
	$(c)$ For all $\alpha, \gamma \in \Kbd$, $b_{\alpha\gamma}$ are bounded Lipschitz functions: 
	There exists $L \in (0,\infty)$ such that for all $(x,y)$ and $(x',y')$ in $\Rmb^d \times \Rmb^d$, we have:
	\begin{gather*}
		\| b_{\alpha\gamma}(x,y)\| \le L, \\
		\|b_{\alpha\gamma}(x,y) - b_{\alpha\gamma}(x',y')\| \le L(\|x-x'\|+\|y-y'\|).
	\end{gather*}
	
\end{condition}

The paper~\cite{BaladronFaugeras2012} also proves the following propagation of chaos property:
For any $n$-tuple $(i_1^N,\dotsc,i_n^N) \in \Nbd^n$, $i_1^N \ne i_2^N \ne \dotsb \ne i_n^N$, with $\pbd(i_j^N) = \alpha_j$, $j = 1,\dotsc,n$,
\begin{equation} \label{eqchap2multitype:propagation of chaos}
	\Lmc(\{Z^{i_1^N,N},\dotsc,Z^{i_n^N,N}\}) \to \mu^{\alpha_1} \otimes \dotsb \otimes \mu^{\alpha_n} \text{ as } N \to \infty,
\end{equation}
in $\mathcal{P}(\mathcal{C}_d^n)$, where for $\alpha \in \Kbd$ and $i \in \Nmb_\alpha$, $\mu^\alpha = \Lmc(X^i) \in \mathcal{P}(\mathcal{C}_d)$ and $X^i$ is as in~\eqref{eqchap2multitype:Xit}.
Using the above result and a straightforward argument similar to~\cite{Sznitman1991} one can show the following law of large numbers.
Proof is omitted.

\begin{theorem} \label{thmchap2multitype:LLN}
	Suppose Condition~\ref{condchap2multitype:cond1} holds.
	For all $f \in \Cmb_b(\Cmc_d^K)$, as $N \to \infty$,
	\begin{equation} \label{eqchap2multitype:LLN}
		\frac{1}{N_1 \dotsm N_K} \sum_{\boldsymbol{i}} f(Z^{i_1,N}, \dotsc, Z^{i_K,N}) \Rightarrow \langle f, \mu^1 \otimes \dotsb \otimes \mu^K \rangle,
	\end{equation}
	where the summation is taken over all $K$-tuples $\boldsymbol{i} = (i_1,\dots,i_K) \in \boldsymbol{N_1} \times \dotsb \times \boldsymbol{N_K}$.
\end{theorem}

Note that Theorem~\ref{thmchap2multitype:LLN} implies in particular that for all $\phi \in \Cmb_b(\Cmc_d)$ and $\alpha \in \Kbd$, as $N \to \infty$,
\begin{equation} \label{eqchap2multitype:LLN alpha}
	\frac{1}{N_\alpha} \sum_{i \in \Nalpha} \phi(Z^{i,N}) \Rightarrow \langle \phi, \mu^\alpha \rangle.
\end{equation}
In this work we are concerned with the fluctuations of expressions as in the LHS of~\eqref{eqchap2multitype:LLN} and~\eqref{eqchap2multitype:LLN alpha} about their law of large number limits given by the RHS of~\eqref{eqchap2multitype:LLN} and~\eqref{eqchap2multitype:LLN alpha}, respectively.
Limit theorems that characterize theses fluctuations are given in Theorems~\ref{thmchap2multitype:CLT} and~\ref{thmchap2multitype:general CLT}.
We will also establish central limit theorems in a setting where there is a common factor that appears in the dynamics of all the particles.
This setting, which in addition to a common factor will allow for a nonlinear dependence of the drift coefficient on the empirical measure, is described in Section~\ref{chap3commonfactor:weakly interacting particle systems with a common factor} and our main result for this case is given in Theorem~\ref{thmchap3commonfactor:CLT}.


\section{Fluctuations for multi-type particle system} \label{secchap2multitype:fluctuations}

Throughout this section Condition~\ref{condchap2multitype:cond1} will be assumed and will not be noted explicitly in statement of results.

\subsection{Canonical processes} \label{secchap2multitype:canonical processes}

We now introduce the following canonical spaces and stochastic processes. 
Let $\Omega_d = \Cmc_d \times \Cmc_d$. 
For $\alpha \in \Kbd$, denote by $\nu_\alpha \in \Pmc(\Omega_d)$ the law of $(W^i, X^i)$ where $i \in \Nalpha$ and $X^i$ is given by~\eqref{eqchap2multitype:Xit}. 
Let $\nuhat = \nu_1 \otimes \dotsb \otimes \nu_K$. 
Define for $N \in \Nmb$ the probability measure $\Pmb^N$ on $\Omega_d^N$ as
\begin{equation*}
	\Pmb^{N} = \Lmc \Big( (W^1, X^1), (W^2, X^2), \dotsc,(W^N, X^N) \Big) \equiv \nu_{\pbd(1)} \otimes \nu_{\pbd(2)} \otimes \dotsb \otimes \nu_{\pbd(N)},
\end{equation*}
For $\omegabar = (\omega_1, \omega_2, \dotsc, \omega_N) \in \Omega_d^N$, let $V^i(\omegabar) = \omega_i, i \in \Nbd$. 
Abusing notation,
\begin{equation} \label{eqchap2multitype:Vi}
	V^i = (W^i, X^i), \quad i \in \Nbd.
\end{equation}
Also define the canonical processes $V_* = (W_*, X_*)$ on $\Omega_d$ as
\begin{equation} \label{eqchap2multitype:V star}
	V_*(\omega) = (W_*(\omega), X_*(\omega)) = (\omega_1, \omega_2), \quad \omega = (\omega_1, \omega_2) \in \Omega_d.
\end{equation}


\subsection{Some integral operators} \label{secchap2multitype:some integral operators}

We will need the following functions for stating our first main theorem. 
Define for $\alpha, \beta, \gamma \in \Kbd$ and $t \in [0,T]$, the function $b_{\alpha\gamma,t}$ from $\Rmb^d \times \Rmb^d$ to $\Rmb^d$ as
\begin{equation} \label{eqchap2multitype:b alpha gamma t}
	b_{\alpha\gamma,t}(x,y) = b_{\alpha\gamma}(x,y) - \langle b_{\alpha\gamma}(x,\cdot), \mu^\gamma_t \rangle, \quad (x,y) \in \Rmb^d \times \Rmb^d.
\end{equation}
Define for $\alpha, \gamma \in \Kbd$, the function $h_{\alpha\gamma}$ from $\Omega_d \times \Omega_d$ to $\Rmb$ ($\nu_\alpha \otimes \nu_\gamma$ a.s.) as
\begin{equation}
	h_{\alpha\gamma}(\omega,\omega') = \sqrt{\frac{\lambda_\alpha}{\lambda_\gamma}} \int_0^T b_{\alpha\gamma,t} (X_{*,t}(\omega),X_{*,t}(\omega')) \cdot dW_{*,t}(\omega), \quad (\omega, \omega') \in \Omega_d \times \Omega_d. \label{eqchap2multitype:h alpha gamma}
\end{equation}
We now define {\em lifted} functions $\hhat_{\alpha\gamma}, \hhat$ from $\Omega_d^K \times \Omega_d^K$ to $\Rmb$ ($\nuhat \otimes \nuhat$ a.s.) for $\alpha, \gamma \in \Kbd$ as follows:
For $\omega = (\omega_1,\dotsc,\omega_K)$ and $\omega = (\omega'_1,\dotsc,\omega'_K)$ in $\Omega_d^K$, let
\begin{equation}
	\hhat_{\alpha\gamma}(\omega,\omega') = h_{\alpha\gamma}(\omega_\alpha,\omega'_\gamma), \quad \hhat(\omega,\omega') = \sum_{\alpha=1}^{K} \sum_{\gamma=1}^{K} \hhat_{\alpha\gamma}(\omega,\omega'). \label{eqchap2multitype:hhat}
\end{equation}

Now consider the Hilbert space $L^2(\Omega_d^K, \nuhat)$.
For $\alpha, \gamma \in \Kbd$, define integral operators $A_{\alpha\gamma}$ and $A$ on $L^2(\Omega_d^K, \nuhat)$ as follows: 
For $f \in L^2(\Omega_d^K, \nuhat)$ and $\omega \in \Omega_d^K$,
\begin{equation}
	A_{\alpha\gamma} f(\omega) = \int_{\Omega_d^K} \hhat_{\alpha\gamma}(\omega',\omega) f(\omega') \, \nuhat(d\omega'), \quad A f(\omega) = \int_{\Omega_d^K} \hhat(\omega',\omega) f(\omega') \, \nuhat(d\omega'). \label{eqchap2multitype:A} 
\end{equation}
It's clear from~\eqref{eqchap2multitype:hhat}-\eqref{eqchap2multitype:A} that $A = \sum_{\alpha=1}^{K} \sum_{\gamma=1}^{K} A_{\alpha\gamma}$. 
Denote by $I$ the identity operator on $L^2(\Omega_d^K, \nuhat)$.
The following lemma will be proved in Section~\ref{secchap2multitype:proof of lemma of trace}.

\begin{lemma} \label{lemchap2multitype:trace}
	$(a)$ $\displaystyle \textnormal{Trace}(AA^*) = \sum_{\alpha,\gamma=1}^K \frac{\lambda_\alpha}{\lambda_\gamma} \int_0^T \int_{\Omega_d^2} \| b_{\alpha\gamma,t}(X_{*,t}(\omega),X_{*,t}(\omega')) \|^2 \, \nu_\alpha(d\omega) \nu_\gamma(d\omega') \, dt$.
	$(b)$ \textnormal{Trace}$(A^n) = 0$ for all $n \ge 2$.
	$(c)$ $I - A$ is invertible.
\end{lemma}


\subsection{Central limit theorem} \label{secchap2multitype:CLT}

We can now present the first main result of this work.
For $\alpha \in \Kbd$, let $L^2_c(\Omega_d,\nu_\alpha)$ be the collection of $\phibd \in L^2(\Omega_d,\nu_\alpha)$ such that $\int_{\Omega_d} \phibd(\omega) \nu_\alpha(d\omega) = 0$.
Denote by $\Amc_\alpha$ the collection of all measurable maps $\phi : \Cmc_d \to \Rmb$ such that $\phibd \doteq \phi(X_*) \in L^2_c(\Omega_d, \nu_\alpha)$.
For $\phi \in \Amc_\alpha$, let 
$\xi^N_\alpha (\phi) = \frac{1}{\sqrt{N_\alpha}} \sum_{i \in \Nalpha} \phi (Z^{i,N})$. 
Given $\alpha \in \Kbd$ and $\phibd \in L^2(\Omega_d,\nu_\alpha)$, define lifted function $\phihatbd_\alpha \in L^2(\Omega_d^K, \nuhat)$ as follows:
\begin{equation}
	\phihatbd_\alpha(\omega) = \phibd(\omega_\alpha), \quad \omega = (\omega_1, \dotsc, \omega_K) \in \Omega_d^K. \label{eqchap2multitype:phihat alpha}
\end{equation}

\begin{theorem} \label{thmchap2multitype:CLT}
	$\{\xi_\alpha^N (\phi) : \phi \in \Amc_\alpha, \alpha \in \Kbd\}$ converges as $N \to \infty$ to a mean $0$ Gaussian field $\{\xi_\alpha(\phi) : \phi \in \Amc_\alpha, \alpha \in \Kbd\}$ in the sense of convergence of finite dimensional distributions, where for $\alpha, \gamma \in \Kbd$ and $\phi \in \Amc_\alpha, \psi \in \Amc_\gamma$,
	\begin{equation*}
		\Emb [\xi_\alpha(\phi) \xi_\gamma(\psi)] = \langle (I-A)^{-1} \phihatbd_\alpha, (I-A)^{-1} \boldsymbol{\psihat}_\gamma \rangle_{L^2(\Omega_d^K, \nuhat)},
	\end{equation*}
	with $\phibd = \phi(X_*)$, $\phihatbd_\alpha$ defined as in~\eqref{eqchap2multitype:phihat alpha}, and $\boldsymbol{\psi}$, $\boldsymbol{\psihat}_\gamma$ given similarly.
\end{theorem}

Proof of the theorem is given in Section~\ref{secchap2multitype:proof of CLT}.


\subsection{An application to single-type particle system} \label{secchap2multitype:application}

Consider the special case of~\eqref{eqchap2multitype:ZiNt} where $K = 1$.
Here, Theorem~\ref{thmchap2multitype:CLT} can be used to describe joint asymptotic distributions of suitably scaled sums formed from disjoint sub-populations of the particle system.
As an illustration, consider the single-type system given through the following collection of equations:
\begin{equation}\label{eqchap2multitype:ZiNt application}
	\ZiNt = Z_0^{i,N} + \int_0^t f(s,\ZiNs) \, ds + \frac{1}{N} \int_0^t \sum_{j=1}^N b(\ZiNs,\ZjNs) \, ds + W^i_t, \quad i \in \Nbd.
\end{equation}
Suppose that we are interested in the joint asymptotic distribution of $(\xi^N_1(\phi_1), \xi^N_2(\phi_2))$ for suitable path functionals $\phi_1$ and $\phi_2$, where 
\begin{equation*}
	\xi^N_1(\phi_1) = \frac{1}{\sqrt{\lfloor \lambda N \rfloor}} \sum_{i=1}^{\lfloor \lambda N \rfloor} \phi_1(Z^{i,N}), \quad \xi^N_2(\phi_2) = \frac{1}{\sqrt{N - \lfloor \lambda N \rfloor}} \sum_{i=\lfloor \lambda N \rfloor + 1}^N \phi_2(Z^{i,N}), \quad \lambda \in (0,1).
\end{equation*}
The propagation of chaos property in~\eqref{eqchap2multitype:propagation of chaos} says that any finite  collection of summands on the right sides of the above display are asymptotically mutually independent.
However, in general, this does not guarantee the asymptotic independence of $\xi^N_1(\phi_1)$ and $\xi^N_2(\phi_2)$.
In fact, one can use Theorem~\ref{thmchap2multitype:CLT} with $K=2$ to show that $(\xi^N_1(\phi_1), \xi^N_2(\phi_2))$ converges in distribution to a bivariate Gaussian random variable and provide expressions for the asymptotic covariance matrix.
We illustrate this below through a toy example.

\begin{example}
	Suppose $Z^{i,N}_0 \equiv 0$, $f \equiv 0$, $T=1$, $d=1$ and $b(x,y) = \beta(y)$, where $\beta$ satisfies Condition~\ref{condchap2multitype:cond1}$(2)$ with $b_{\alpha\gamma}$ replaced by $\beta$. 
	Further suppose that $\beta$ is an odd function, namely $\beta(y) = - \beta(-y)$ for all $y \in \Rmb$. 
	Due to the special form of $b$, one can explicitly characterize the measure $\nu_i$, $i = 1,2$.
	Indeed, noting that for a one dimensional BM, $\{W_t\}_{t \in [0,1]}$,
	\begin{equation*}
		W_t = \lambda \int_0^t \langle \beta, \mu_s \rangle \, ds + (1 - \lambda) \int_0^t \langle \beta, \mu_s \rangle \, ds + W_t, \quad \mu_s = \Lmc(W_s),
	\end{equation*}
	we see that $\nu_i = \Lmc(W, W)$ for $i=1,2$.
	Consider for $\omega \in \Cmc_1$,
	\begin{equation*}
		\phi_i(\omega) = \kappa_i \Big(\omega_1 - \int_0^1 \beta(\omega_t) \, dt\Big), \quad \kappa_i \in \Rmb, \quad i = 1,2.
	\end{equation*}
	Note that $\phi_i \in \Amc_i$ since $\beta$ is odd, $i=1,2$.
	For this example one can explicitly describe the asymptotic distribution of $(\xi^N_1(\phi_1), \xi^N_2(\phi_2))$ by regarding~\eqref{eqchap2multitype:ZiNt application} as a $2$-type population with $N_1 = \lfloor N \lambda \rfloor$ and $N_2 = N - \lfloor N \lambda \rfloor$.
	Following~~\eqref{eqchap2multitype:hhat} and~\eqref{eqchap2multitype:phihat alpha}, define for $\omega = (\omega_1,\omega_2) \in \Omega_d^2$,
	\begin{equation*}
		\phihatbd_i(\omega) = \kappa_i \Big(X_{*,1}(\omega_i) - \int_0^1 \beta(X_{*,t}(\omega_i)) \, dt\Big), \quad i =1,2,
	\end{equation*}
	and for $\omega' = (\omega'_1,\omega'_2) \in \Omega_d^2$,
	\begin{align*}
		\hhat(\omega',\omega) & = \int_0^1 \Big( \lambda \beta(X_{*,t}(\omega_1)) + \sqrt{\lambda(1-\lambda)} \beta(X_{*,t}(\omega_2)) \Big) \, dW_{*,t}(\omega_1') \\
		& \quad + \int_0^1 \Big( \sqrt{\lambda(1-\lambda)} \beta(X_{*,t}(\omega_1)) + (1-\lambda) \beta(X_{*,t}(\omega_2)) \Big) \, dW_{*,t}(\omega_2').
	\end{align*}
	The operator $A$ is then defined by~\eqref{eqchap2multitype:A}.
	The special form of $\phi_i$ allows us to determine $(I-A)^{-1} \phihatbd_i$, $i = 1,2$.
	Indeed, for $\omega = (\omega_1,\omega_2) \in \Omega_d^2$, let 
	\begin{align*}
		\psibd_1(\omega) & = \kappa_1 \Big[ X_{*,1}(\omega_1) + \int_0^1 \Big( \sqrt{\lambda(1-\lambda)} \beta(X_{*,t}(\omega_2)) - (1-\lambda) \beta(X_{*,t}(\omega_1)) \Big) \, dt \Big], \\
		\psibd_2(\omega) & = \kappa_2 \Big[ X_{*,1}(\omega_2) + \int_0^1 \Big( \sqrt{\lambda(1-\lambda)} \beta(X_{*,t}(\omega_1)) - \lambda \beta(X_{*,t}(\omega_2)) \Big) \, dt \Big].
	\end{align*}
	Then 
	\begin{align*}
		A \psibd_1(\omega) & = \int_{\Omega_d^2} \hhat(\omega',\omega) \psibd_1(\omega') \nuhat(d\omega') = \kappa_1 \int_0^1 \Big( \lambda \beta(X_{*,t}(\omega_1)) + \sqrt{\lambda(1-\lambda)} \beta(X_{*,t}(\omega_2)) \Big) \, dt, \\
		A \psibd_2(\omega) & = \kappa_2 \int_0^1 \Big( \sqrt{\lambda(1-\lambda)} \beta(X_{*,t}(\omega_1)) + (1-\lambda) \beta(X_{*,t}(\omega_2)) \Big) \, dt.
	\end{align*}
	This shows that $\psibd_i = (I-A)^{-1} \phihatbd_i$ for $i = 1$, $2$.
	From Theorem~\ref{thmchap2multitype:CLT} we then have that $(\xi^N_1(\phi_1), \xi^N_2(\phi_2))$ converges in distribution to a bivariate Gaussian random variable $(X,Y)$ with mean $0$ and covariance matrix $(\sigma_{ij})_{i,j=1}^2$, where
	\begin{align*}
		\sigma_{11} & = \kappa_1^2 \Emb \Big[ \Big( W_1 - (1-\lambda) \int_0^1 \beta(W_t) \, dt \Big)^2 + \lambda(1-\lambda) \Big( \int_0^1 \beta(W_t) \, dt \Big)^2 \Big], \\
		\sigma_{22} & = \kappa_2^2 \Emb \Big[ \Big( W_1 - \lambda \int_0^1 \beta(W_t) \, dt \Big)^2 + \lambda(1-\lambda) \Big( \int_0^1 \beta(W_t) \, dt \Big)^2 \Big], \\
		\sigma_{12} & = \sqrt{\lambda(1-\lambda)} \kappa_1\kappa_2 \Emb \Big[ \Big(2W_1 - \int_0^1 \beta(W_t) \, dt \Big) \int_0^1 \beta(W_t) \, dt \Big],
	\end{align*}
	and $W$ is a one dimensional BM.
\end{example}


\subsection{Asymptotics of symmetric statistics} \label{secchap2multitype:asymptotics of symmetric statistics}
The proofs of Theorems~\ref{thmchap2multitype:CLT} and~\ref{thmchap2multitype:general CLT} in Section~\ref{secchap2multitype:generalization of CLT} crucially rely on certain classical results from~\cite{Dynkin1983} on limit laws of degenerate symmetric statistics.
In this section we briefly review these results.

Let $\Smb$ be a Polish space and let $\{X_n\}_{n=1}^\infty$ be a sequence of i.i.d. $\Smb$-valued random variables having common probability law $\nu$.
For $k \in \Nmb$, let $L^2(\nu^{\otimes k})$ be the space of all real valued square integrable functions on $(\Smb^k, \Bmc(\Smb)^{\otimes k}, \nu^{\otimes k})$.
Denote by $L^2_c(\nu^{\otimes k})$ the subspace of {\em centered} functions, namely $\phi \in L^2(\nu^{\otimes k})$ such that for all $1 \le j \le k$,
\begin{equation*}
	\int_\Smb \phi(x_1,\dots,x_{j-1},x,x_{j+1},\dotsc,x_k) \, \nu(dx) = 0, \quad \nu^{\otimes k-1} \text{ a.e.\ } (x_1,\dots,x_{j-1},x_{j+1},\dots,x_K).
\end{equation*}
Denote by $L^2_{sym}(\nu^{\otimes k})$ the subspace of symmetric functions, namely $\phi \in L^2(\nu^{\otimes k})$ such that for every permutation $\pi$ on $\{1,\dots,k\}$,
\begin{equation*}
	\phi(x_1,\dots,x_k) = \phi(x_{\pi(1)},\dots,x_{\pi(k)}), \quad \nu^{\otimes k} \text{ a.e.\ } (x_1,\dots,x_k).
\end{equation*}
Also,  denote by $L^2_{c,sym}(\nu^{\otimes k})$ the subspace of centered symmetric functions in $L^2(\nu^{\otimes k})$, namely $L^2_{c,sym}(\nu^{\otimes k}) = L^2_c(\nu^{\otimes k}) \cap L^2_{sym}(\nu^{\otimes k})$.
Given $\phi_k \in L^2_{c,sym}(\nu^{\otimes k})$ define the symmetric statistic $\sigma^n_k (\phi_k)$ as
\begin{equation*}
	\sigma^n_k (\phi_k) =
	\begin{cases}
		\displaystyle \sum_{1 \le i_1 < i_2 < \dotsb < i_k \le n} \phi_k(X_{i_1},\dots,X_{i_k}) & \text{for } n \ge k, \\
		0 & \text{for } n<k.
	\end{cases}
\end{equation*}
In order to describe the asymptotic distributions of such statistics consider a Gaussian field $\{I_1(h); h \in L^2(\nu)\}$ such that
\begin{equation*}
	\Emb \big( I_1(h) \big) = 0, \: \Emb \big( I_1(h)I_1(g) \big) = \langle h,g \rangle_{L^2(\nu)}, \quad h,g \in L^2(\nu).
\end{equation*}
For $h \in L^2(\nu)$, define $\phi_k^h \in L^2_{sym}(\nu^{\otimes k})$ as
\begin{equation*}
	\phi_k^h(x_1,\dots,x_k) = h(x_1) \dotsm h(x_k)
\end{equation*}
and set $\phi_0^h = 1$.

The MWI of $\phi_k^h$, denoted as $I_k(\phi_k^h)$, is defined through the following formula.
For $k \ge 1$,
\begin{equation} \label{eqchap2multitype:MWI formula}
	I_k(\phi_k^h) = \sum_{j=0}^{\lfloor k/2 \rfloor} (-1)^j C_{k,j} ||h||^{2j}_{L^2(\nu)} (I_1(h))^{k-2j}, \text{ where } C_{k,j} = \frac{k!}{(k-2j)! 2^j j!}, j=0,\dots,\lfloor k/2 \rfloor.
\end{equation}
The following representation gives an equivalent way to characterize the MWI of $\phi_k^h $:
\begin{equation*}
	\sum_{k=0}^{\infty} \frac{t^k}{k!} I_k(\phi_k^h) = \exp \Big( tI_1(h) - \frac{t^2}{2} ||h||^2_{L^2(\nu)} \Big), t \in \Rmb,
\end{equation*}
where we set $I_0(\phi_0^h) = 1$. 
We extend the definition of $I_k$ to the linear span of $\{\phi_k^h, h \in L^2(\nu)\}$ by linearity. 
It can be checked that for all $f$ in this linear span,
\begin{equation} \label{eqchap2multitype:moments of MWI}
	\Emb(I_k(f))^2 = k! \, ||f||^2_{L^2(\nu^{\otimes k})}.
\end{equation}
Using this identity and standard denseness arguments, the definition of $I_k(f)$ can be extended to all $f \in L^2_{sym}(\nu^{\otimes k})$ and the identity $\eqref{eqchap2multitype:moments of MWI}$ holds for all $f \in L^2_{sym}(\nu^{\otimes k})$. 
The following theorem is taken from~\cite{Dynkin1983}.

\begin{theorem}[Dynkin-Mandelbaum~\cite{Dynkin1983}] \label{thmchap2multitype:Dynkin}
	Let $\{ \phi_k \}_{k=1}^\infty$ be such that, for each $k \ge 1$, $\phi_k \in L^2_{c,sym}(\nu^{\otimes k})$. 
	Then the following convergence holds as $n \to \infty$:
	\begin{equation*}
		\Big( n^{-\frac{k}{2}} \sigma^n_k (\phi_k) \Big)_{k \ge 1} \Rightarrow \Big( \frac{1}{k!} I_k(\phi_k) \Big)_{k \ge 1}
	\end{equation*}
	as a sequence of $\Rmb^\infty$ valued random variables.
\end{theorem}


\subsection{A limit theorem for statistics of multi-type populations} \label{secchap2multitype:generalization of CLT}

In this section we provide a CLT for statistics of the form given on the left side of~\eqref{eqchap2multitype:LLN}.

Denote by $L^2_c(\Omega_d^K, \nuhat)$ the subspace of  $\phibd \in L^2(\Omega_d^K, \nuhat)$ such that for each $\alpha \in \Kbd$,
\begin{equation*}
	\int_{\Omega_d} \phibd(\omega_1,\dots,\omega_K) \, \nu_\alpha(d\omega_\alpha) = 0, \nu_1 \otimes \dotsb \otimes \nu_{\alpha-1} \otimes \nu_{\alpha+1} \otimes \dotsb \otimes \nu_K \text{ a.e.\ } (\omega_1,\dots,\omega_{\alpha-1},\omega_{\alpha+1},\dots,\omega_K).
\end{equation*}
Denote by $\Amc^K$ the collection of measurable maps $\phi : \Cmc_d^K \to \Rmb$ such that $(\omega_1,\dotsc,\omega_K) \mapsto \phibd(\omega_1,\dotsc,\omega_K) \doteq \phi(X_*(\omega_1),\dotsc,X_*(\omega_K)) \in L^2_c(\Omega_d^K, \nuhat)$.
For $\phi \in \Amc^K$ and with $\{ Z^{i,N} \}$ as in~\eqref{eqchap2multitype:ZiNt}, let 
\begin{equation} \label{eqchap2multitype:xi N}
	\xi^N(\phi) = \frac{1}{\sqrt{N_1 \dotsm N_K}} \sum_{i_1 \in \boldsymbol{N_1}} \dotsb \sum_{i_K \in \boldsymbol{N_K}} \phi(Z^{i_1,N},\dots,Z^{i_K,N}).
\end{equation}
We will like to study the asymptotic behavior of $\xi^N(\phi)$ as $N \to \infty$.

Let $\{I_k(h), h \in L^2_{sym}(\Smb^k, \nu^{\otimes k})\}_{k \ge 1}$ be the MWI defined as in the previous section with $\Smb = \Omega_d^K$ and $\nu = \nuhat$, on some probability space $(\Omegatil,\tilde{\Fmc},\tilde{\Pmb})$.
Recall $\hhat$ defined in~\eqref{eqchap2multitype:hhat}.
Define $f \in L^2(\Omega_d^K \times \Omega_d^K, \nuhat \otimes \nuhat)$ as follows: For $(\omega,\omega') \in \Omega_d^K \times \Omega_d^K$
\begin{equation} \label{eqchap2multitype:f} 
	f(\omega, \omega') = \hhat(\omega, \omega') + \hhat(\omega', \omega) - \int_{\Omega_d^K} \hhat(\omega'',\omega) \hhat(\omega'', \omega') \, \nuhat(d\omega'').
\end{equation}
Define a random variable $J$ on $(\Omegatil,\tilde{\Fmc},\tilde{\Pmb})$ as
\begin{equation} \label{eqchap2multitype:J} 
	J = \frac{1}{2} [I_2(f)-\textnormal{Trace}(AA^*)],
\end{equation}
where $A$ is as introduced in~\eqref{eqchap2multitype:A}.
For $\phibd \in L^2_c(\Omega_d^K, \nuhat)$, define lifted symmetric function $\phihatbd^{sym} \in L^2((\Omega_d^K)^K,\nuhat^{\otimes K})$ as follows: 
For $\omega^j = (\omega_1^j,\dotsc,\omega_K^j) \in \Omega_d^K, j \in \Kbd$,
\begin{equation} \label{eqchap2multitype:phi hat sym}
	\phihatbd^{sym}(\omega^1,\omega^2,\dotsc,\omega^K) = \frac{1}{K!} \sum_\pi \phibd(\omega^{\pi(1)}_1,\omega^{\pi(2)}_2,\dotsc,\omega^{\pi(K)}_K),
\end{equation}
where the summation is taken over all permutations $\pi$ on $\Kbd$.
The following result characterizes the asymptotic distribution of $\xi^N(\phi)$.

\begin{theorem} \label{thmchap2multitype:general CLT}
	$\Emb_{\tilde{\Pmb}} \exp(J) = 1$.
	Denote by $\tilde{\Qmb}$ the probability measure on $(\Omegatil,\tilde{\Fmc})$ such that $d\tilde{\Qmb} = \exp(J) \cdot d\tilde{\Pmb}$.
	Then for $\phi \in \Amc^K$, $\Lmc(\xi^N(\phi)) \to \tilde{\Qmb} \circ (I_K(\phihatbd^{sym}))^{-1}$ as $N \to \infty$, where $\phibd = \phi(X_*(\cdot),\dotsc,X_*(\cdot))$  and $\phihatbd^{sym}$ is defined as in~\eqref{eqchap2multitype:phi hat sym}.
\end{theorem}

Proof of the theorem is given in Section~\ref{secchap2multitype:proof of general CLT}.

\begin{remark}
	In Theorems~\ref{thmchap2multitype:CLT} and~\ref{thmchap2multitype:general CLT}, the assumption that diffusion coefficients in~\eqref{eqchap2multitype:ZiNt} are identity matrices can be relaxed as follows:
	Replace $W^i_t$ in~\eqref{eqchap2multitype:ZiNt} by $\int_0^t \Sigma_\alpha(s,\ZiNs) \, d\Wtil^i_s + \int_0^t \sigma_\alpha(s,\ZiNs) \, d\Wbar^i_s$, where $\Sigma_\alpha : [0,T] \times \Rmb^d \to \Rmb^{d \times d}$ and $\sigma_\alpha : [0,T] \times \Rmb^d \to \Rmb^{d \times r}$ are bounded Lipschitz (in the second variable, uniformly in the first variable) functions such that $\Sigma_\alpha$ is invertible and $\Sigma_\alpha^{-1}$ is bounded. 
	Here $\{\Wtil^i, i \in \Nbd\}$ and $\{\Wbar^i, i \in \Nbd\}$ are mutually independent $d$ and $r$ dimensional BM respectively defined on $(\Omega, \Fmc, \Pbd, \{\Fmc_t\})$.
	Proof of the central limit theorem under this weaker condition is similar except that the change of measure introduced in Section~\ref{secchap2multitype:proof of CLT} takes a slightly different form.
	Specifically the function $b_{\alpha\gamma,t}$ appearing in~\eqref{eqchap2multitype:JN1} and~\eqref{eqchap2multitype:JN2} (see~\eqref{eqchap2multitype:b alpha gamma t}) is replaced by
	\begin{equation*}
		b_{\alpha\gamma,t}(x,y) = \Sigma_\alpha^{-1}(t,x) \Big( b_{\alpha\gamma}(x,y) - \langle b_{\alpha\gamma}(x,\cdot), \mu^\gamma_t \rangle \Big), \quad (x,y) \in \Rmb^d \times \Rmb^d, \quad \alpha, \gamma \in \Kbd,
	\end{equation*}
	and $\{W^i, i \in \Nbd\}$ is replaced by $\{ \Wtil^i, i \in \Nbd \}$.
	However, to keep the presentation less technical we will assume diffusion coefficients to be identity matrices throughout.
\end{remark}


\section{Weakly interacting particle systems with a common factor} \label{chap3commonfactor:weakly interacting particle systems with a common factor}

In this section we consider a setting where the drift coefficients of the interacting diffusions are suitable functions of not only the state of individual particles but also another stochastic process that represents a common source of random input to particle dynamics.
Additionally, unlike Section~\ref{secchap2multitype:model}, we consider a general nonlinear dependence of particle dynamics on empirical measures of particles of different types.

For fixed $N \ge 1$, consider the system of equations for the $\Rmb^d$-valued continuous stochastic processes $Z^{i,N}$, $i \in \Nbd$, and the $\Rmb^m$-valued continuous stochastic processes $U^N$, given on $(\Omega, \Fmc, \Pbd, \{\Fmc_t\})$.
For $i \in \Nalpha$,
\begin{gather} 
	\ZiNt = Z_0^{i,N} + \int_0^t b_\alpha (\ZiNs,U^N_s,\mubd^N_s) \, ds + W^i_t, \label{eqchap3commonfactor:ZiNt} \\
	U^N_t = U_0 + \int_0^t \bbar(U^N_s,\mubd^N_s) \, ds + \int_0^t \sigmabar(U^N_s,\mubd^N_s) \, d\Wbar_s, \label{eqchap3commonfactor:U N t} \\
	\mubd^N_t = (\mu^{1,N}_t, \dotsc, \mu^{K,N}_t),  \quad \mu^{\gamma,N}_t = \frac{1}{N_\gamma} \sum_{j \in \Ngamma} \delta_{\ZjNt} \label{eqchap3commonfactor:mu N t}.
\end{gather}
Here $(\Omega, \Fmc, \Pbd, \{\Fmc_t\})$, $\{ W^i \}$ are as in Section~\ref{secchap2multitype:model}, $\Wbar$ is an $m$-dimensional $\{\Fmc_t\}$-BM independent of $\{W^i\}$.
We assume that for $\alpha \in \Kbd$, $\{ Z^{i,N}_0 \}_{i \in \Nalpha}$ are i.i.d. with common distribution $\mu^\alpha_0$ and are also mutually independent. 
Moreover, $U_0$ is independent of $\{ Z^{i,N}_0 \}_{i \in \Nbd}$ and has probability distribution $\mubar_0$.
$\{Z^{i,N}_0\}_{i \in \Nbd}$ and $U_0$ are $\Fmc_0$ measurable.

We note that the model studied in Sections~\ref{secchap2multitype:model} and~\ref{secchap2multitype:fluctuations} corresponds to a setting where 
\begin{equation*}
b_\alpha(z,u,\nubd) = \sum_{\gamma=1}^K \langle b_{\alpha\gamma}(z,\cdot), \nu_\gamma \rangle, \quad \nubd = (\nu_1,\dotsc,\nu_K) \in [\Pmc(\Rmb^d)]^K, \quad (z,u) \in \Rmb^d \times \Rmb^m.
\end{equation*}

As in Section~\ref{secchap2multitype:model}, along with above $N$-particle equations, we will also consider a related infinite system of equations for the $\Rmb^d$-valued continuous stochastic processes $X^i$, $i \in \Nmb$, and the $\Rmb^m$-valued continuous stochastic processes $Y$, given on $(\Omega, \Fmc, \Pbd, \{\Fmc_t\})$. 
For $\alpha \in \Kbd$ and $i \in \Nmb_\alpha$,
\begin{gather}
	\Xit = X_0^i + \int_0^t b_\alpha(\Xis,Y_s,\mubd_s) \, ds + W^i_t, \label{eqchap3commonfactor:Xit} \\
	Y_t = Y_0 + \int_0^t \bbar(Y_s,\mubd_s) \, ds + \int_0^t \sigmabar(Y_s,\mubd_s) \, d\Wbar_s, \label{eqchap3commonfactor:Yt} \\
	\mubd_t = (\mu^1_t, \dotsc, \mu^K_t), \quad	\mu^{\gamma}_t = \lim_{N \to \infty} \frac{1}{N_\gamma} \sum_{j \in \Ngamma} \delta_{\Xjt}. \label{eqchap3commonfactor:mu t}
\end{gather}
Here $Y_0 = U_0$ and $\{X^i_0\}_{i \in \Nmb}$ are independent $\Fmc_0$-measurable random variables, with $\Lmc(X^i_0) = \mu^\alpha_0$ for $i \in \Nmb_\alpha$ and $\alpha \in \Kbd$,  and the limit in~\eqref{eqchap3commonfactor:mu t} is in a.s.\ sense.


\subsection{Well-posedness} \label{secchap3commonfactor:well posedness}

We now give conditions on the coefficient functions under which the systems of equations~\eqref{eqchap3commonfactor:ZiNt}-\eqref{eqchap3commonfactor:mu N t} and~\eqref{eqchap3commonfactor:Xit}-\eqref{eqchap3commonfactor:mu t} have unique pathwise solutions.
A pathwise solution of~\eqref{eqchap3commonfactor:Xit}-\eqref{eqchap3commonfactor:mu t} is a collection of continuous processes $(X^i,Y)$, $i \in \Nmb$, with values in $\Rmb^d \times \Rmb^m$ such that: $(a)$ $Y$ is $\{ \Gmc_t \}$-adapted, where $\Gmc_t = \sigma \{ Y_0, \Wbar_s, s \le t \}$; $(b)$ $\Xbd$ is $\{ \Fmc_t \}$-adapted where $\Xbd = (X_i)_{i \in \Nmb}$; $(c)$ stochastic integrals on the right sides of~\eqref{eqchap3commonfactor:Xit}-\eqref{eqchap3commonfactor:Yt} are well defined; (d) Equations~\eqref{eqchap3commonfactor:Xit}-\eqref{eqchap3commonfactor:mu t} hold a.s..
Uniqueness of pathwise solutions says that if $(\Xbd, Y)$ and $(\Xbd', Y')$ are two such solutions with $(\Xbd_0, Y_0) = (\Xbd_0', Y_0')$ then they must be indistinguishable. 
Existence and uniqueness of solutions to~\eqref{eqchap3commonfactor:ZiNt}-\eqref{eqchap3commonfactor:mu N t} are defined in a similar manner. 
In particular, in this case $(a)$ and $(b)$ are replaced by the requirement that $(Z^{i,N}, U^N)_{i \in \Nbd}$ are $\{ \Fmc_t \}$-adapted.

Define the metric $d_{BL}^{(K)}$ on $[\Pmc(\Rmb^d)]^K$ as $d_{BL}^{(K)}(\nubd,\nubd') = \sum_{\alpha=1}^K d_{BL}(\nu_\alpha,\nu_\alpha')$ for $\nubd,\nubd' \in [\Pmc(\Rmb^d)]^K$, $\nubd = (\nu_\alpha)_{\alpha \in \Kbd}$, $\nubd' = (\nu_\alpha')_{\alpha \in \Kbd}$. 
We now introduce conditions on the coefficients that will ensure existence and uniqueness of solutions.

\begin{condition} \label{condchap3commonfactor:cond1}
	There exists $L \in (0,\infty)$ such that
	
	\noi $(a)$ For all $x \in \Rmb^d$, $y \in \Rmb^m$, $\nubd \in [\Pmc(\Rmb^d)]^K$ and $\alpha \in \Kbd$,
	\begin{equation*}
		\max \{ \| b_\alpha(x,y,\nubd) \|, \| \bbar(y,\nubd) \|, \| \sigmabar(y,\nubd) \| \} \le L.
	\end{equation*}
	$(b)$ For all $x, x' \in \Rmb^d$, $y, y' \in \Rmb^m$, $\nubd, \nubd' \in [\Pmc(\Rmb^d)]^K$ and $\alpha \in \Kbd$,
	\begin{gather*}
		\| b_\alpha(x,y,\nubd) - b_\alpha(x',y',\nubd') \| \le L (\| x - x' \| + \| y - y' \| + d_{BL}^{(K)}(\nubd,\nubd')), \\
		\| \bbar(y,\nubd) - \bbar(y',\nubd') \| + \| \sigmabar(y,\nubd) - \sigmabar(y',\nubd') \| \le L (\| y - y' \| + d_{BL}^{(K)}(\nubd,\nubd')).
	\end{gather*}
\end{condition}

Under the above condition we can establish the following well-posedness result.
The proof follows along the lines of Theorem~$2.1$ of~\cite{BudhirajaSaha2014} and is therefore omitted.

\begin{theorem} \label{thmchap3commonfactor:wellposedness}
	Suppose that
	\begin{equation*}
		\sum_{\alpha \in \Kbd}\int \|x\|^2 \mu^\alpha_0(dx) + \int \|y\|^2 \mubar_0(dy) < \infty 
	\end{equation*}
	and Condition~\ref{condchap3commonfactor:cond1} holds.
	Then:
	
	\noi $(a)$ The systems of equations~\eqref{eqchap3commonfactor:Xit}-\eqref{eqchap3commonfactor:mu t} has a unique pathwise solution.
	
	\noi $(b)$ The systems of equations~\eqref{eqchap3commonfactor:ZiNt}-\eqref{eqchap3commonfactor:mu N t} has a unique pathwise solution.
\end{theorem}

\begin{remark} \label{rmkchap3commonfactor:rmk1}
	(i) We note that the unique pathwise solvability in Theorem~\ref{thmchap3commonfactor:wellposedness}$(a)$ implies that there is a measurable map $\Umc : \Rmb^m \times \Cmc_m \to \Cmc_m$ such that the process $Y$ that solves~\eqref{eqchap3commonfactor:Xit}-\eqref{eqchap3commonfactor:mu t} is given as $Y = \Umc(Y_0, \Wbar)$.

	\noi (ii) Recall that $\Gmc_s = \sigma \{ Y_0, \Wbar_r, r \le s \}, s \in [0,T]$. Let $\Gmc = \Gmc_T$. Then similar to Theorem~$2.3$ of~\cite{KurtzXiong1999} (see also Remark~$2.1$ in~\cite{BudhirajaSaha2014}) we can show that if $(\{X^i\},Y)$ is a solution of~\eqref{eqchap3commonfactor:Xit}-\eqref{eqchap3commonfactor:mu t} then
	\begin{equation*}
		\mu^\alpha_t = \Lmc(\Xit | \Gmc) = \Lmc(\Xit | \Gmc_t), t \in [0,T], i \in \Nalpha, \alpha \in \Kbd.
	\end{equation*}
	In particular, there are measurable maps $\Pi^\alpha : \Rmb^m \times \Cmc_m \to \Pmc(\Cmc_d)$ such that $\Pi^\alpha(Y_0,\Wbar) = \mu^\alpha$ a.s., where $\mu^\alpha = \Lmc(X^i | \Gmc)$ for $\alpha \in \Kbd$.
	Clearly $\mu^\alpha_t$ is identical to the marginal of $\mu^\alpha$ at time $t$.
\end{remark}


\subsection{Central limit theorem} \label{secchap3commonfactor:CLT}

The conditions and proof for the central limit theorem when there is a common factor require some notations which we prefer to introduce in later sections.
In this section we will present the basic limit result while referring the reader to Section~\ref{secchap3commonfactor:preparations} for precise conditions and definitions.

Recall $\Omega_d = \Cmc_d \times \Cmc_d$ and let $\Omega_m = \Cmc_m \times \Cmc_m$.
Define for $N \in \Nmb$ the probability measure $\Pmbbar^N$ on $(\Omegabar^N,\Bmc(\Omegabar^N))$, where $\Omegabar^N = \Omega_m \times \Omega_d^N$, as
\begin{equation} \label{eqchap3commonfactor:P N}
	\Pmbbar^{N} = \Lmc \big( (\Wbar, Y), (W^1, X^1), (W^2, X^2), \dotsc,(W^N, X^N) \big),
\end{equation}
where processes on the right side are as introduced below~\eqref{eqchap3commonfactor:mu N t}. 
Note that $\Pmbbar^N$ can be disintegrated as
\begin{equation*}
	\Pmbbar^N(d\omegabar \, d\omega_1 \dotsm d\omega_N) = \rho_{\pbd(1)}(\omegabar,d\omega_1) \rho_{\pbd(2)}(\omegabar,d\omega_2) \dotsm \rho_{\pbd(N)}(\omegabar,d\omega_N) \Pbar(d\omegabar),
\end{equation*}
where $\Pbar = \Lmc(\Wbar, Y)$ and for $\alpha \in \Kbd$, $\rho_\alpha(\omegabar,d\omega) = \Pi^{\alpha}((\omegabar^2_0,\omegabar^1))(d\omega)$ for $\Pbar$ a.e.\ $\omegabar = (\omegabar^1,\omegabar^2) \in \Omega_m$.

We can now present the second main result of this work.
Recall $X_*$ introduced in~\eqref{eqchap2multitype:V star}.
For $\alpha \in \Kbd$, denote by $\Amcbar_\alpha$ the collection of all measurable maps $\phi : \Cmc_d \to \Rmb$ such that $\phi(X_*) \in L^2(\Omega_d, \rho_\alpha(\omegabar,\cdot))$ for $\Pbar$ a.e.\ $\omegabar \in \Omega_m$.
For $\phi \in \Amcbar_\alpha$ and $\omegabar \in \Omega_m$, let
\begin{gather}
	m_{\phi}^\alpha(\omegabar) = \int_{\Omega_d} \phi(X_*(\omega)) \rho_\alpha(\omegabar,d\omega), \label{eqchap3commonfactor:m alpha phi alpha} \\ 
	\Vmc^N_\alpha(\phi) = \sqrt{N_\alpha} \big( \frac{1}{N_\alpha} \sum_{i \in \Nalpha} \phi(Z^{i,N}) - m_{\phi}^\alpha(V^0) \big), \label{eqchap3commonfactor:V alpha phi alpha}
\end{gather}
where $V^0 = (\Wbar, Y)$.
For $\phi_\alpha \in \Amcbar_\alpha$, $\alpha \in \Kbd$, let $\pi^N(\phi_1,\dotsc,\phi_K) \in \Pmc(\Rmb^K)$ be the probability distribution of $(\Vmc^N_1(\phi_1), \dotsc, \Vmc^N_K(\phi_K))$.
For $\omegabar \in \Omega_m$, let $\pi_\omegabar(\phi_1,\dotsc,\phi_K)$ be the $K$-dimensional multivariate normal distribution with mean $0$ and covariance matrix $\Sigma_\omegabar = (\Sigma^{\alpha\gamma}_\omegabar)_{\alpha,\gamma \in \Kbd}$ introduced in~\eqref{eqchap3commonfactor:Sigma omegabar}.
Let $\pi(\phi_1,\dotsc,\phi_K) \in \Pmc(\Rmb^K)$ be defined as
\begin{equation} \label{eqchap3commonfactor:pi phi 1 phi K}
	\pi(\phi_1,\dotsc,\phi_K) = \int_{\Omega_m} \pi_\omegabar(\phi_1,\dotsc,\phi_K) \Pbar(d\omegabar).
\end{equation}
The following is the second main result of this work. The proof is given in Section~\ref{secchap3commonfactor:proofs}.

\begin{theorem} \label{thmchap3commonfactor:CLT}
	Suppose Conditions~\ref{condchap3commonfactor:cond1},~\ref{condchap3commonfactor:cond2} and~\ref{condchap3commonfactor:cond3} hold. Then for all $\alpha \in \Kbd$ and $\phi_\alpha \in \Amcbar_\alpha$, $\pi^N(\phi_1,\dotsc,\phi_K)$ converges weakly to $\pi(\phi_1,\dotsc,\phi_K)$ as $N \to \infty$, where $\pi(\phi_1,\dotsc,\phi_K) \in \Pmc(\Rmb^K)$ is as in~\eqref{eqchap3commonfactor:pi phi 1 phi K}.
\end{theorem}

\begin{remark} \label{rmkchap3commonfactor:rmk2}
	(i) We note that unlike in Theorem~\ref{thmchap2multitype:CLT}, there is a random centering term $m_{\phi_\alpha}^\alpha(V^0)$ in the limit theorem (cf.~\eqref{eqchap3commonfactor:V alpha phi alpha}).
	Also, as seen for the definition of $\pi$ in~\eqref{eqchap3commonfactor:pi phi 1 phi K}, the asymptotic distribution of $\Vmc^N_\alpha(\phi_\alpha)$ is not Gaussian but rather a mixture of Gaussian distributions.
	
	\noi (ii) One can also establish a result similar to Theorem~\ref{thmchap2multitype:general CLT} for the setting with a common factor.
	We leave the formulation and proof of such a result to the reader.
\end{remark}

{\bf Organization of the paper:} Rest of the paper is organized as follows.
In Section~\ref{secchap2multitype:proofs} we prove Lemma~\ref{lemchap2multitype:trace}, Theorems~\ref{thmchap2multitype:CLT} and~\ref{thmchap2multitype:general CLT}.
These results concern the setting where there is no common factor present.
Proof of Lemma~\ref{lemchap2multitype:trace} is in Section~\ref{secchap2multitype:proof of lemma of trace} while Theorems~\ref{thmchap2multitype:CLT} and~\ref{thmchap2multitype:general CLT} are proved in Sections~\ref{secchap2multitype:proof of CLT} and~\ref{secchap2multitype:proof of general CLT} respectively.
Sections~\ref{secchap3commonfactor:preparations} and~\ref{secchap3commonfactor:proofs} treat the setting of a common factor described in Section~\ref{chap3commonfactor:weakly interacting particle systems with a common factor}.
In Section~\ref{secchap3commonfactor:assumptions for CLT} we will introduce the main assumptions, Conditions~\ref{condchap3commonfactor:cond2} and~\ref{condchap3commonfactor:cond3}, needed for Theorem~\ref{thmchap3commonfactor:CLT} and Section~\ref{secchap3commonfactor:canonical processes} -~\ref{secchap3commonfactor:some random integral operators} will introduce some functions and (random) integral operators that will be needed in the proof of the theorem.
Proof of Theorem~\ref{thmchap3commonfactor:CLT} and also the proof of Lemma~\ref{lemchap3commonfactor:Trace} will be presented in Section~\ref{secchap3commonfactor:proofs}.


\section{Proofs for the no-common-factor case} \label{secchap2multitype:proofs}

For $N \in \Nmb$, let $\Omega_d^N, \Pmb^N, V_*, V^i, i \in \Nbd, \nu_\alpha, \alpha \in \Kbd, \nuhat$ be as in Section~\ref{secchap2multitype:canonical processes}.


\subsection{Proof of Lemma~\ref{lemchap2multitype:trace}} \label{secchap2multitype:proof of lemma of trace}

Recall the definition of $b_{\alpha\gamma,t}$, $\hhat_{\alpha\gamma}$ and $\hhat$ in~\eqref{eqchap2multitype:b alpha gamma t} and~\eqref{eqchap2multitype:hhat}.
Define for $\alpha,\beta,\gamma \in \Kbd$, the functions $b_{\alpha\beta\gamma,t} : \Rmb^d \times \Rmb^d \to \Rmb$ and $l_{\alpha,\beta\gamma}$ from $\Omega_d \times \Omega_d$ to $\Rmb$ ($\nu_\beta \otimes \nu_\gamma$ a.s.) as follows:
\begin{gather}
	b_{\alpha\beta\gamma,t} (x,y) = \int_{\Rmb^d} b_{\alpha\beta,t}(z,x) \cdot b_{\alpha\gamma,t}(z,y) \, \mu^\alpha_t(dz), \quad (x,y) \in \Rmb^d \times \Rmb^d, \label{eqchap2multitype:b alpha beta gamma t} \\
	l_{\alpha,\beta\gamma}(\omega,\omega') = \frac{\lambda_\alpha}{\sqrt{\lambda_\beta \lambda_\gamma}} \int_0^T b_{\alpha\beta\gamma,t}(X_{*,t}(\omega), X_{*,t}(\omega')) \, dt, \quad (\omega,\omega') \in \Omega_d \times \Omega_d. \label{eqchap2multitype:l alpha beta gamma}
\end{gather}
Also define lifted function $\lhat_{\alpha,\beta\gamma}$ from $\Omega_d^K \times \Omega_d^K$ to $\Rmb$ ($\nuhat \otimes \nuhat$ a.s.) as
\begin{equation}
	\lhat_{\alpha,\beta\gamma}(\omega,\omega') = l_{\alpha,\beta\gamma}(\omega_\beta,\omega_\gamma'), \quad \omega = (\omega_1,\dotsc,\omega_K) \in \Omega_d^K, \omega' = (\omega_1',\dotsc,\omega_K') \in \Omega_d^K. \label{eqchap2multitype:lhat alpha beta gamma}
\end{equation}
It is easily seen that for $\alpha, \alpha', \beta, \gamma \in \Kbd$,
\begin{equation}
	\int_{\Omega_d^K} \hhat_{\alpha\beta}(\omega'',\omega) \hhat_{\alpha'\gamma}(\omega'',\omega') \, \nuhat(d\omega'') = \one_{\{ \alpha = \alpha' \}} \lhat_{\alpha,\beta\gamma}(\omega,\omega') . \label{eqchap2multitype:relation between lhat alpha beta gamma and hhat}
\end{equation}
Recall the definition of the integral operator $A_{\alpha\gamma}$ for $\alpha, \gamma \in \Kbd$ in~\eqref{eqchap2multitype:A}, and note that its adjoint $A_{\alpha\gamma}^* : L^2(\Omega_d^K, \nuhat) \to L^2(\Omega_d^K, \nuhat)$ is given as follows:
For $f \in L^2(\Omega_d^K, \nuhat)$ and $\omega \in \Omega_d^K$,
\begin{equation*}
	A_{\alpha\gamma}^* f(\omega) = \int_{\Omega_d^K} \hhat_{\alpha\gamma}(\omega,\omega') f(\omega') \, \nuhat(d\omega').
\end{equation*}
Then for $\alpha, \alpha', \beta, \gamma \in \Kbd$, the operator $A_{\alpha\beta} A_{\alpha'\gamma}^* : L^2(\Omega_d^K, \nuhat) \to L^2(\Omega_d^K, \nuhat)$ is given as follows:
For $f \in L^2(\Omega_d^K, \nuhat)$ and $\omega \in \Omega_d^K$,
\begin{align*}
	A_{\alpha\beta} A_{\alpha'\gamma}^* f(\omega) 
	& = \int_{\Omega_d^K} \Big( \int_{\Omega_d^K} \hhat_{\alpha\beta}(\omega',\omega) \hhat_{\alpha'\gamma}(\omega',\omega'') \, \nuhat(d\omega') \Big) f(\omega'') \, \nuhat(d\omega'') \\
	& = \one_{\{\alpha = \alpha'\}} \int_{\Omega_d^K} \lhat_{\alpha,\beta\gamma}(\omega,\omega'') f(\omega'') \, \nuhat(d\omega''),
\end{align*}
where the last equality follows from~\eqref{eqchap2multitype:relation between lhat alpha beta gamma and hhat}.
In particular, $A_{\alpha\beta} A_{\alpha'\gamma}^* = 0$ if $\alpha \ne \alpha'$. 
Moreover, for $\alpha, \gamma, \gamma' \in \Kbd$,
\begin{align*}
	& \textnormal{Trace}(A_{\alpha \gamma} A_{\alpha \gamma'}^*) \\
	& \quad = \int_{\Omega_d^K \times \Omega_d^K} \hhat_{\alpha\gamma'}(\omega,\omega') \hhat_{\alpha\gamma}(\omega,\omega') \, \nuhat(d\omega) \,  \nuhat(d\omega') \\
	& \quad = \frac{\lambda_\alpha}{\sqrt{\lambda_{\gamma'} \lambda_\gamma}} \int_{\Omega_d^K \times \Omega_d^K} \int_0^T b_{\alpha\gamma',t}(X_{*,t}(\omega_\alpha),X_{*,t}(\omega_{\gamma'}')) \cdot b_{\alpha\gamma,t}(X_{*,t}(\omega_\alpha),X_{*,t}(\omega_\gamma')) \, dt \, \nuhat(d\omega) \,  \nuhat(d\omega') \\
	& \quad = \one_{\{\gamma' = \gamma\}} \frac{\lambda_\alpha}{\lambda_\gamma} \int_0^T \int_{\Omega_d \times \Omega_d} \| b_{\alpha\gamma,t}(X_{*,t}(\omega),X_{*,t}(\omega')) \|^2 \, \nu_\alpha(d\omega) \nu_\gamma(d\omega') \, dt.
\end{align*}
In particular, Trace$(A_{\alpha \gamma} A_{\alpha \gamma'}^*) = 0$ if $\gamma \ne \gamma'$.
Hence we have
\begin{align*}
	\textnormal{Trace}(AA^*) & = \textnormal{Trace} \Big( (\sum_{\alpha=1}^K \sum_{\gamma=1}^K A_{\alpha\gamma})(\sum_{\alpha=1}^K \sum_{\gamma=1}^K A_{\alpha\gamma})^* \Big) = \sum_{\alpha=1}^K \sum_{\gamma=1}^K \textnormal{Trace}(A_{\alpha\gamma} A_{\alpha\gamma}^*) \\
	& = \sum_{\alpha,\gamma=1}^K \frac{\lambda_\alpha}{\lambda_\gamma} \int_0^T \int_{\Omega_d \times \Omega_d} \| b_{\alpha\gamma,t}(X_{*,t}(\omega),X_{*,t}(\omega')) \|^2 \, \nu_\alpha(d\omega) \nu_\gamma(d\omega') \, dt,
\end{align*}
which proves part $(a)$ of Lemma~\ref{lemchap2multitype:trace}. 
Noting that
\begin{equation*}
	\textnormal{Trace}(A^n) = \int_{\Omega_d^K \times \Omega_d^K \times \dotsb \times \Omega_d^K} \hhat(\omega_1,\omega_2) \hhat(\omega_2,\omega_3) \dotsm \hhat(\omega_n,\omega_1) \nuhat(d\omega_1) \nuhat(d\omega_2) \dotsm \nuhat(d\omega_n),
\end{equation*}
part $(b)$ follows from the proof of Lemma~$2.7$ of~\cite{ShigaTanaka1985}. 
Part $(c)$ is now immediate from Lemma~$1.3$ of~\cite{ShigaTanaka1985} (cf.\ Lemma~\ref{lemappendixchap2multitype:Shiga and Tanaka} in Appendix~\ref{Appendixchap2multitype:Restating}). \qed


\subsection{Proof of Theorem~\ref{thmchap2multitype:CLT}} \label{secchap2multitype:proof of CLT}

Recall the canonical processes $X^i, W^i, V^i$ and probability measure $\Pmb^N$ introduced in Section~\ref{secchap2multitype:canonical processes}.
For $t \in [0,T]$, define
\begin{equation*}
	J^N(t) = J^{N,1}(t) - \frac{1}{2} J^{N,2}(t),
\end{equation*}
where
\begin{equation} \label{eqchap2multitype:JN1}
	J^{N,1}(t) = \sum_{\alpha=1}^{K} \sum_{i \in \Nalpha} \int_0^t \sum_{\gamma=1}^{K} \frac{1}{N_\gamma} \sum_{j \in \Ngamma} b_{\alpha\gamma,s}(\Xis,\Xjs) \cdot d \Wis
\end{equation}
and
\begin{equation} \label{eqchap2multitype:JN2}
	J^{N,2}(t) = \sum_{\alpha=1}^{K} \sum_{i \in \Nalpha} \int_0^t \Big\| \sum_{\gamma=1}^{K} \frac{1}{N_\gamma} \sum_{j \in \Ngamma} b_{\alpha\gamma,s}(\Xis,\Xjs) \Big\|^2 ds.
\end{equation}
Let $\bar{\Fmc}^N_t = \sigma\{V^i(s), 0 \le s \le t, i \in \Nbd \}$.
Note that $\{\exp\big(J^N(t)\big)\}$ is a $\{ \bar{\Fmc}^N_t \}$-martingale under $\Pmb^N$.
Define a new probability measure $\Qmb^N$ on $\Omega_d^N$ by 
\begin{equation*}
	\frac{d\Qmb^N}{d\Pmb^N} = \exp\big(J^N(T)\big).
\end{equation*}

By Girsanov's theorem, $(X^1, \dotsc, X^N)$ has the same probability distribution under $\Qmb^N$ as $(Z^{1,N}, \dotsc, Z^{N,N})$ under $\Pbd$.
For $\phi \in \Amc_\alpha$, $\alpha \in \Kbd$, let
\begin{equation} \label{eqchap2multitype:xi N til}
	\xitil^N_\alpha (\phi) = \frac{1}{\sqrt{N_\alpha}} \sum_{i \in \Nalpha} \phi (X^i).
\end{equation}
Thus in order to prove the theorem it suffices to show that for any $\phi^{(\alpha)} \in \Amc_{\alpha}, \alpha \in \Kbd$,
\begin{equation*}
	\lim_{N \to \infty} \Emb_{\Qmb^N} \exp \Big( i \sum_{\alpha \in \Kbd} \xitil^N_{\alpha}(\phi^{(\alpha)}) \Big) = \exp \Big( -\frac{1}{2} \Big\| (I-A)^{-1} \sum_{\alpha \in \Kbd} \phihatbd_{\alpha}^{(\alpha)} \Big\|^2_{L^2(\Omega_d^K,\nuhat)} \Big),
\end{equation*}
where as in Section~\ref{secchap2multitype:CLT}, $\phibd^{(\alpha)} = \phi^{(\alpha)}(X_*)$, and $\phihatbd^{(\alpha)}_\alpha$ is defined by~\eqref{eqchap2multitype:phihat alpha}, replacing $\phibd$ with $\phibd^{(\alpha)}$.
This is equivalent to showing
\begin{align} \label{eqchap2multitype:convergence of characteristic function}
	\begin{split}
		& \lim_{N \to \infty} \Emb_{\Pmb^N} \exp \Big( i \sum_{\alpha \in \Kbd} \xitil^N_{\alpha}(\phi^{(\alpha)}) + J^{N,1}(T) - \frac{1}{2} J^{N,2}(T) \Big) \\
		& \hspace{50 mm} = \exp \Big( -\frac{1}{2} \Big\| (I-A)^{-1} \sum_{\alpha \in \Kbd} \phihatbd_{\alpha}^{(\alpha)} \Big\|^2_{L^2(\Omega_d^K,\nuhat)} \Big).
	\end{split}
\end{align}
In order to prove~\eqref{eqchap2multitype:convergence of characteristic function}, we will need to study the asymptotics of $J^{N,1}$ and $J^{N,2}$ as $N \to \infty$. 
The proof of~\eqref{eqchap2multitype:convergence of characteristic function} is completed in Section~\ref{secchap2multitype:completing the proof of CLT}.
We begin by studying a generalization of Theorem~\ref{thmchap2multitype:Dynkin} to the case of $K$ populations.


\subsubsection{Asymptotics of statistics of multi-type populations of independent particles} \label{secchap2multitype:extending Dynkin-Mandelbaum}

Throughout this subsection, let $\{\Hbd^i = (H_1^i, H_2^i, \dots, H_K^i)\}_{i \ge 1}$ be a sequence of i.i.d. $\Omega_d^K$-valued random variables with law $\nuhat = \nu_1 \otimes \dotsb \otimes \nu_K$.
We introduce the following notion of {\em lifted functions} and {\em lifted symmetric functions}.
Given $\alpha, \gamma \in \Kbd$ and $\psibd_{\alpha\gamma} \in L^2(\Omega_d \times \Omega_d, \nu_\alpha \otimes \nu_\gamma)$, define $\psihatbd_{\alpha\gamma}$ and $\psihatbd^{sym}_{\alpha\gamma}$ in $L^2(\Omega_d^K \times \Omega_d^K, \nuhat \otimes \nuhat)$ as follows: 
For $\omega = (\omega_1,\dotsc,\omega_K)$ and $\omega' = (\omega'_1,\dotsc,\omega'_K)$ in $\Omega_d^K$,
\begin{equation}
	\psihatbd_{\alpha\gamma}(\omega,\omega') = \psibd_{\alpha\gamma}(\omega_\alpha,\omega'_\gamma), \quad \psihatbd^{sym}_{\alpha\gamma}(\omega,\omega') = \frac{1}{2} \Big( \psihatbd_{\alpha\gamma}(\omega,\omega') + \psihatbd_{\alpha\gamma}(\omega',\omega) \Big). \label{eqchap2multitype:psi hat sym}
\end{equation}
Recall $\phihatbd^{sym} \in L^2((\Omega_d^K)^K, \nuhat^{\otimes K})$ defined in~\eqref{eqchap2multitype:phi hat sym} for $\phibd \in L^2(\Omega_d^K, \nuhat)$.

Recall the definition of $L^2_c(\Omega_d, \nu_\alpha)$ and $L^2_c(\Omega_d^K, \nuhat)$ in Section~\ref{secchap2multitype:CLT} and~\ref{secchap2multitype:generalization of CLT}, respectively.
Also, for $\alpha, \gamma \in \Kbd$, denote by $L^2_c(\Omega_d \times \Omega_d, \nu_\alpha \otimes \nu_\gamma)$ the subspace of $\psibd \in L^2(\Omega_d \times \Omega_d, \nu_\alpha \otimes \nu_\gamma)$ such that
\begin{equation*}
	\int_{\Omega_d} \psibd(\omega', \omega) \, \nu_\alpha(d\omega') = 0, \nu_\gamma \text{ a.e.\ } \omega \text{ and } \int_{\Omega_d} \psibd(\omega, \omega') \, \nu_\gamma(d\omega') = 0, \nu_\alpha \text{ a.e.\ } \omega.
\end{equation*}
The following lemma gives asymptotic distribution of certain statistics of the $K$-type population introduced above.

\begin{lemma} \label{lemchap2multitype:balanced populations}
	Let $\{\phibd_\alpha, \psibd_{\alpha\gamma}, \phibd: \alpha, \gamma \in \Kbd \}$ be a collection of functions such that, $\phibd_\alpha \in  L^2_c(\Omega_d, \nu_\alpha)$ for each $\alpha \in \Kbd$, $\psibd_{\alpha\gamma} \in L^2_c(\Omega_d \times \Omega_d, \nu_\alpha \otimes \nu_\gamma)$ for each pair of $\alpha, \gamma \in \Kbd$, and $\phibd \in L^2_c(\Omega_d^K, \nuhat)$.
	For each $N \in \Nmb$, let
	\begin{equation*}
		\zeta^N = (\zeta^N_1, \zeta^N_2, \zeta^N_3, \zeta^N_4), \quad \eta = (\eta_1, \eta_2, \eta_3, \eta_4),
	\end{equation*}
	where
	\begin{align*}
		\zeta^N_1 = \big( \frac{1}{\sqrt{N}} \sum_{i=1}^N \phibd_\alpha(H^i_\alpha) \big)_{\alpha = 1}^K, & & \zeta^N_2 & = \big( \frac{1}{N} \sum_{1 \le i \ne j \le N} \psibd_{\alpha\alpha}(H^i_\alpha,H^j_\alpha) \big)_{\alpha = 1}^K, \\
		\zeta^N_3 =  \big( \frac{1}{N} \sum_{i,j=1}^N \psibd_{\alpha\gamma}(H^i_\alpha,H^j_\gamma) \big)_{1 \le \alpha \ne \gamma \le K}, & & \zeta^N_4 & =  \frac{1}{N^{K/2}} \sum_{i_1,\dotsc,i_K=1}^N \phibd(H_1^{i_1},\dots,H_K^{i_K}),
	\end{align*}
	and
	\begin{equation*}
		\eta_1 = \big( I_1(\phihatbd_\alpha) \big)_{\alpha = 1}^K, \quad
		\eta_2 = \big( I_2(\psihatbd_{\alpha\alpha}^{sym}) \big)_{\alpha = 1}^K, \quad
		\eta_3 = \big( I_2(\psihatbd_{\alpha\gamma}^{sym}) \big)_{1 \le \alpha \ne \gamma \le K}, \quad
		\eta_4 = I_K(\phihatbd^{sym}).
	\end{equation*}
	Here $\{ I_k \}_{k \ge 1}$ are as defined in Section~\ref{secchap2multitype:generalization of CLT}.
	Then $\zeta^N \Rightarrow \eta$ as a sequence of $\Rmb^q$ valued random variables, as $N \to \infty$, where $q = K + K + K(K-1) + 1$.
\end{lemma}

\begin{proof}
	Note that for $\alpha \in \Kbd$, we have
	\begin{gather*}
		\frac{1}{\sqrt{N}} \sum_{i=1}^N \phibd_\alpha(H^i_\alpha) = \frac{1}{\sqrt{N}} \sum_{i=1}^N \phihatbd_\alpha(\Hbd^i), \\
		\frac{1}{N} \sum_{1 \le i \ne j \le N} \psibd_{\alpha\alpha}(H^i_\alpha,H^j_\alpha) = \frac{2}{N} \sum_{1 \le i < j \le N} \psihatbd^{sym}_{\alpha\alpha} (\Hbd^i,\Hbd^j).
	\end{gather*}
	Also, for $\alpha, \gamma \in \Kbd$ such that $\alpha \ne \gamma$, we have
	\begin{align*}
		\frac{1}{N} \sum_{i,j=1}^N \psibd_{\alpha\gamma}(H^i_\alpha,H^j_\gamma) 
		& = \frac{2}{N} \sum_{1 \le i < j \le N} \psihatbd^{sym}_{\alpha\gamma} (\Hbd^i,\Hbd^j) + \frac{1}{N} \sum_{i=1}^N \psibd_{\alpha\gamma}(H^i_\alpha,H^i_\gamma),
	\end{align*}
	Let $\Smc = \{(i_1,\dotsc,i_K) \in \Nbd^K : i_1,\dotsc,i_K \text{ are distinct}\}$.
	Then we have
	\begin{align*}
		& \frac{1}{N^{K/2}} \sum_{i_1,\dotsc,i_K=1}^N \phibd(H_1^{i_1},\dots,H_K^{i_K}) \\
		& \quad = \frac{1}{N^{K/2}} \sum_{(i_1,\dotsc,i_K) \in \Smc} \phibd(H_1^{i_1},\dots,H_K^{i_K}) + \frac{1}{N^{K/2}} \sum_{(i_1,\dotsc,i_K) \notin \Smc} \phibd(H^{i_1}_1,\dots,H^{i_K}_K)\\
		& \quad = \frac{K!}{N^{K/2}} \sum_{1 \le i_1 < \dotsb < i_K \le N} \phihatbd^{sym}(\Hbd^{i_1},\dots,\Hbd^{i_K}) + \frac{1}{N^{K/2}} \sum_{(i_1,\dotsc,i_K) \notin \Smc} \phibd(H^{i_1}_1,\dots,H^{i_K}_K).
	\end{align*}
	Combining above results gives us
	\begin{equation} \label{eqchap2multitype:zeta N}
		\zeta^N = \zetatil^N + \Rmc^{N,1} = (\zetatil^N_1, \zetatil^N_2, \zetatil^N_3, \zetatil^N_4) + (\mathbf{0}_{K \times 1}, \mathbf{0}_{K \times 1}, \Rmc^{N,1}_3, \Rmc^{N,1}_4),
	\end{equation}
	where
	\begin{align*}
		\zetatil^N_1 = \big( \frac{1}{\sqrt{N}} \sum_{i=1}^N \phihatbd_\alpha(\Hbd^i) \big)_{\alpha = 1}^K, & &
		\zetatil^N_2 & =  \big( \frac{2}{N} \sum_{1 \le i < j \le N} \psihatbd^{sym}_{\alpha\alpha} (\Hbd^i,\Hbd^j) \big)_{\alpha = 1}^K, \\
		\zetatil^N_3 = \big( \frac{2}{N} \sum_{1 \le i < j \le N} \psihatbd^{sym}_{\alpha\gamma} (\Hbd^i,\Hbd^j) \big)_{1 \le \alpha \ne \gamma \le K}, & &
		\zetatil^N_4 & = \frac{K!}{N^{K/2}} \sum_{1 \le i_1 < \dotsb < i_K \le N} \phihatbd^{sym}(\Hbd^{i_1},\dots,\Hbd^{i_K}),
	\end{align*}
	and
	\begin{gather*}
		\Rmc^{N,1}_3 = \big( \frac{1}{N} \sum_{i=1}^N \psibd_{\alpha\gamma}(H^i_\alpha,H^i_\gamma) \big)_{1 \le \alpha \ne \gamma \le K}, \quad
		\Rmc^{N,1}_4 = \frac{1}{N^{K/2}} \sum_{(i_1,\dotsc,i_K) \notin \Smc} \phibd(H^{i_1}_1,\dots,H^{i_K}_K).
	\end{gather*}
	From Theorem~\ref{thmchap2multitype:Dynkin} it follows that as $N \to \infty$,
	\begin{equation} \label{eqchap2multitype:zetatil converge to eta}
		\zetatil^N \Rightarrow \eta.
	\end{equation}
	By law of large numbers, we have $\Rmc^{N,1}_3 \to \mathbf{0}_{K(K-1) \times 1}$ in probability as $N \to \infty$.
	Note that as $N \to \infty$,
	\begin{equation*}
		\Emb (\Rmc^{N,1}_4)^2 = \frac{1}{N^K} \sum_{(i_1,\dotsc,i_K) \notin \Smc} \Emb [\phibd(H^{i_1}_1,\dots,H^{i_K}_K)]^2 \to 0,
	\end{equation*}
	since $|\Nbd^K \backslash \Smc| = N^K - N (N-1) \dotsm (N-K+1) = o(N^K)$.
	Hence
	\begin{equation} \label{eqchap2multitype:R N1 vanishes}
		\Rmc^{N,1} \to \mathbf{0}_{q \times 1}
	\end{equation} 
	in probability as $N \to \infty$. 
	The desired result follows by combining~\eqref{eqchap2multitype:zeta N},~\eqref{eqchap2multitype:zetatil converge to eta} and~\eqref{eqchap2multitype:R N1 vanishes}.
\end{proof}

For studying the asymptotic behavior of $J^{N,1}$ and $J^{N,2}$, we will need an extension of Lemma~\ref{lemchap2multitype:balanced populations} to a setting where the numbers of particles of different types may differ.
Before discussing this extension, we present an elementary lemma, the proof of which is given in Appendix~\ref{Appendixchap2multitype:Proof of Lemma of convergence via approximation}.

\begin{lemma} \label{lemchap2multitype:convergence via approximation}
	For $m,n = 1,2,\dotsc,$ let $\xi_{mn}, \xi_n$ be $\Rmb^d$-valued random variables defined on some probability space $(\Omega_0, \Fmc_0, \Pmb_0)$, and $\eta_m, \eta$ be $\Rmb^d$-valued random variables defined on some other probability space $(\Omega_0', \Fmc_0', \Pmb_0')$. 
	If for each $m \ge 1$, $\xi_{mn} \Rightarrow \eta_m$ as $n \to \infty$ and the following condition holds:
	\begin{equation*}
		\lim_{m \to \infty} \sup_{n \ge 1} \Emb_{\Pmb_0} \big( \|\xi_{mn} - \xi_n\| \wedge 1 \big) = 0, \quad \lim_{m \to \infty} \Emb_{\Pmb_0'} \big( \|\eta_m - \eta\| \wedge 1 \big) = 0,
	\end{equation*}
	then $\xi_n \Rightarrow \eta$ as $n \to \infty$.
\end{lemma}

We now study the following extension of Lemma~\ref{lemchap2multitype:balanced populations}.
\begin{lemma} \label{lemchap2multitype:unbalanced populations}
	For $\alpha \in \Kbd$, let $N_\alpha$ be a function of $N$ such that $\lim_{N \to \infty} N_\alpha = \infty$.
	Let $\{\phibd_\alpha, \psibd_{\alpha\gamma}, \phibd, \alpha, \gamma \in \Kbd \}$, $\eta$ and $q$ be as in Lemma~\ref{lemchap2multitype:balanced populations}.
	For each $N \in \Nmb$, let
	\begin{equation*}
		\xi^N = (\xi^N_1, \xi^N_2, \xi^N_3, \xi^N_4),
	\end{equation*}
	where
	\begin{gather*}
		\xi^N_1 = \big( \frac{1}{\sqrt{N_\alpha}} \sum_{i=1}^{N_\alpha} \phibd_\alpha(H^i_\alpha) \big)_{\alpha = 1}^K, \quad
		\xi^N_2 = \big( \frac{1}{N_\alpha} \sum_{1 \le i \ne j \le N_\alpha} \psibd_{\alpha\alpha}(H^i_\alpha,H^j_\alpha) \big)_{\alpha = 1}^K, \\
		\xi^N_3 = \big( \frac{1}{\sqrt{N_\alpha N_\gamma}} \sum_{i=1}^{N_\alpha} \sum_{j=1}^{N_\gamma} \psibd_{\alpha\gamma}(H^i_\alpha,H^j_\gamma) \big)_{1 \le \alpha \ne \gamma \le K}, \\
		\xi^N_4 = \frac{1}{\sqrt{N_1 \dotsm N_K}} \sum_{i_1=1}^{N_1} \dotsb \sum_{i_K=1}^{N_K} \phibd(H_1^{i_1},\dots,H_K^{i_K}).
	\end{gather*}
	Then $\xi^N \Rightarrow \eta$ as a sequence of $\Rmb^q$ valued random variables, as $N \to \infty$.
\end{lemma}

\begin{proof}
	For each $\beta = 1,\dots,K$, let $\{\one_{\Omega_d}, e_\beta^1, e_\beta^2, \dotsc\}$ be a complete orthonormal system (CONS) in $L^2(\Omega_d,\nu_\beta)$.
	Note that $\langle \phibd_\alpha, \one_{\Omega_d} \rangle_{L^2(\Omega_d,\nu_\alpha)} = 0$ for all $\alpha \in \Kbd$, and analogous orthogonality properties with the function $\one_{\Omega_d}$ hold for $\psibd_{\alpha \gamma}$ and $\phibd$, for all $\alpha, \gamma \in \Kbd$.
	(For example, $\langle \psibd_{\alpha\gamma}, \one_{\Omega_d} \otimes e^j_\gamma \rangle_{L^2(\Omega_d \times \Omega_d,\nu_\alpha \otimes \nu_\gamma)} = 0$ for all $\alpha, \gamma \in \Kbd$ and $j \ge 1$).
	Because of this, 
	$\psibd_{\alpha\gamma}^M \to \psibd_{\alpha\gamma}$ in $L^2(\Omega_d \times \Omega_d, \nu_\alpha \otimes \nu_\gamma)$ and $\phibd^M \to \phibd$ in $L^2(\Omega_d^K, \nuhat)$, as $M \to \infty$, where
	\begin{gather*}
		\psibd_{\alpha\gamma}^M(\omega,\omega') = \sum_{m_\alpha=1}^M \sum_{m_\gamma=1}^M c_{\alpha\gamma}^{m_\alpha m_\gamma} e_\alpha^{m_\alpha}(\omega) e_\gamma^{m_\gamma}(\omega'), \quad (\omega, \omega') \in \Omega_d \times \Omega_d, \quad \alpha, \gamma \in \Kbd, \\
		\phibd^M(\omega_1,\dotsc,\omega_K) = \sum_{m_1=1}^M \dotsb \sum_{m_K=1}^M c^{m_1 \dotsm m_K} e_1^{m_1}(\omega_1) \dotsm e_K^{m_K}(\omega_K), \quad (\omega_1, \dotsc, \omega_K) \in \Omega_d^K,
	\end{gather*}
	and $c_{\alpha\gamma}^{m_\alpha m_\gamma} = \langle \psibd_{\alpha\gamma}, e_\alpha^{m_\alpha} \otimes e_\gamma^{m_\gamma} \rangle_{L^2(\Omega_d \times \Omega_d, \nu_\alpha \otimes \nu_\gamma)}$, $c^{m_1 \dotsm m_K} = \langle \phibd, e_1^{m_1} \otimes \dotsb \otimes e_K^{m_K} \rangle_{L^2(\Omega_d^K, \nuhat)}$.
	It follows that as $M \to \infty$, $\psihatbd_{\alpha\gamma}^{M,sym} \to \psihatbd_{\alpha\gamma}^{sym}$ in $L^2(\Omega_d^K \times \Omega_d^K, \nuhat \otimes \nuhat)$ and $\phihatbd^M \to \phihatbd$ in $L^2((\Omega_d^K)^K, \nuhat^{\otimes K})$.
	For $M, N \in \Nmb$, let 
	\begin{equation*}
		\xi^{MN} = (\xi^N_1, \xi^N_2, \xi^{MN}_3, \xi^{MN}_4), \quad \eta^{M} = (\eta_1, \eta_2, \eta^{M}_3, \eta^{M}_4),
	\end{equation*}
	where $\xi^{MN}_3$, $\xi^{MN}_4$ are defined as $\xi^{N}_3$, $\xi^{N}_4$ but with $\psibd^M_{\alpha\gamma}$ and
	$\phibd^M$ instead of $\psibd_{\alpha\gamma}$ and $\phibd$,
	and
	\begin{equation*}
		\eta^{M}_3 = \big( I_2(\psihatbd_{\alpha\gamma}^{M,sym}) \big)_{1 \le \alpha \ne \gamma \le K}, \quad
		\eta^{M}_4 = I_K(\phihatbd^{M,sym}).
	\end{equation*}
	From Lemma~\ref{lemchap2multitype:convergence via approximation}, in order to finish the proof, it suffices to check the following three properties hold:
	\begin{equation}
		\lim_{M \to \infty} \sup_{N \ge 1} \Emb \|\xi^{MN} - \xi^N\|^2 = 0,\; \lim_{M \to \infty} \Emb \|\eta^{M} - \eta\|^2 = 0,\;
			\xi^{MN} \Rightarrow \eta^{M} \text{ as } N \to \infty. \label{eqchap2multitype:xi converge in L2} \\
	\end{equation}
	Note that the first two coordinates of $\xi^{MN}$ and $\eta^M$ are the same as those of $\xi^N$ and $\eta$, respectively.
	So for the first two statements in~\eqref{eqchap2multitype:xi converge in L2}, we only need to check for the third and last coordinates. Consider the first statement. 
	\begin{align*}
		& \lim_{M \to \infty} \sup_{N \ge 1} \Emb (\xi^{MN}_4-\xi^N_4)^2 \\
		& \quad = \lim_{M \to \infty} \sup_{N \ge 1} \frac{1}{N_1 \dotsm N_K} \sum_{i_1=1}^{N_1} \dotsb \sum_{i_K=1}^{N_K} \Emb \big( \phibd^M(H_1^{i_1},\dots,H_K^{i_K}) - \phibd(H_1^{i_1},\dots,H_K^{i_K}) \big)^2 \\
		& \quad = \lim_{M \to \infty} \Emb \big( \phibd^M(H^1_1, \dotsc, H^1_K) - \phibd(H^1_1, 	\dotsc, H^1_K) \big)^2 = 0.
	\end{align*}
	Similarly, for each pair of $\alpha, \gamma \in \Kbd$ such that $\alpha \ne \gamma$, 
	$$\lim_{M \to \infty} \sup_{N \ge 1} \Emb \Big( \frac{1}{\sqrt{N_\alpha N_\gamma}} \sum_{i=1}^{N_\alpha} \sum_{j=1}^{N_\gamma} \psibd^M_{\alpha\gamma}(H^i_\alpha,H^j_\gamma) - \frac{1}{\sqrt{N_\alpha N_\gamma}} \sum_{i=1}^{N_\alpha} \sum_{j=1}^{N_\gamma} \psibd_{\alpha\gamma}(H^i_\alpha,H^j_\gamma) \Big)^2=0.$$
	This proves the first statement in~\eqref{eqchap2multitype:xi converge in L2}.
	Recalling the isometry property of MWI $\eqref{eqchap2multitype:moments of MWI}$, we have for each pair of $\alpha, \gamma \in \Kbd$ such that $\alpha \ne \gamma$,
	\begin{gather*}
		\lim_{M \to \infty} \Emb \Big( I_2(\psihatbd_{\alpha\gamma}^{M,sym}) - I_2(\psihatbd_{\alpha\gamma}^{sym}) \Big)^2 = \lim_{M \to \infty} 2 \Big\| \psihatbd_{\alpha\gamma}^{M,sym} - \psihatbd_{\alpha\gamma}^{sym} \Big\|^2_{L^2(\Omega_d^K \times \Omega_d^K, \nuhat \otimes \nuhat)} = 0.
	\end{gather*}
	Thus the second statement in~\eqref{eqchap2multitype:xi converge in L2} holds with $\eta^M$ and $\eta$ replaced by $\eta^M_3$ and $\eta_3$, respectively.
	Similarly, for $\eta^{M}_4$, we have
	\begin{equation*}
		\lim_{M \to \infty} \Emb (\eta^{M}_4-\eta_4)^2 = \lim_{M \to \infty} K! \, \Big\| \phihatbd^{M,sym} - \phihatbd^{sym} \Big\|^2_{L^2((\Omega_d^K)^K, \nuhat^{\otimes K})} = 0.
	\end{equation*}
	Combining the above observations, we have the second statement in~\eqref{eqchap2multitype:xi converge in L2}.

	Finally we check the third statement in~\eqref{eqchap2multitype:xi converge in L2}.
	Let $\{\Itil_1^\alpha(h), h \in L^2(\nu_\alpha)\}_{\alpha=1}^K$ be $K$ mutually independent Gaussian fields such that
	\begin{equation*}
		\Emb \big( \Itil_1^\alpha(h) \big) = 0, \: \Emb \big (\Itil_1^\alpha(h)\Itil_1^\alpha(g) \big) = \langle h,g \rangle_{L^2(\nu_\alpha)} \text{ for all } h,g \in L^2(\nu_\alpha), \alpha \in \Kbd.
	\end{equation*}
	Let
	$\etatil^{M} = (\etatil_1, \etatil_2, \etatil^{M}_3, \etatil^{M}_4),
	$
	where
	\begin{gather*}
		\etatil_1 = \big( \Itil_1^\alpha(\phihatbd_\alpha) \big)_{\alpha = 1}^K, \quad
		\etatil_2 = \big( \Itil_2^\alpha(\psihatbd_{\alpha\alpha}^{sym}) \big)_{\alpha = 1}^K, \\
		\etatil^{M}_3 = \big( \sum_{m_\alpha=1}^M \sum_{m_\gamma=1}^M c_{\alpha\gamma}^{m_\alpha m_\gamma} \Itil_1^\alpha(e_\alpha^{m_\alpha}) \Itil_1^\gamma(e_\gamma^{m_\gamma}) \big)_{1 \le \alpha \ne \gamma \le K}, \\
		\etatil^{M}_4 = \sum_{m_1=1}^M \dotsb \sum_{m_K=1}^M c^{m_1 \dotsm m_K} \Itil_1^1(e_1^{m_1}) \dotsm \Itil_1^K(e_K^{m_K}),
	\end{gather*}
	and $\Itil_2^\alpha$ is as defined below~\eqref{eqchap2multitype:MWI formula} by replacing $I_1$ there with $\Itil_1^\alpha$.
	Theorem~\ref{thmchap2multitype:Dynkin} and mutual independence of $\{H_\alpha^i : i = 1, \dotsc, N_\alpha, \alpha = 1, \dotsc, K\}$ imply that for each $M \ge 1$, as $N \to \infty$,
$\xi^{MN} \Rightarrow \etatil^{M}.$

	In order to verify the third condition in~\eqref{eqchap2multitype:xi converge in L2},  it now suffices to show $\etatil^{M}$ and $\eta^M$ have the same probability distribution.
	However, this follows easily by considering the asymptotic behavior of
	\begin{equation*}
		\zeta^{MN} = (\zeta^N_1, \zeta^N_2, \zeta^{MN}_3, \zeta^{MN}_4),
	\end{equation*}
	where $\zeta^N_1$ and $\zeta^N_2$ are as in Lemma~\ref{lemchap2multitype:balanced populations}, and $\zeta^{MN}_3$ [resp.\ $\zeta^{MN}_4$] are also as in Lemma~\ref{lemchap2multitype:balanced populations} but with $\psibd_{\alpha\gamma}$ [resp.\ $\phibd$] replaced with $\psibd^M_{\alpha\gamma}$ [resp.\ $\phibd^M$].
	Once more by Theorem~\ref{thmchap2multitype:Dynkin} and mutual independence of $\{H_\alpha^i : i = 1, \dotsc, N_\alpha, \alpha = 1, \dotsc, K\}$, we get as $N \to \infty$, 
	$\zeta^{MN} \Rightarrow \etatil^{M}$.
	On the other hand, Lemma~\ref{lemchap2multitype:balanced populations} implies that as $N \to \infty$,
	$\zeta^{MN} \Rightarrow \eta^{M}$.
	Combining these observations, we see that $\etatil^M$ and $\eta^M$ have the same probability distribution, which finishes the proof.
\end{proof}


\subsubsection{Asymptotics of $J^N$} \label{secchap2multitype:asymptotics of JN}

Recall the definition of $J^{N,1}$, $V^i$, $h_{\alpha\gamma}$, $\hhat_{\alpha\gamma}$ and $\hhat$ in~\eqref{eqchap2multitype:JN1},~\eqref{eqchap2multitype:Vi},~\eqref{eqchap2multitype:h alpha gamma} and~\eqref{eqchap2multitype:hhat} respectively.  All convergence statements in this section are under $\Pmb^N$.
It follows by law of large numbers that for $\alpha \in \Kbd$, as $N \to \infty$,
\begin{equation*}
	\frac{1}{N_\alpha} \sum_{i \in \Nalpha} h_{\alpha\alpha}(V^i,V^i) \Rightarrow \int_{\Omega_d} h_{\alpha\alpha}(\omega,\omega) \, \nu_\alpha(d\omega) = 0.
\end{equation*}
By Lemma~\ref{lemchap2multitype:unbalanced populations} and the above result, we have the following convergence  as $N \to \infty$:
\begin{align*}
	J^{N,1}(T) & = \sum_{\alpha=1}^{K} \sum_{i \in \Nalpha} \int_0^T \sum_{\gamma=1}^{K} \frac{1}{N_\gamma} \sum_{j \in \Ngamma} b_{\alpha\gamma,t}(\Xit,\Xjt) \cdot d \Wit \\
	& = \sum_{\alpha=1}^{K} \sum_{\gamma=1}^{K} \sqrt{\frac{\lambda_\gamma N_\alpha}{\lambda_\alpha N_\gamma}} \frac{1}{\sqrt{N_\alpha N_\gamma}} \sum_{i \in \Nalpha} \sum_{j \in \Ngamma} h_{\alpha\gamma}(V^i,V^j) \Rightarrow I_2(\hhat^{sym}),
\end{align*}
where $\hhat^{sym}$ is defined as in~\eqref{eqchap2multitype:psi hat sym} with $\psihatbd$ replaced by $\hhat$.

Recall the definition of $J^{N,2}$ in~\eqref{eqchap2multitype:JN2}.
We split $J^{N,2}$ as follows:
\begin{align*}
	J^{N,2}(T) & = \sum_{\alpha=1}^{K} \sum_{i \in \Nalpha} \sum_{\beta=1}^{K} \sum_{j \in \Nbeta} \sum_{\gamma=1}^{K} \sum_{k \in \Ngamma} \frac{1}{N_\beta N_\gamma} \int_0^T b_{\alpha\beta,t}(\Xit,\Xjt) \cdot b_{\alpha\gamma,t}(\Xit,\Xkt) \, dt \\
	& = \sum_{n=1}^5 \sum_{\Smc_n} \frac{1}{N_\beta N_\gamma} \int_0^T b_{\alpha\beta,t}(\Xit,\Xjt) \cdot b_{\alpha\gamma,t}(\Xit,\Xkt) \, dt = \sum_{n=1}^5 \Tmc^N_n,
\end{align*}
where $\Smc_1$, $\Smc_2$, $\Smc_3$, $\Smc_4$ and $\Smc_5$ are collection of $(\alpha,\beta,\gamma) \in \Kbd^3$ and $(i,j,k) \in \Nalpha \times \Nbeta \times \Ngamma$ such that $\{i=j=k\}$, $\{i=j \ne k\}$, $\{i=k \ne j\}$, $\{j=k \ne i\}$ and $\{i,j,k \text{ distinct}\}$, respectively.
For the term $\Tmc^N_1$, from the boundedness of $b_{\alpha\beta,t}$ it follows that as $N \to \infty$,
\begin{equation*}
	\Tmc^N_1 = \sum_{\alpha=1}^{K} \frac{1}{N_\alpha^2} \sum_{i \in \Nalpha} \int_0^T \| b_{\alpha\alpha,t}(\Xit,\Xit) \|^2 \, dt \Rightarrow 0.
\end{equation*}
For $\Tmc^N_2$, let $\Smc_{\alpha\gamma} = \{(i,k) \in \Nalpha \times \Ngamma : i \ne k\}$ for $\alpha, \gamma \in \Kbd$.
Then
\begin{equation*}
	\Tmc^N_2 = \sum_{\alpha,\gamma=1}^K \frac{1}{N_\alpha N_\gamma} \sum_{(i,k) \in \Smc_{\alpha\gamma}} \int_0^T b_{\alpha\alpha,t}(\Xit,\Xit) \cdot b_{\alpha\gamma,t}(\Xit,\Xkt) \, dt = \Tmc^N_{2,1} + \Tmc^N_{2,2},
\end{equation*}
where 
\begin{align*}
	\Tmc^N_{2,1} & = \sum_{\alpha,\gamma=1}^K \frac{1}{N_\alpha N_\gamma} \sum_{(i,k) \in \Smc_{\alpha\gamma}} \int_0^T \Big( b_{\alpha\alpha,t}(\Xit,\Xit) \cdot b_{\alpha\gamma,t}(\Xit,\Xkt) \\
	& \quad - \int_{\Omega_d} b_{\alpha\alpha,t}(X_{*,t}(\omega),X_{*,t}(\omega)) \cdot b_{\alpha\gamma,t}(X_{*,t}(\omega),\Xkt) \, \nu_\alpha(d\omega) \Big) \, dt
\end{align*}
converges to $0$ in $L^2(\Omega_d^N, \Pmb^N)$ as $N \to \infty$, and
\begin{equation*}
	\Tmc^N_{2,2} = \sum_{\alpha,\gamma=1}^K \frac{1}{N_\alpha N_\gamma} \sum_{(i,k) \in \Smc_{\alpha\gamma}} \int_0^T \int_{\Omega_d} b_{\alpha\alpha,t}(X_{*,t}(\omega),X_{*,t}(\omega)) \cdot b_{\alpha\gamma,t}(X_{*,t}(\omega),\Xkt) \, \nu_\alpha(d\omega) \, dt \Rightarrow 0
\end{equation*} 
by the law of large numbers.
Hence $\Tmc^N_2 \Rightarrow 0$ as $N \to \infty$.
Similarly $\Tmc^N_3 \Rightarrow 0$.
Consider now $\Tmc^N_4$.
\begin{equation*}
	\Tmc^N_4 = \sum_{\alpha,\gamma=1}^K \frac{1}{N_\gamma^2} \sum_{(i,k) \in \Smc_{\alpha\gamma}} \int_0^T \| b_{\alpha\gamma,t}(\Xit,\Xkt) \|^2 \, dt = \Tmc^N_{4,1} + \Tmc^N_{4,2},
\end{equation*}
where
\begin{align*}
	\Tmc^N_{4,1} & = \sum_{\alpha,\gamma=1}^K \frac{1}{N_\gamma^2} \sum_{(i,k) \in \Smc_{\alpha\gamma}} \int_0^T \Big( \| b_{\alpha\gamma,t}(\Xit,\Xkt) \|^2
	- \int_{\Omega_d} \| b_{\alpha\gamma,t}(X_{*,t}(\omega),\Xkt) \|^2 \, \nu_\alpha(d\omega) \\
	& \quad - \int_{\Omega_d} \| b_{\alpha\gamma,t}(\Xit,X_{*,t}(\omega)) \|^2 \, \nu_\gamma(d\omega) + \int_{\Omega_d^2} \| b_{\alpha\gamma,t}(X_{*,t}(\omega),X_{*,t}(\omega')) \|^2 \, \nu_\alpha(d\omega) \nu_\gamma(d\omega') \Big) \, dt
\end{align*}
converges to $0$ in $L^2(\Omega_d^N, \Pmb^N)$ as $N \to \infty$, and
\begin{align*}
	\Tmc^N_{4,2} & = \sum_{\alpha,\gamma=1}^K \frac{1}{N_\gamma^2} \sum_{(i,k) \in \Smc_{\alpha\gamma}} \int_0^T \Big( \int_{\Omega_d} \| b_{\alpha\gamma,t}(X_{*,t}(\omega),\Xkt) \|^2 \, \nu_\alpha(d\omega) \\
	& \quad + \int_{\Omega_d} \| b_{\alpha\gamma,t}(\Xit,X_{*,t}(\omega)) \|^2 \, \nu_\gamma(d\omega) - \int_{\Omega_d^2} \| b_{\alpha\gamma,t}(X_{*,t}(\omega),X_{*,t}(\omega')) \|^2 \, \nu_\alpha(d\omega) \nu_\gamma(d\omega') \Big) \, dt \\
	& \Rightarrow \sum_{\alpha,\gamma=1}^K \frac{\lambda_\alpha}{\lambda_\gamma} \int_0^T \int_{\Omega_d \times \Omega_d} \| b_{\alpha\gamma,t}(X_{*,t}(\omega),X_{*,t}(\omega')) \|^2 \, \nu_\alpha(d\omega) \nu_\gamma(d\omega') \, dt = \textnormal{Trace}(AA^*)
\end{align*}
as $N \to \infty$, by the law of large numbers and Lemma~\ref{lemchap2multitype:trace}.
So $\Tmc^N_4 \Rightarrow \textnormal{Trace}(AA^*)$ as $N \to \infty$.
Finally consider $\Tmc^N_5$. Let $\Smc_{\alpha\beta\gamma} = \{(i,j,k) \in \Nalpha \times \Nbeta \times \Ngamma : i,j,k \text{ distinct}\}$ for $\alpha, \beta, \gamma \in \Kbd$.
Recalling the definition of $b_{\alpha\beta\gamma,t}$ in~\eqref{eqchap2multitype:b alpha beta gamma t}, we have
\begin{equation*}
	\Tmc^N_5 = \sum_{\alpha,\beta,\gamma=1}^{K} \frac{1}{N_\beta N_\gamma} \sum_{(i,j,k) \in \Smc_{\alpha\beta\gamma}} \int_0^T b_{\alpha\beta,t}(\Xit,\Xjt) \cdot b_{\alpha\gamma,t}(\Xit,\Xkt) \, dt = \sum_{\alpha,\beta,\gamma=1}^{K} (\Tmc^N_{\alpha\beta\gamma,1} + \Tmc^N_{\alpha\beta\gamma,2}),
\end{equation*}
where
\begin{equation*}
	\Tmc^N_{\alpha\beta\gamma,1} = \frac{1}{N_\beta N_\gamma} \sum_{(i,j,k) \in \Smc_{\alpha\beta\gamma}} \int_0^T \Big( b_{\alpha\beta,t}(\Xit,\Xjt) \cdot b_{\alpha\gamma,t}(\Xit,\Xkt) -  b_{\alpha\beta\gamma,t}(\Xjt,\Xkt) \Big) \, dt 
\end{equation*}
converges to $0$ in $L^2(\Omega_d^N, \Pmb^N)$ as $N \to \infty$,
and
\begin{equation*}
	\Tmc^N_{\alpha\beta\gamma,2} = \frac{1}{N_\beta N_\gamma} \sum_{(i,j,k) \in \Smc_{\alpha\beta\gamma}} \int_0^T b_{\alpha\beta\gamma,t}(\Xjt,\Xkt) \, dt.
\end{equation*}
Recall the definition of $l_{\alpha,\beta\gamma}$ and $\lhat_{\alpha,\beta\gamma}$ in~\eqref{eqchap2multitype:l alpha beta gamma} and~\eqref{eqchap2multitype:lhat alpha beta gamma} respectively.
By Lemma~\ref{lemchap2multitype:unbalanced populations}, we have the following convergence in law,
\begin{align*}
	\lim_{N \to \infty} \sum_{\alpha,\beta,\gamma=1}^{K} \Tmc^N_{\alpha\beta\gamma,2} & = \lim_{N \to \infty} \sum_{\alpha,\beta,\gamma=1}^{K} \frac{N_\alpha}{N_\beta N_\gamma} \sum_{(j,k) \in \Smc_{\beta\gamma}} \int_0^T b_{\alpha\beta\gamma,t}(\Xjt,\Xkt) \, dt \\
	& = \lim_{N \to \infty} \sum_{\alpha,\beta,\gamma=1}^{K} \frac{1}{\sqrt{N_\beta N_\gamma}} \sum_{(j,k) \in \Smc_{\beta\gamma}} l_{\alpha,\beta\gamma}(V^j,V^k) = \sum_{\alpha,\beta,\gamma=1}^{K} I_2(\lhat^{sym}_{\alpha,\beta\gamma}),
\end{align*}
where $\lhat^{sym}_{\alpha,\beta\gamma}$ is defined as in~\eqref{eqchap2multitype:psi hat sym} with $\psihatbd$ replaced by $\lhat_{\alpha,\beta\gamma}$.
So $\Tmc^N_5 \Rightarrow \sum_{\alpha,\beta,\gamma=1}^{K} I_2(\lhat^{sym}_{\alpha,\beta\gamma})$ as $N \to \infty$.
Define  $\lhat$ from $\Omega_d^K \times \Omega_d^K$ to $\Rmb$ ($\nuhat \otimes \nuhat$ a.s.) as
\begin{equation*}
	\lhat(\omega,\omega') = \sum_{\alpha=1}^{K} \sum_{\beta=1}^{K} \sum_{\gamma=1}^{K} \lhat_{\alpha,\beta\gamma}(\omega,\omega'), \quad (\omega, \omega') \in \Omega_d^K \times \Omega_d^K.
\end{equation*}
Noting that for all $(\omega, \omega') \in \Omega_d^K \times \Omega_d^K$, $\lhat_{\alpha,\beta\gamma}(\omega',\omega) = \lhat_{\alpha,\gamma\beta}(\omega,\omega')$, we have $\sum_{\alpha,\beta,\gamma=1}^{K} \lhat^{sym}_{\alpha,\beta\gamma} = \lhat$.
Combining the above observations we have as $N \to \infty$, 
\begin{equation*}
	J^{N,2}(T) \Rightarrow \textnormal{Trace}(AA^*) + \sum_{\alpha,\beta,\gamma=1}^{K} I_2(\lhat^{sym}_{\alpha,\beta\gamma}) = \textnormal{Trace}(AA^*) + I_2(\lhat),
\end{equation*}
In fact by Lemma~\ref{lemchap2multitype:unbalanced populations} we have as $N \to \infty$, under $\Pmb^N$,
\begin{equation} \label{eqchap2multitype:joint convergence of JN1 and JN2}
	(J^{N,1}(T), J^{N,2}(T)) \Rightarrow (I_2(\hhat^{sym}), \textnormal{Trace}(AA^*) + I_2(\lhat)).
\end{equation}

Recall the function $f \in L^2(\Omega_d^K \times \Omega_d^K, \nuhat \otimes \nuhat)$ defined in~\eqref{eqchap2multitype:f}.
It follows from~\eqref{eqchap2multitype:relation between lhat alpha beta gamma and hhat} that
\begin{equation*}
	\lhat(\omega,\omega') = \int_{\Omega_d^K} \hhat(\omega'',\omega) \hhat(\omega'',\omega') \, \nuhat(d\omega''),
\end{equation*}
which implies $f = 2\hhat^{sym} - \lhat$.
From~\eqref{eqchap2multitype:joint convergence of JN1 and JN2} we get as $N \to \infty$,
\begin{equation} \label{eqchap2multitype:convergence of JN}
	J^N(T) \Rightarrow I_2(\hhat^{sym}) - \frac{1}{2} \big( \textnormal{Trace}(AA^*) + I_2(\lhat) \big) = \frac{1}{2} \big( I_2(f) - \textnormal{Trace}(AA^*) \big) = J,
\end{equation}
where $J$ was introduced in~\eqref{eqchap2multitype:J}.
In fact from Lemma~\ref{lemchap2multitype:unbalanced populations} it follows that with $\xitil^N_\alpha$ as in~\eqref{eqchap2multitype:xi N til}, and $\phi^{(\alpha)}$ introduced below~\eqref{eqchap2multitype:xi N til}, as $N \to \infty$,
\begin{equation} \label{eqchap2multitype:convergence of xi tilde and JN}
	(\sum_{\alpha \in \Kbd} \xitil^N_{\alpha}(\phi^{(\alpha)}), J^N(T)) \Rightarrow (\sum_{\alpha \in \Kbd} I_1(\phihatbd^{(\alpha)}_{\alpha}), J).
\end{equation}
We will now use the above convergence result to complete the proof of Theorem~\ref{thmchap2multitype:CLT}.


\subsubsection{Completing the proof of Theorem~\ref{thmchap2multitype:CLT}} \label{secchap2multitype:completing the proof of CLT}

It follows from Lemma~$1.2$ of~\cite{ShigaTanaka1985} (cf.\ Lemma~\ref{lemappendixchap2multitype:Shiga and Tanaka} in Appendix~\ref{Appendixchap2multitype:Restating}) and Lemma~\ref{lemchap2multitype:trace} that $\Emb_{\tilde{\Pmb}} \big( \exp (J) \big) = 1$, where $\tilde{\Pmb}$ is as introduced above~\eqref{eqchap2multitype:f}.
Along with~\eqref{eqchap2multitype:convergence of JN} and the fact that $\Emb_{\Pmb^N} \big( \exp (J^N(T)) \big) = 1$, we have from Scheff\'e's theorem that $\exp \big( J^N(T) \big)$ is uniformly integrable.
Since $\big| \exp \big( i \sum_{\alpha \in \Kbd} \xitil^N_{\alpha}(\phi^{(\alpha)}) \big) \big| = 1$,
\begin{equation*}
	\exp \Big( i \sum_{\alpha \in \Kbd} \xitil^N_{\alpha}(\phi^{(\alpha)}) + J^N(T) \Big)
\end{equation*}
is also uniformly integrable. 
From~\eqref{eqchap2multitype:convergence of xi tilde and JN}, we have as $N \to \infty$,
\begin{equation*}
	\exp \Big( i \sum_{\alpha \in \Kbd} \xitil^N_{\alpha}(\phi^{(\alpha)}) + J^N(T) \Big) \Rightarrow \exp \Big( i \sum_{\alpha \in \Kbd} I_1(\phihatbd^{(\alpha)}_{\alpha}) + J \Big).
\end{equation*}
Hence, using uniform integrability,
\begin{align*} \label{eqchap2multitype:CLT proof final part}
	 \lim_{N \to \infty} \Emb_{\Pmb^N} \Big( \exp \big( i \sum_{\alpha \in \Kbd} \xitil^N_{\alpha}(\phi^{(\alpha)}) + J^N(T) \big) \Big) 
	&  = \Emb_{\tilde{\Pmb}} \Big[ \exp \big( i \sum_{\alpha \in \Kbd} I_1(\phihatbd^{(\alpha)}_{\alpha}) + J \big) \Big] \\
	&  = \exp \big( -\frac{1}{2} \| (I-A)^{-1} (\sum_{\alpha \in \Kbd} \phihatbd^{(\alpha)}_{\alpha}) \|^2_{L^2(\Omega_d^K,\nuhat)} \big),
\end{align*}
where the last equality is a consequence of Lemma~$1.3$ of~\cite{ShigaTanaka1985} (cf.\ Lemma~\ref{lemappendixchap2multitype:Shiga and Tanaka} in Appendix~\ref{Appendixchap2multitype:Restating}) and Lemma~\ref{lemchap2multitype:trace}.
Thus we have proved~\eqref{eqchap2multitype:convergence of characteristic function}, which completes the proof of Theorem~\ref{thmchap2multitype:CLT}. \qed


\subsection{Proof of Theorem~\ref{thmchap2multitype:general CLT}} \label{secchap2multitype:proof of general CLT}

It was argued in Section~\ref{secchap2multitype:completing the proof of CLT} that $\Emb_{\tilde{\Pmb}} \big( \exp (J) \big) = 1$.
Consider now the second statement in the theorem.
Recall the definition of $\xi^N(\phi)$ in~\eqref{eqchap2multitype:xi N}.
For $\phi \in \Amc^K$, let
\begin{equation*}
	\xitil^N(\phi) = \frac{1}{\sqrt{N_1 \dotsm N_K}} \sum_{i_1 \in \boldsymbol{N_1}} \dotsb \sum_{i_K \in \boldsymbol{N_K}} \phi(X^{i_1},\dotsc,X^{i_K}).
\end{equation*}
Then $\Pbd \circ (\xi^N(\phi))^{-1} = \Qmb^N \circ (\xitil^N(\phi))^{-1}$.
Using Lemma~\ref{lemchap2multitype:unbalanced populations} as for the proof of~\eqref{eqchap2multitype:convergence of xi tilde and JN}, we see that under $\Pmb^N$, with $\phibd = \phi(X_*(\cdot),\dotsc,X_*(\cdot))$, as $N \to \infty$,
\begin{equation*}
	\exp \Big( i \xitil^N(\phi) +J^N(T) \Big) \Rightarrow \exp \Big( i I_K(\phihatbd^{sym}) + J \Big).
\end{equation*}
As before, $\exp (i \xitil^N(\phi) +J^N(T))$ is uniformly integrable.
Hence as $N \to \infty$,
\begin{align*}
	\Emb_\Pbd \exp \Big( i \xi^N(\phi) \Big) & = \Emb_{\Qmb^N} \exp \Big( i \xitil^N(\phi) \Big) = \Emb_{\Pmb^N} \exp \Big( i \xitil^N(\phi) +J^N(T) \Big) \\
	& \to \Emb_{\tilde{\Pmb}} \exp \Big( i I_K(\phihatbd^{sym}) + J \Big) = \Emb_{\tilde{\Qmb}} \exp \Big( i I_K(\phihatbd^{sym}) \Big),
\end{align*}
which finishes the proof. \qed


\section{Conditions and notations for CLT in the common factor case} \label{secchap3commonfactor:preparations} 

In this section we will present the main condition that is assumed in Theorem~\ref{thmchap3commonfactor:CLT} and also introduce some functions and operators needed in its proof.


\subsection{Assumptions for the central limit theorem} \label{secchap3commonfactor:assumptions for CLT}

Consider the systems of equations given by~\eqref{eqchap3commonfactor:ZiNt}-\eqref{eqchap3commonfactor:mu N t} and~\eqref{eqchap3commonfactor:Xit}-\eqref{eqchap3commonfactor:mu t}.
Since, unlike the model considered in Section~\ref{secchap2multitype:model}, here the dependence of the coefficients on the empirical measure is nonlinear, we will need to impose suitable smoothness conditions.
These smoothness conditions can be formulated as follows.

Denote by $\Jmc$ [resp.\ $\Jmcbar$] the collection of all real measurable functions $f$ on $\Rmb^{d+m+d}$ [resp.\ $\Rmb^{m+d}$] that are bounded by $1$.
Denote by $\Thetabar$ the class of all $g : \Rmb^m \times [\Pmc(\Rmb^d)]^K \to \Rmb^m$ such that there exist $c_g \in (0,\infty)$; a finite subset $\Jmcbar_g$ of $\Jmcbar$; continuous and bounded functions $g_{(1)}, g_{(2)}$ from $\Rmb^m \times [\Pmc(\Rmb^d)]^K$ to $\Rmb^{m \times m}$ and from $\Rmb^m \times [\Pmc(\Rmb^d)]^K \times \Rmb^d$ to $\Rmb^m$ respectively; and $\theta_g : \Rmb^m \times [\Pmc(\Rmb^d)]^K \times \Rmb^m \times [\Pmc(\Rmb^d)]^K \to \Rmb^m$ such that for all $\rbd = (y, \nubd), \rbd' = (y', \nubd') \in \Rmb^m \times [\Pmc(\Rmb^d)]^K$
\begin{equation*}
	g(\rbd') - g(\rbd) = g_{(1)}(\rbd) (y' - y) + \sum_{\gamma=1}^K \langle g_{(2),\gamma}(\rbd,\cdot), (\nu_\gamma' - \nu_\gamma) \rangle + \theta_g(\rbd,\rbd'),
\end{equation*}
and
\begin{equation*}
	\| g(\rbd') - g(\rbd) \| \le c_g \big( \| y' - y \| + \sum_{\gamma=1}^K \max_{f \in \Jmcbar_g} | \langle f(y,\cdot) , (\nu_\gamma'-\nu_\gamma) \rangle | \big),
\end{equation*}
where
\begin{equation*} 
	\| \theta_g(\rbd,\rbd') \| \le c_g \big( \| y' - y \|^2 + \sum_{\gamma=1}^K \max_{f \in \Jmcbar_g} | \langle f(y,\cdot) , (\nu_\gamma'-\nu_\gamma) \rangle |^2 \big).
\end{equation*}
Write $\sigmabar = (\sigmabar_1,\dotsc,\sigmabar_m)$ where $\sigmabar_k$ is an $\Rmb^m$-valued function for $k=1,\dotsc,m$. 
The following will be the key assumptions needed in Theorem~\ref{thmchap3commonfactor:CLT}.

\begin{condition} \label{condchap3commonfactor:cond2}
	$\bbar$, $\sigmabar_k$, $k=1,\dotsc,m$ are in class $\Thetabar$.
\end{condition}

We impose analogous smoothness conditions on $b_\alpha$ as follows.

\begin{condition} \label{condchap3commonfactor:cond3}
	There exist $c_b \in (0,\infty)$; a finite subset $\Jmc_b$ of $\Jmc$; continuous and bounded functions $b_{\alpha,(1)}, b_{\alpha\gamma,(2)}$ from $\Rmb^{d+m} \times [\Pmc(\Rmb^d)]^K$ to $\Rmb^{d \times m}$ and from $\Rmb^{d+m} \times [\Pmc(\Rmb^d)]^K \times \Rmb^d$ to $\Rmb^d$ respectively; and $\theta_{b_\alpha} : \Rmb^{d+m} \times [\Pmc(\Rmb^d)]^K \times \Rmb^m \times [\Pmc(\Rmb^d)]^K \to \Rmb^d$ such that for all $\alpha \in \Kbd$, $x \in \Rmb^d$, $\rbd = (y,\nubd)$ and $\rbd' = (y',\nubd') \in \Rmb^m \otimes [\Pmc(\Rmb^d)]^K$
		\begin{equation*}
			b_\alpha(x,\rbd') - b_\alpha(x,\rbd) = b_{\alpha,(1)}(x,\rbd) (y' - y) + \sum_{\gamma=1}^K \langle b_{\alpha\gamma,(2)}(x,\rbd,\cdot), (\nu_\gamma' - \nu_\gamma) \rangle + \theta_{b_\alpha}(x,\rbd,\rbd')
		\end{equation*}
	and
		\begin{equation} \label{eqchap3commonfactor:assumptions of theta b}
			\| \theta_{b_\alpha}(x,\rbd,\rbd') \| \le c_b \big( \| y' - y \|^2 + \sum_{\gamma=1}^K \max_{f \in \Jmc_b} | \langle f(x,y,\cdot) , (\nu_\gamma'-\nu_\gamma) \rangle |^2 \big).
		\end{equation}
\end{condition}

The above conditions on $b_\alpha, \bbar, \sigmabar$ are satisfied quite generally.
For example, let $d = m = 1$ and
\begin{gather*}
	b_\alpha(x,y,\nubd) = b_{\alpha,0}(x, y, \langle f_{1},\nubd \rangle, \dotsc, \langle f_{k},\nubd \rangle), \\
	\bbar(y,\nubd) = \bbar_0(y, \langle f_1,\nubd \rangle, \dotsc, \langle f_k,\nubd \rangle), \\
	\sigmabar(y,\nubd) = \sigmabar_0(y, \langle f_1,\nubd \rangle, \dotsc, \langle f_k,\nubd \rangle),
\end{gather*}
where $b_{\alpha,0}$, $\bbar_0$ and $\sigmabar_0$ are bounded twice continuously differentiable functions with bounded derivatives; and $f_i$ are bounded Lipschitz functions. 
Then above conditions hold.
We refer the reader to~\cite{BudhirajaSaha2014} for details and other examples.


\subsection{Canonical processes} \label{secchap3commonfactor:canonical processes}

Recall the canonical space $\Omegabar^N = \Omega_m \times \Omega_d^N$ defined in Section~\ref{secchap3commonfactor:CLT}.
We introduce the following canonical stochastic processes. 

For $\omega = (\omegabar, \omega_1, \omega_2, \dotsc, \omega_N) \in \Omegabar^N$, let $V^i(\omega) = \omega_i, i \in \Nbd$ and $\Vbar(\omega) = \omegabar$.
As before, abusing notation, we write
\begin{equation*}
	V^i = (W^i, X^i), i \in \Nbd, \quad \Vbar = (\Wbar, Y).
\end{equation*}
Also with $\Pi^\alpha$ as in Remark~\ref{rmkchap3commonfactor:rmk1} let $\mu^\alpha = \Pi^\alpha(Y_0,\Wbar)$ for $\alpha \in \Kbd$ and $\mubd = (\mu^1,\dotsc,\mu^K)$. Recall $\Pmbbar^N \in \mathcal{P}(\Omegabar^N)$ introduced in~\eqref{eqchap3commonfactor:P N}.
With these definitions, under $\Pmbbar^N$,~\eqref{eqchap3commonfactor:Xit}-\eqref{eqchap3commonfactor:mu t} are satisfied a.s.\ for $i \in \Nbd$, where $\mu^\alpha_t$ is the marginal of $\mu^\alpha$ at time instant $t$; also with $\Gmc_t = \sigma \{ Y_0, \Wbar_s, s \le t\}$ and $\Gmc = \Gmc_T$, $\mu^\alpha_t = \Lmc(\Xit | \Gmc) = \Lmc(\Xit | \Gmc_t)$, $t \in [0,T]$, $i \in \Nalpha$, $\alpha \in \Kbd$, $\Pmbbar^N$ a.s.; and $Y$ is $\{ \Gmc_t \}$ adapted.
Recall the process $V_* = (W_*, X_*)$ defined on $\Omega_d$ in Section~\ref{secchap2multitype:canonical processes}.
Also define $\Vbar_* = (\Wbar_*, Y_*)$ on $\Omega_m$ as follows:
For $\omegabar = (\omegabar_1, \omegabar_2) \in \Omega_m$,
\begin{equation*}
	\Vbar_*(\omegabar) = (\Wbar_*(\omegabar), Y_*(\omegabar)) = (\omegabar_1, \omegabar_2).
\end{equation*}
Let $\mu^\alpha_* : \Omega_m \to \Pmc(\Cmc_d)$ be defined as $\mu^\alpha_*(\omegabar) = \Pi^\alpha(Y_{*,0}(\omegabar), \Wbar_*(\omegabar))$ for $\omegabar \in \Omega_m$ and $\alpha \in \Kbd$.
Note that $t \mapsto \mu^\alpha_{*,t}$ is a continuous map, i.e.\ an element of $\Cmb_{\Pmc(\Rmb^d)}[0,T]$, which once more  we will denote as $\mu^\alpha_*$.
Finally let $\mubd_* = (\mu^1_*,\dotsc,\mu^K_*)$ and $\Dbd_* = (\Wbar_*, Y_*, \mubd_*)$.


\subsection{Some random integral operators} \label{secchap3commonfactor:some random integral operators}

We now introduce some random integral operators, similar to the integral operators introduced in Section~\ref{secchap2multitype:some integral operators}, which will be needed to formulate the CLT.
Randomness of the integral operators is due to the fact that the kernel function of these operators will depend on the common factor.
Recall $b_{\alpha\gamma,(2)}$ in Condition~\ref{condchap3commonfactor:cond3}.
Define function $b^c_{\alpha\gamma,(2)}$ from $\Rmb^{d+m} \times [\Pmc(\Rmb^d)]^K \times \Rmb^d$ to $\Rmb^d$ as follows: 
For $(x,\rbd,\xtil) \in \Rmb^{d+m} \times [\Pmc(\Rmb^d)]^K \times \Rmb^d$ with $\rbd = (y,\nubd)$,
\begin{equation*}
	b^c_{\alpha\gamma,(2)}(x,\rbd,\xtil) = b_{\alpha\gamma,(2)}(x,\rbd,\xtil) - \langle b_{\alpha\gamma,(2)}(x,\rbd,\cdot), \nu_\gamma \rangle.
\end{equation*}
Recall $\bbar_{(2),\gamma}$ and $\sigmabar_{k,(2),\gamma}$ in Condition~\ref{condchap3commonfactor:cond2}.
Similarly, define functions $\bbar^c_{(2),\gamma}$ and $\sigmabar^c_{k,(2),\gamma}$ from $\Rmb^m \times [\Pmc(\Rmb^d)]^K \times \Rmb^d$ to $\Rmb^m$ as follows:
For $(\rbd,\xtil) \in \Rmb^m \times [\Pmc(\Rmb^d)]^K \times \Rmb^d$ with $\rbd = (y,\nubd)$,
\begin{align}
	\bbar^c_{(2),\gamma}(\rbd,\xtil) & = \bbar_{(2),\gamma}(\rbd,\xtil) - \langle \bbar_{(2),\gamma}(\rbd,\cdot), \nu_\gamma \rangle, \label{eqchap3commonfactor:bbar c}\\
	\sigmabar^c_{k,(2),\gamma}(\rbd,\xtil) & = \sigmabar_{k,(2),\gamma}(\rbd,\xtil) - \langle \sigmabar_{k,(2),\gamma}(\rbd,\cdot), \nu_\gamma \rangle. \label{eqchap3commonfactor:sigmabar c}
\end{align}
We now introduce another function given on a suitable path space that will be used to define the kernels in our integral operators.
One ingredient in the definition of this function requires additional notational preparation and its precise definition is postponed to Section~\ref{secchap3commonfactor:proofs}.
Define for $t \in [0,T]$ and $\alpha,\gamma \in \Kbd$, the function $\fbd_{\alpha\gamma,t}$ from $\Rmb^d \times \Cmb_{\Rmb^{d+2m} \times [\Pmc(\Rmb^d)]^K}[0,t]$ to $\Rmb^d$ as follows:
For $(x^{(1)},x^{(2)}_{[0,t]},\dbd_{[0,t]}) \in \Rmb^d \times \Cmb_{\Rmb^{d+2m} \times [\Pmc(\Rmb^d)]^K}[0,t]$ with $\dbd = (w,\rbd) = (w,y,\nubd)$,
\begin{equation} \label{eqchap3commonfactor:f alpha gamma t}
	\fbd_{\alpha\gamma,t}(x^{(1)},x^{(2)}_{[0,t]},\dbd_{[0,t]}) = b_{\alpha\gamma,(2)}^c(x^{(1)},\rbd_t,x^{(2)}_t) + b_{\alpha,(1)}(x^{(1)},\rbd_t) \sbd_{\gamma,t}(x^{(2)}_{[0,t]},\dbd_{[0,t]}),
\end{equation}
where the function $\sbd_{\gamma,t}$ from $\Cmb_{\Rmb^{d+2m} \times [\Pmc(\Rmb^d)]^K }[0,t]$ to $\Rmb^m$ will be introduced in Lemma~\ref{lemchap3commonfactor:difference of Y}.

Recall the transition probability kernel $\rho_\alpha$ introduced below~\eqref{eqchap3commonfactor:P N}. 
Fix $\omegabar \in \Omega_m$ and consider the Hilbert space $\Hmc_{\omegabar} = L^2(\Omega_d^K,\rhohat(\omegabar,\cdot))$, where $\rhohat(\omegabar,d\omega_1,\dotsc,d\omega_K) = \rho_1(\omegabar,d\omega_1) \otimes \dotsb \otimes \rho_K(\omegabar,d\omega_K)$.
Define for $\bar P$ a.e.\ $\bar \omega$, $h_{\omegabar}^{\alpha\gamma} \in L^2(\Omega_d \times \Omega_d, \rho_\alpha(\omegabar,\cdot) \times \rho_\gamma(\omegabar,\cdot))$ as
\begin{equation*}
	h_{\omegabar}^{\alpha\gamma}(\omega,\omega') = \sqrt{\frac{\lambda_\alpha}{\lambda_\gamma}} \int_0^T \fbd_{\alpha\gamma,t}(X_{*,t}(\omega),X_{*,[0,t]}(\omega'),\Dbd_{*,[0,t]}(\omegabar)) \cdot dW_{*,t}(\omega), \quad (\omega,\omega') \in \Omega_d \times \Omega_d.
\end{equation*}
Let $\hhat_{\omegabar}^{\alpha\gamma} \in L^2(\Omega_d^K \times \Omega_d^K, \rhohat(\omegabar,\cdot) \otimes \rhohat(\omegabar,\cdot))$ be the lifted version of $h_{\omegabar}^{\alpha\gamma}$, namely $\hhat_{\omegabar}^{\alpha\gamma}(\omega,\omega') = h_{\omegabar}^{\alpha\gamma}(\omega_\alpha,\omega'_\gamma)$ for $\omega = (\omega_1,\dotsc,\omega_K) \in \Omega_d^K$ and $\omega' = (\omega_1',\dotsc,\omega_K') \in \Omega_d^K$.
Define the integral operator $A^{\alpha\gamma}_{\omegabar}$ on $\Hmc_\omegabar$ as follows.
For $g \in \Hmc_{\omegabar}$ and $\omega \in \Omega_d^K$,
\begin{equation} \label{eqchap3commonfactor:A alpha gamma}
	A^{\alpha\gamma}_{\omegabar} g(\omega) = \int_{\Omega_d^K} g(\omega') \hhat_{\omegabar}^{\alpha\gamma}(\omega',\omega) \, \rhohat(\omegabar,d\omega').  
\end{equation}
Let $A_{\omegabar} = \sum_{\alpha,\gamma=1}^{K} A^{\alpha\gamma}_{\omegabar}$.
Then this is the integral operator on $\Hmc_\omegabar$ associated with the kernel $\hhat_\omegabar = \sum_{\alpha,\gamma=1}^K \hhat_{\omegabar}^{\alpha\gamma} \in L^2(\Omega_d^K \times \Omega_d^K, \rhohat(\omegabar,\cdot) \otimes \rhohat(\omegabar,\cdot))$. 
Denote by $I$ the identity operator on $\Hmc_{\omegabar}$.
The following lemma is proved similarly as Lemma~\ref{lemchap2multitype:trace}.
Proof is omitted.
\begin{lemma} \label{lemchap3commonfactor:Trace}
	For $\Pbar$ a.e.\ $\omegabar$, $(a)$ \textnormal{Trace}$(A_{\omegabar}^n) = 0$ for all $n \ge 2$, and $(b)$ $I - A_{\omegabar}$ is invertible.
\end{lemma}

Recall the collection $\Amcbar_\alpha, \alpha \in \Kbd$ introduced in Section~\ref{secchap3commonfactor:CLT}.
For $\alpha \in \Kbd$, $\phi_\alpha \in \Amcbar_\alpha$ and $\omegabar \in \Omega_m$, let
\begin{gather}
	\Phi^\alpha_\omegabar(\omega) = \phi_\alpha(X_*(\omega)) - m_{\phi_\alpha}^\alpha(\omegabar), \quad \omega \in \Omega_d \notag \\
	\Sigma^{\alpha\gamma}_\omegabar = \langle (I - A_\omegabar)^{-1} \hat{\Phi}^\alpha_\omegabar, (I - A_\omegabar)^{-1} \hat{\Phi}^\gamma_\omegabar \rangle_{L^2(\Omega_d^K,\rhohat(\omegabar,\cdot))}, \label{eqchap3commonfactor:Sigma omegabar}
\end{gather} 
where $m_{\phi_\alpha}^\alpha$ is as in~\eqref{eqchap3commonfactor:m alpha phi alpha} and $\hat{\Phi}^\alpha_\omegabar$ is the lifted function defined as in~\eqref{eqchap2multitype:phihat alpha}, namely for $\omega = (\omega_1,\dotsc,\omega_K) \in \Omega_d^K$ and $\alpha \in \Kbd$,
\begin{equation} \label{eqchap3commonfactor:Phi hat alpha}
	\hat{\Phi}^\alpha_\omegabar(\omega) = \Phi^\alpha_\omegabar(\omega_\alpha) = \phi_\alpha(X_*(\omega_\alpha)) - m_{\phi_\alpha}^\alpha(\omegabar).
\end{equation}
The quantities $\Sigma^{\alpha\gamma}_\omegabar$, $\alpha,\gamma \in \Kbd$ were used in Section~\ref{secchap3commonfactor:CLT} to characterize the limit distribution of $(\Vmc^N_1(\phi_1), \dotsc, \Vmc^N_K(\phi_K))$.
In particular, recall that $\pi_\omegabar(\phi_1,\dotsc,\phi_K)$ is the $K$-dimensional multivariate normal distribution with mean $0$ and variance-covariance matrix $\Sigma_\omegabar = (\Sigma^{\alpha\gamma}_\omegabar)_{\alpha,\gamma \in \Kbd}$, and $\pi(\phi_1,\dotsc,\phi_K)$ is the Gaussian mixture defined by~\eqref{eqchap3commonfactor:pi phi 1 phi K}.
Theorem~\ref{thmchap3commonfactor:CLT}, which is proved in Section~\ref{secchap3commonfactor:proofs} below, says that under Conditions~\ref{condchap3commonfactor:cond1},~\ref{condchap3commonfactor:cond2} and~\ref{condchap3commonfactor:cond3}, $(\Vmc^N_1(\phi_1), \dotsc, \Vmc^N_K(\phi_K))$ converges in distribution to $\pi(\phi_1,\dotsc,\phi_K)$, where $\Vmc_{i}^N(\phi_i)$ are as in~\eqref{eqchap3commonfactor:V alpha phi alpha}.


\section{Proofs of Theorem~\ref{thmchap3commonfactor:CLT} and related results} \label{secchap3commonfactor:proofs}

In this section we will present the proof of Theorem~\ref{thmchap3commonfactor:CLT}.
With $\{V^i=(W^i,X^i)\}_{i \in \Nbd}$, $\Vbar=(\Wbar,Y)$ as introduced in Section~\ref{secchap3commonfactor:canonical processes}, define $Y^N$ as the unique solution of the following equation
\begin{equation} \label{eqchap3commonfactor:YNt}
	Y^N_t = Y_0 + \int_0^t \bbar(Y^N_s,\mubd^N_s) \, ds + \int_0^t \sigmabar(Y^N_s,\mubd^N_s) \, d\Wbar_s, 
\end{equation}
where $\mubd^N_t = (\mu^{1,N}_t, \dotsc, \mu^{K,N}_t)$ and $\mu^{\gamma,N}_t = \frac{1}{N_\gamma} \sum_{j \in \Ngamma} \delta_{\Xjt}$ for $\gamma \in \Kbd$.
We begin in Section~\ref{secchap3commonfactor:Grisanov's Change of Measure} by introducing the Girsanov's change of measure that is key to the proof.
The main additional work required for the proof of CLT in the presence of a common factor is in the estimation of the difference between $Y^N$ and $Y$.
This is done in Section~\ref{secchap3commonfactor:study YN - Y}.
These estimates are used in Sections~\ref{secchap3commonfactor:asymptotics of JN1} and~\ref{secchap3commonfactor:asymptotics of JN2} to study the asymptotics of the Radon-Nikodym derivative.
Finally Sections~\ref{secchap3commonfactor:combining contributions from JN1 and JN2} and~\ref{secchap3commonfactor:completing the proof} combine these asymptotic results to complete the proof of Theorem~\ref{thmchap3commonfactor:CLT}.


\subsection{Girsanov's change of measure} \label{secchap3commonfactor:Grisanov's Change of Measure}

Let $\Rbd^N = (Y^N,\mubd^N)$ and $\Rbd = (Y,\mubd)$.
For $t \in [0,T]$, define $\{ J^N(t) \}$ as
\begin{equation} \label{eqchap3commonfactor:JN}
	J^N(t) = J^{N,1}(t) - \frac{1}{2} J^{N,2}(t),
\end{equation}
where
\begin{equation} \label{eqchap3commonfactor:JN1}
	J^{N,1}(t) = \sum_{\alpha=1}^K \sum_{i \in \Nalpha} \int_0^t [b_\alpha(\Xis,\Rbd^N_s) - b_\alpha(\Xis,\Rbd_s)] \cdot d\Wis,
\end{equation}
and
\begin{equation} \label{eqchap3commonfactor:JN2}
	J^{N,2}(t) = \sum_{\alpha=1}^K \sum_{i \in \Nalpha} \int_0^t \| b_\alpha(\Xis,\Rbd^N_s) - b_\alpha(\Xis,\Rbd_s) \|^2 \, ds.
\end{equation}
Letting for $t \in [0,T]$, $\bar{\Fmc}^N_t = \sigma \{ \Vbar(s), V^i(s) : 0 \le s \le t, i \in \Nbd \}$, we see that $\{ \exp (J^N(t)) \}$ is a $\{ \bar{\Fmc}^N_t \}$-martingale under $\Pmbbar^N$.
Define a new probability measure $\Qmbbar^N$ on $\Omegabar^N$ by
\begin{equation} \label{eqchap3commonfactor:QN}
	\frac{d\Qmbbar^N}{d\Pmbbar^N} = \exp (J^N(T)).
\end{equation}

By Girsanov's theorem, $(X^1,\dotsc,X^N,Y^N,\Vbar)$ has the same probability law under $\Qmbbar^N$ as $(Z^{1,N},\dotsc,Z^{N,N},U^N,V^0)$ under $\Pbd$, where $V^0$ is introduced below~\eqref{eqchap3commonfactor:V alpha phi alpha}.
For $\phi_\alpha \in \Amcbar_\alpha$, $\alpha \in \Kbd$, let
\begin{equation} \label{eqchap3commonfactor:Vtil alpha phi alpha}
	\Vmctil^N_\alpha(\phi_\alpha) = \sqrt{N_\alpha} \big( \frac{1}{N_\alpha} \sum_{i \in \Nalpha} \phi_\alpha (X^i) - m^\alpha_{\phi_\alpha} (\Vbar) \big).
\end{equation}
In order to prove the theorem, from the definition of $\Sigma_\omegabar$ in Section~\ref{secchap3commonfactor:some random integral operators} and $\pi(\phi_1,\dotsc,\phi_K)$ in Section~\ref{secchap3commonfactor:CLT}, it suffices to show that
\begin{equation*}
	\lim_{N \to \infty} \Emb_{\Qmbbar^N} \exp \Big( i \sum_{\alpha=1}^K \Vmctil^N_\alpha(\phi_\alpha) \Big) = \int_{\Omega_m} \exp \Big( -\frac{1}{2} \Big\| (I-A_\omegabar)^{-1} \sum_{\alpha = 1}^K \hat{\Phi}^\alpha_\omegabar \Big\|^2_{L^2(\Omega_d^K,\rhohat(\omegabar,\cdot))} \Big) \Pbar(d\omegabar),
\end{equation*}
which from~\eqref{eqchap3commonfactor:QN} is equivalent to showing
\begin{align} \label{eqchap3commonfactor:convergence of characteristic function}
	\begin{aligned}
		& \lim_{N \to \infty} \Emb_{\Pmbbar^N} \exp \Big( i \sum_{\alpha=1}^K \Vmctil^N_\alpha(\phi_\alpha) + J^{N,1}(T) - \frac{1}{2} J^{N,2}(T) \Big) \\
		& \hspace{20 mm} = \int_{\Omega_m} \exp \Big( -\frac{1}{2} \Big\| (I-A_\omegabar)^{-1} \sum_{\alpha = 1}^K \hat{\Phi}^\alpha_\omegabar \Big\|^2_{L^2(\Omega_d^K,\rhohat(\omegabar,\cdot))} \Big) \Pbar(d\omegabar).
	\end{aligned}
\end{align}
The above equality will be established in Section~\ref{secchap3commonfactor:completing the proof}.


\subsection{Studying $Y^N - Y$} \label{secchap3commonfactor:study YN - Y}

The following lemma is an immediate consequence of the fact that for each $\gamma \in \Kbd$, conditionally on $\Gmc$, $X^j$ are i.i.d. for $j \in \Ngamma$.
Proof is omitted.

\begin{lemma} \label{lemchap3commonfactor:order}
	For each $\gamma \in \Kbd$ and $r \in \Nmb$, there exists $\atil_r \in (0,\infty)$ such that for all $N \in \Nmb$
	\begin{equation*}
		\sup_{\| f \|_\infty \le 1} \EmbPbar | \langle f(\cdot), (\mu^\gamma_t - \mu^{\gamma,N}_t) \rangle |^r \le \frac{\atil_r}{N_\gamma^{r/2}}.
	\end{equation*}
\end{lemma}

As an immediate consequence of the above lemma we have the following lemma.

\begin{lemma} \label{lemchap3commonfactor:expectation of difference of Y}
	For each $r \in \Nmb$, there exists $a_r \in (0,\infty)$ such that
	\begin{equation*}
		\sup_{t \in [0,T]} \EmbPbar \| Y^N_t - Y_t \|^r \le \frac{a_r}{N^{r/2}}.
	\end{equation*}
\end{lemma}

\begin{proof}
	Fix $r \in \Nmb$ and $t \in [0,T]$. By~\eqref{eqchap3commonfactor:Yt},~\eqref{eqchap3commonfactor:YNt} and Burkholder-Davis-Gundy inequality, there exist $\kappa_1, \kappa_2 \in (0,\infty)$ such that
	\begin{align*}
		& \EmbPbar \| Y^N_t - Y_t \|^r \\
		& \quad \le \kappa_1 \EmbPbar \int_0^t \| \bbar(\Rbd^N_s) - \bbar(\Rbd_s) \|^r \, ds + \kappa_1 \EmbPbar \big( \int_0^t \| \sigmabar(\Rbd^N_s) - \sigmabar(\Rbd_s) \|^2 \, ds \big)^{r/2} \\
		& \quad \le \kappa_2 \EmbPbar \int_0^t \| Y^N_s - Y_s \|^r \, ds + \kappa_2 \EmbPbar \sum_{\gamma=1}^K \int_0^t \max_{f \in \Jmcbar_\bbar \cup \Jmcbar_\sigmabar} | \langle f(Y_s,\cdot),(\mu^{\gamma,N}_s - \mu^\gamma_s) \rangle |^r \, ds,
	\end{align*}
	where the last inequality follows from the fact that $\bbar$ and $\sigmabar$ satisfy Condition~\ref{condchap3commonfactor:cond2}. 
	Thus by Gronwall's inequality and Lemma~\ref{lemchap3commonfactor:order}, $\forall N \in \Nmb$
	\begin{equation*}
		\EmbPbar \| Y^N_t - Y_t \|^r \le \kappa_3 \sum_{\gamma=1}^K \int_0^t \EmbPbar \max_{f \in \Jmcbar_\bbar \cup \Jmcbar_\sigmabar} | \langle f(Y_s,\cdot),(\mu^{\gamma,N}_s - \mu^\gamma_s) \rangle |^r \, ds \le \frac{\kappa_4}{N^{r/2}}
	\end{equation*}
	for some $\kappa_4 \in (0,\infty)$.
\end{proof}

The following lemma follows from standard uniqueness results for stochastic differential equations (see eg. Theorem~$5.1.1$ of~\cite{Kallianpur1980}) and straightforward applications of It$\hat{\text{o}}$'s formula.
Recall the canonical space $\Omegabar^N$ from Section~\ref{secchap3commonfactor:CLT} along with the Borel $\sigma$-field $\bar{\Fmc}^N = \Bmc(\Omegabar^N)$ and probability measures $\Pmbbar^N$ (see~\eqref{eqchap3commonfactor:P N}).
Let $\{ \bar{\Fmc}^N_t \}_{t \in [0,T]}$ denote the canonical filtration on $(\Omegabar^N, \bar{\Fmc}^N)$.
Note that $\Wbar$ introduced in Section~\ref{secchap3commonfactor:canonical processes} is an $m$-dimensional $\{ \bar{\Fmc}^N_t \}$-BM under $\Pmbbar^N$.

\begin{lemma} \label{lemchap3commonfactor:SDE result}
	Let $\{ A_t \}_{t \in [0,T]}$, $\{ F^k_t \}_{t \in [0,T]}$, $k = 1,\dotsc,m$ be continuous bounded $\{ \bar{\Fmc}^N_t \}$-adapted processes with values in $\Rmb^{m \times m}$, given on $(\Omegabar^N, \bar{\Fmc}^N, \Pmbbar^N)$.
	Also let $\{ a_t \}_{t \in [0,T]}$, $\{ f^k_t \}_{t \in [0,T]}$, $k = 1,\dotsc,m$ be progressively measurable processes with values in $\Rmb^m$ such that
	\begin{equation*}
		\EmbPbar \int_0^T \|a_s\|^2 \, ds + \sum_{k=1}^m \EmbPbar \int_0^T \|f^k_s\|^2 \, ds < \infty.
	\end{equation*}
	Write $\Wbar = (\Wbar^1, \dotsc, \Wbar^m)$. 
	Then
	
	\noi $(a)$ The following $m \times m$ dimensional equation has a unique pathwise solution:
	\begin{align}
		\Phi_t & = I_m + \int_0^t A_s \Phi_s \, ds + \sum_{k=1}^m \int_0^t F^k_s \Phi_s \, d\Wbar^k_s, \label{eqchap3commonfactor:fundamental solution} \\
		\Psi_t & = I_m - \int_0^t \Psi_s A_s \, ds - \sum_{k=1}^m \int_0^t \Psi_s F^k_s \, d\Wbar^k_s + \sum_{k=1}^m \int_0^t \Psi_s (F^k_s)^2 \, ds, \label{eqchap3commonfactor:inverse of fundamental solution}
	\end{align}
	where $I_m$ is the $m \times m$ identity matrix.
	Furthermore, $\Phi_t$, $\Psi_t$ are $m \times m$ invertible matrices a.s.\ and $\Psi_t = \Phi_t^{-1}$. 

	\noi $(b)$ Given a square integrable $\bar{\Fmc}^N_0$-measurable random variable $\Yhat_0$, the following $m$-dimensional equation has a unique pathwise solution:
	\begin{equation} \label{eqchap3commonfactor:SDE}
		\Yhat_t = \Yhat_0 + \int_0^t (A_s \Yhat_s + a_s) \, ds + \sum_{k=1}^m \int_0^t (F^k_s \Yhat_s + f^k_s) \, d\Wbar^k_s.
	\end{equation}
	Furthermore the solution is given as
	\begin{equation} \label{eqchap3commonfactor:solution to SDE}
		\Yhat_t = \Phi_t \Big[ \Yhat_0 + \int_0^t \Phi_s^{-1} a_s \, ds + \sum_{k=1}^m \int_0^t \Phi_s^{-1} f^k_s \, d\Wbar^k_s - \sum_{k=1}^m \int_0^t \Phi_s^{-1} F^k_s f^k_s \, ds \Big].
	\end{equation}
\end{lemma}

The following lemma will give a useful representation for $Y^N - Y$, and the function $\sbd_{\gamma,t}$ from $\Cmb_{\Rmb^{d+2m} \times [\Pmc(\Rmb^m)]^K}[0,t]$ to $\Rmb^m$ introduced in this lemma is used to define the integral operator $A^{\alpha\gamma}_\omegabar$ in Section~\ref{secchap3commonfactor:some random integral operators}.
Recall the functions $g_{(1)}$, $g_{(2)}$ and $\theta_g$ introduced above Condition~\ref{condchap3commonfactor:cond2} and centered functions defined in~\eqref{eqchap3commonfactor:bbar c} and~\eqref{eqchap3commonfactor:sigmabar c}.
Let $\Dbd = (\Wbar,\Rbd) = (\Wbar,Y,\mubd)$.

\begin{lemma} \label{lemchap3commonfactor:difference of Y}
	For $t \in [0,T]$,
	\begin{equation*}
		Y^N_t - Y_t = \sum_{\gamma=1}^K \frac{1}{N_\gamma} \sum_{j \in \Ngamma} \sbd_{\gamma,t}(X^j_{[0,t]},\Dbd_{[0,t]}) + \Tmc^N_1(t),
	\end{equation*}
	where
	\begin{align}
		\sbd_{\gamma,t}(X^j_{[0,t]},\Dbd_{[0,t]}) & = \Emc_t \int_0^t \Emc_s^{-1} \bbar^c_{(2),\gamma}(\Rbd_s,X^j_s) \, ds + \sum_{k=1}^m \Emc_t \int_0^t \Emc_s^{-1} \sigmabar^c_{k,(2),\gamma}(\Rbd_s,X^j_s) \, d\Wbar^k_s \notag \\
		& \qquad - \sum_{k=1}^m \Emc_t \int_0^t \Emc_s^{-1} \sigmabar_{k,(1)}(\Rbd_s) \sigmabar^c_{k,(2),\gamma}(\Rbd_s,X^j_s) \, ds \label{eqchap3commonfactor:s gamma t}\\
		\Tmc^N_1(t) & = \Emc_t \int_0^t \Emc_s^{-1} \theta_{\bbar}(\Rbd_s,\Rbd^N_s) \, ds + \sum_{k=1}^m \Emc_t \int_0^t \Emc_s^{-1} \theta_{\sigmabar_k}(\Rbd_s,\Rbd^N_s) \, d\Wbar^k_s \notag \\
		& \qquad - \sum_{k=1}^m \Emc_t \int_0^t \Emc_s^{-1} \sigmabar_{k,(1)}(\Rbd_s) \theta_{\sigmabar_k}(\Rbd_s,\Rbd^N_s) \, ds, \notag
	\end{align}
	and $\Emc_t = \tilde{\Emc}_t(\Dbd_{[0,t]})$ is the unique solution of the $m \times m$ dimensional SDE
	\begin{equation*}
		\Emc_t = I_m + \int_0^t \bbar_{(1)}(\Rbd_s) \Emc_s \, ds + \sum_{k=1}^m \int_0^t \sigmabar_{k,(1)}(\Rbd_s) \Emc_s \, d\Wbar^k_s.
	\end{equation*}
\end{lemma}

\begin{proof}
	For $t \in [0,T]$, we have
	\begin{align*}
		& Y^N_t - Y_t \\
		& \quad = \int_0^t \Big( \bbar(\Rbd^N_s) - \bbar(\Rbd_s) \Big) \, ds + \sum_{k=1}^m \int_0^t \Big( \sigmabar_k(\Rbd^N_s) - \sigmabar_k(\Rbd_s) \Big) \, d\Wbar^k_s \\
		& \quad = \int_0^t \Big( \bbar_{(1)}(\Rbd_s) (Y^N_s - Y_s) + \sum_{\gamma=1}^K \langle \bbar_{(2),\gamma}(\Rbd_s,\cdot), (\mu^{\gamma,N}_s - \mu^\gamma_s) \rangle + \theta_{\bbar}(\Rbd_s,\Rbd^N_s) \Big) \, ds \\
		& \qquad + \sum_{k=1}^m \int_0^t \Big( \sigmabar_{k,(1)}(\Rbd_s) (Y^N_s - Y_s) + \sum_{\gamma=1}^K \langle \sigmabar_{k,(2),\gamma}(\Rbd_s,\cdot), (\mu^{\gamma,N}_s - \mu^\gamma_s) \rangle + \theta_{\sigmabar_k}(\Rbd_s,\Rbd^N_s) \Big) \, d\Wbar^k_s.
	\end{align*}
	The result is now immediate on applying Lemma~\ref{lemchap3commonfactor:SDE result} with $\Yhat = Y^N - Y$, $\Phi = \Emc$ and
	\begin{align*}
		A_s & = \bbar_{(1)}(\Rbd_s), & a_s & = \sum_{\gamma=1}^K \langle \bbar_{(2),\gamma}(\Rbd_s,\cdot), (\mu^{\gamma,N}_s - \mu^\gamma_s) \rangle + \theta_{\bbar}(\Rbd_s,\Rbd^N_s), & & \quad & \\
		F^k_s & = \sigmabar_{k,(1)}(\Rbd_s), & f^k_s & = \sum_{\gamma=1}^K \langle \sigmabar_{k,(2),\gamma}(\Rbd_s,\cdot), (\mu^{\gamma,N}_s - \mu^\gamma_s) \rangle + \theta_{\sigmabar_k}(\Rbd_s,\Rbd^N_s), \quad k= 1,\dotsc,m. & & \quad & \qedhere
	\end{align*}
\end{proof}

\begin{lemma} \label{lemchap3commonfactor:expectation of Emc Smc and T1}
	For every $r \in \Nmb$, we have $	\sup_{t \in [0,T]} \EmbPbar \| \Emc_t \|^r < \infty$, $\sup_{t \in [0,T]} \EmbPbar \| \Emc_t^{-1} \|^r < \infty$ and 
$$
\max_{\gamma \in \Kbd} \sup_{t \in [0,T]} \sup_{j \in \Ngamma} \EmbPbar \| \sbd_{\gamma,t}(X^j_{[0,t]},\Dbd_{[0,t]}) \|^r < \infty.
	$$
	There exist $a_0 \in (0,\infty)$ such that for all $t \in [0,T]$,
	\begin{equation*}
		\EmbPbar \Big\| \sum_{\gamma=1}^K \frac{1}{N_\gamma} \sum_{j \in \Ngamma} \sbd_{\gamma,t}(X^j_{[0,t]},\Dbd_{[0,t]}) \Big\|^2 \le \frac{a_0}{N}, \quad \EmbPbar \| \Tmc^N_1(t) \|^2 \le \frac{a_0}{N^2}.
	\end{equation*}
\end{lemma}

\begin{proof}
	For fixed $r \in \Nmb$ and $t \in [0,T]$, it follows by the boundedness of $\bbar_{(1)}$ and $\sigmabar_{\cdot,(1)}$ that
	\begin{align*}
		& \EmbPbar \| \Emc_t \|^r \\
		& \quad = \EmbPbar \Big\| I_m + \int_0^t \bbar_{(1)}(\Rbd_s) \Emc_s \, ds + \sum_{k=1}^m \int_0^t \sigmabar_{k,(1)}(\Rbd_s) \Emc_s \, d\Wbar^k_s \Big\|^r \\
		& \quad \le 3^{r-1} \| I_m \|^r + (3t)^{r-1} \int_0^t \EmbPbar \| \bbar_{(1)}(\Rbd_s) \Emc_s \|^r \, ds + (3mt)^{r-1} \sum_{k=1}^m \int_0^t \EmbPbar \| \sigmabar_{k,(1)}(\Rbd_s) \Emc_s \|^r \, ds \\
		& \quad \le \kappa_1 \int_0^t \EmbPbar \| \Emc_s \|^r \, ds + \kappa_1.
	\end{align*}
	By Gronwall's inequality,
	\begin{equation} \label{eqchap3commonfactor:bound of Emb Emc}
		\sup_{t \in [0,T]} \EmbPbar \| \Emc_t \|^r < \infty.
	\end{equation}
	Using Lemma~\ref{lemchap3commonfactor:SDE result}, it follows by a similar argument that, for each $r \in \Nmb$
	\begin{equation} \label{eqchap3commonfactor:bound of Emb Emc inverse}
		\sup_{t \in [0,T]} \EmbPbar \| \Emc_t^{-1} \|^r < \infty.
	\end{equation}
	From~\eqref{eqchap3commonfactor:s gamma t},~\eqref{eqchap3commonfactor:bound of Emb Emc},~\eqref{eqchap3commonfactor:bound of Emb Emc inverse}, and boundedness of $\bbar_{(2),\cdot}, \sigmabar_{\cdot,(2),\cdot}, \sigmabar_{\cdot,(1)}$, it follows now that for each $\gamma \in \Kbd$,
	\begin{equation*}
		\sup_{t \in [0,T]} \sup_{j \in \Ngamma} \EmbPbar \| \sbd_{\gamma,t}(X^j_{[0,t]},\Dbd_{[0,t]}) \|^r < \infty.
	\end{equation*}
	Once again, by boundedness of $\sigmabar_{\cdot,(2),\cdot}$ and~\eqref{eqchap3commonfactor:bound of Emb Emc},~\eqref{eqchap3commonfactor:bound of Emb Emc inverse}, we get that
	\begin{align*}
		& \EmbPbar \Big\| \sum_{\gamma=1}^K \frac{1}{N_\gamma} \sum_{j \in \Ngamma}  \sum_{k=1}^m \Emc_t \int_0^t \Emc_s^{-1} \sigmabar^c_{k,(2),\gamma}(\Rbd_s,X^j_s) \, d\Wbar^k_s \Big\|^2 \\
		& \quad \le \big( \EmbPbar \| \Emc_t \|^4 \big)^{1/2} \Big( \EmbPbar \Big\|  \sum_{k=1}^m \int_0^t \Emc_s^{-1} \sum_{\gamma=1}^K \frac{1}{N_\gamma} \sum_{j \in \Ngamma} \sigmabar^c_{k,(2),\gamma}(\Rbd_s,X^j_s) \, d\Wbar^k_s \Big\|^4 \Big)^{1/2} \\
		& \quad \le \kappa_2 \Big[  \sum_{k=1}^m \int_0^t \EmbPbar \Big( \| \Emc_s^{-1} \|^4 \Big\| \sum_{\gamma=1}^K \frac{1}{N_\gamma} \sum_{j \in \Ngamma} \sigmabar^c_{k,(2),\gamma}(\Rbd_s,X^j_s) \Big\|^4 \Big) \, ds \Big]^{1/2} \\
		& \quad \le \kappa_2 \Big[ \sum_{k=1}^m \int_0^t \big( \EmbPbar \| \Emc_s^{-1} \|^8 \big)^{1/2} \Big( \EmbPbar \Big\| \sum_{\gamma=1}^K \frac{1}{N_\gamma} \sum_{j \in \Ngamma} \sigmabar^c_{k,(2),\gamma}(\Rbd_s,X^j_s) \Big\|^8 \Big)^{1/2} \, ds \Big]^{1/2} \\
		& \quad \le \frac{\kappa_3}{N},
	\end{align*}
	where the last inequality is from Lemma~\ref{lemchap3commonfactor:order}.
	Similarly, by boundedness of $\bbar_{(2),\cdot}, \sigmabar_{\cdot,(1)}, \sigmabar_{\cdot,(2),\cdot}$ and Lemma~\ref{lemchap3commonfactor:order}, we have that
	\begin{align*}
		& \EmbPbar \Big\| \sum_{\gamma=1}^K \frac{1}{N_\gamma} \sum_{j \in \Ngamma} \Big( \Emc_t \int_0^t \Emc_s^{-1} \bbar^c_{(2),\gamma}(\Rbd_s,X^j_s) \, ds \\
		& \hspace{20 mm} -  \sum_{k=1}^m \Emc_t \int_0^t \Emc_s^{-1} \sigmabar_{k,(1)}(\Rbd_s) \sigmabar^c_{k,(2),\gamma}(\Rbd_s,X^j_s) \, ds \Big) \Big\|^2 \le \frac{\kappa_4}{N}.
	\end{align*}
	Combining the above two observations and recalling the definition of $\sbd_{\gamma,t}$ from~\eqref{eqchap3commonfactor:s gamma t}, we have
	\begin{equation*}
		\EmbPbar \Big\| \sum_{\gamma=1}^K \frac{1}{N_\gamma} \sum_{j \in \Ngamma} \sbd_{\gamma,t}(X^j_{[0,t]},\Dbd_{[0,t]}) \Big\|^2 \le \frac{\kappa_5}{N}.
	\end{equation*}
	A similar argument using Condition~\ref{condchap3commonfactor:cond2},~\eqref{eqchap3commonfactor:bound of Emb Emc},~\eqref{eqchap3commonfactor:bound of Emb Emc inverse}, Lemmas~\ref{lemchap3commonfactor:order} and~\ref{lemchap3commonfactor:expectation of difference of Y} shows that
$\EmbPbar \| \Tmc^N_1(t) \|^2 \le \frac{\kappa_6}{N^2}$.
	The result follows.
\end{proof}


\subsection{Asymptotics of $J^{N,1}(T)$} \label{secchap3commonfactor:asymptotics of JN1}

In this section we analyze the term $J^{N,1}(T)$ defined in~\eqref{eqchap3commonfactor:JN1}.
Recall $\Smc_{\alpha\gamma} = \{(i,k) \in \Nalpha \times \Ngamma : i \ne k\}$ defined in Section~\ref{secchap2multitype:asymptotics of JN} for $\alpha, \gamma \in \Kbd$.

\begin{lemma} \label{lemchap3commonfactor:JN1}
	\begin{equation*}
		J^{N,1}(T) = \sum_{\alpha,\gamma=1}^K \frac{1}{N_\gamma} \sum_{(i,j) \in \Smc_{\alpha\gamma}} \int_0^T \fbd_{\alpha\gamma,t}(X^i_t,X^j_{[0,t]},\Dbd_{[0,t]}) \cdot d\Wit + \Rmc^N_1,
	\end{equation*}
	where $\Rmc^N_1 \to 0$ in probability under $\Pmbbar^N$, as $N \to \infty$.
\end{lemma}

\begin{proof}
	Note that for each $\alpha \in \Kbd$ and $i \in \Nalpha$,
	\begin{align} \label{eqchap3commonfactor:difference of b alpha}
		\begin{aligned}
			& b_\alpha(\Xit,\Rbd^N_t) - b_\alpha(\Xit,\Rbd_t) \\
			& \quad = b_{\alpha,(1)}(\Xit,\Rbd_t) (Y^N_t - Y_t) + \sum_{\gamma=1}^K \langle b_{\alpha\gamma,(2)}(\Xit,\Rbd_t,\cdot), (\mu^{\gamma,N}_t - \mu^\gamma_t) \rangle + \theta_{b_\alpha}(\Xit,\Rbd_t,\Rbd^N_t).
		\end{aligned}
	\end{align}
	For the last term in above display, we have
	\begin{align}
		& \EmbPbar \Big( \sum_{\alpha=1}^K \sum_{i \in \Nalpha} \int_0^T \theta_{b_\alpha}(\Xit,\Rbd_t,\Rbd^N_t) \cdot d\Wit \Big)^2 \notag \\
		& \quad = \sum_{\alpha=1}^K \sum_{i \in \Nalpha} \int_0^T \EmbPbar \| \theta_{b_\alpha}(\Xit,\Rbd_t,\Rbd^N_t) \|^2 \, dt \notag \\
		& \quad \le \kappa_1 \sum_{\alpha=1}^K \sum_{i \in \Nalpha} \int_0^T \EmbPbar \| Y^N_t - Y_t \|^4 \, dt \notag \\
		& \qquad + \kappa_1 \sum_{\alpha,\gamma=1}^K \sum_{i \in \Nalpha} \sum_{f \in \Jmc_b} \int_0^T \EmbPbar | \langle f(\Xit,Y_t,\cdot),(\mu^{\gamma,N}_t - \mu^\gamma_t) \rangle |^4 \, dt \notag \\
		& \quad \le \frac{\kappa_2}{N} \to 0,  \label{eqchap3commonfactor:first term in difference of b alpha}
	\end{align}
	where the first inequality is from~\eqref{eqchap3commonfactor:assumptions of theta b} and the last inequality follows from Lemmas~\ref{lemchap3commonfactor:order} and~\ref{lemchap3commonfactor:expectation of difference of Y}.
	
	Now consider the second term on the right side of~\eqref{eqchap3commonfactor:difference of b alpha}:
	\begin{align*}
		& \sum_{\alpha=1}^K \sum_{i \in \Nalpha} \int_0^T \sum_{\gamma=1}^K \langle b_{\alpha\gamma,(2)}(\Xit,\Rbd_t,\cdot), (\mu^{\gamma,N}_t - \mu^\gamma_t) \rangle \cdot d\Wit \\
		& \, = \sum_{\alpha,\gamma=1}^K \frac{1}{N_\gamma} \sum_{(i,j) \in \Smc_{\alpha\gamma}} \int_0^T b^c_{\alpha\gamma,(2)}(\Xit,\Rbd_t,\Xjt) \cdot d\Wit + \sum_{\alpha=1}^K \frac{1}{N_\alpha} \sum_{i \in \Nalpha} \int_0^T b^c_{\alpha\alpha,(2)}(\Xit,\Rbd_t,\Xit) \cdot d\Wit.
	\end{align*}
	Using the boundedness of $b_{\cdot \cdot,(2)}$ it follows that,
	\begin{equation*}
		\EmbPbar \Big( \sum_{\alpha=1}^K \frac{1}{N_\alpha} \sum_{i \in \Nalpha} \int_0^T b^c_{\alpha\alpha,(2)}(\Xit,\Rbd_t,\Xit) \cdot d\Wit \Big)^2 \to 0.
	\end{equation*}
	
	Finally consider the first term on the right side of~\eqref{eqchap3commonfactor:difference of b alpha}.
	It follows from Lemma~\ref{lemchap3commonfactor:difference of Y} that
	\begin{align*}
		& \quad \sum_{\alpha=1}^K \sum_{i \in \Nalpha} \int_0^T b_{\alpha,(1)}(\Xit,\Rbd_t) (Y^N_t - Y_t) \cdot d\Wit \\
		& = \sum_{\alpha,\gamma=1}^K \frac{1}{N_\gamma} \sum_{(i,j) \in \Smc_{\alpha\gamma}} \int_0^T b_{\alpha,(1)}(\Xit,\Rbd_t) \sbd_{\gamma,t}(X^j_{[0,t]},\Dbd_{[0,t]}) \cdot d\Wit \\
		& \quad + \sum_{\alpha=1}^K \frac{1}{N_\alpha} \sum_{i \in \Nalpha} \int_0^T b_{\alpha,(1)}(\Xit,\Rbd_t) \sbd_{\alpha,t}(X^i_{[0,t]},\Dbd_{[0,t]}) \cdot d\Wit \\
		& \quad + \sum_{\alpha=1}^K \sum_{i \in \Nalpha} \int_0^T b_{\alpha,(1)}(\Xit,\Rbd_t) \Tmc^N_1(t) \cdot d\Wit.
	\end{align*}
	By boundedness of $b_{\cdot,(1)}$ and Lemma~\ref{lemchap3commonfactor:expectation of Emc Smc and T1}, we have
	\begin{align*}
		 \EmbPbar \Big\| \sum_{\alpha=1}^K \sum_{i \in \Nalpha} \int_0^T b_{\alpha,(1)}(\Xit,\Rbd_t) \Tmc^N_1(t) \cdot d\Wit \Big\|^2 
		\le \frac{\kappa_3}{N}.
	\end{align*}
	and
	\begin{align*}
		 \EmbPbar \Big\| \sum_{\alpha=1}^K \frac{1}{N_\alpha} \sum_{i \in \Nalpha} \int_0^T b_{\alpha,(1)}(\Xit,\Rbd_t) \sbd_{\alpha,t}(X^i_{[0,t]},\Dbd_{[0,t]}) \cdot d\Wit \Big\|^2 
		\le \frac{\kappa_4}{N}.
	\end{align*}
	Result now follows by combining above observations and recalling $\fbd_{\alpha\gamma,t}$ defined in~\eqref{eqchap3commonfactor:f alpha gamma t}.
\end{proof}


\subsection{Asymptotics of $J^{N,2}(T)$} \label{secchap3commonfactor:asymptotics of JN2} 

In this section we analyze the term $J^{N,2}(T)$ defined in~\eqref{eqchap3commonfactor:JN2}.
We will need some notations.

Let $x \in \Cmc_d$ and $\zbd = (x^{(1)},x^{(2)},\dbd) \in \Cmb_{\Rmb^{2d+2m} \times [\Pmc(\Rmb^d)]^K}[0,T]$ with $\dbd = (w,\rbd) = (w,y,\nubd)$.
Define for $\alpha, \beta, \gamma \in \Kbd$, functions $\sbd_{\alpha\beta\gamma,i}$, $i=1,2,3$, and $\sbd_{\alpha\beta\gamma}$ from $\Cmb_{\Rmb^{3d+2m} \times [\Pmc(\Rmb^d)]^K}[0,T]$ to $\Rmb$ as follows:
\begin{align}
	\sbd_{\alpha\beta\gamma,1}(x,\zbd) & = \int_0^T b_{\alpha,(1)}(x_t,\rbd_t)\sbd_{\beta,t}(x^{(1)}_{[0,t]},\dbd_{[0,t]}) \cdot b_{\alpha,(1)}(x_t,\rbd_t)\sbd_{\gamma,t}(x^{(2)}_{[0,t]},\dbd_{[0,t]}) \, dt, \notag \\
	\sbd_{\alpha\beta\gamma,2}(x,\zbd) & = \int_0^T \Big( b_{\alpha,(1)}(x_t,\rbd_t)\sbd_{\beta,t}(x^{(1)}_{[0,t]},\dbd_{[0,t]}) \cdot b^c_{\alpha\gamma,(2)}(x_t,\rbd_t,x^{(2)}_t) \notag \\
	& \hspace{20 mm} + b_{\alpha,(1)}(x_t,\rbd_t)\sbd_{\gamma,t}(x^{(2)}_{[0,t]},\dbd_{[0,t]}) \cdot b^c_{\alpha\beta,(2)}(x_t,\rbd_t,x^{(1)}_t) \Big) \, dt, \notag \\
	\sbd_{\alpha\beta\gamma,3}(x,\zbd) & = \int_0^T b^c_{\alpha\beta,(2)}(x_t,\rbd_t,x^{(1)}_t) \cdot b^c_{\alpha\gamma,(2)}(x_t,\rbd_t,x^{(2)}_t) \, dt, \notag \\
	\sbd_{\alpha\beta\gamma}(x,\zbd) & = \sum_{i=1}^3 \sbd_{\alpha\beta\gamma,i}(x,\zbd). \label{eqchap3commonfactor:s alpha beta gamma}
\end{align}
Note that 
\begin{equation} \label{eqchap3commonfactor:inner product of f}
	\sbd_{\alpha\beta\gamma}(x,\zbd) = \int_0^T \fbd_{\alpha\beta,t}(x_t,x^{(1)}_{[0,t]},\dbd_{[0,t]}) \cdot \fbd_{\alpha\gamma,t}(x_t,x^{(2)}_{[0,t]},\dbd_{[0,t]}) \, dt,
\end{equation}
where $\fbd_{\alpha\beta,t}$ is as in~\eqref{eqchap3commonfactor:f alpha gamma t}.
Define $\mbd_{\alpha\beta\gamma} : \Cmb_{\Rmb^{2d+2m} \times [\Pmc(\Rmb^d)]^K}[0,T] \to \Rmb$ as
\begin{equation} \label{eqchap3commonfactor:m alpha beta gamma}
	\mbd_{\alpha\beta\gamma}(\zbd) = \int_0^T \int_{\Rmb^d} \fbd_{\alpha\beta,t}(x',x^{(1)}_{[0,t]},\dbd_{[0,t]}) \cdot \fbd_{\alpha\gamma,t}(x',x^{(2)}_{[0,t]},\dbd_{[0,t]}) \, \nu_{\alpha,t}(dx') \, dt,
\end{equation}
and let $\sbd_{\alpha\beta\gamma}^c : \Cmb_{\Rmb^{3d+2m} \times [\Pmc(\Rmb^d)]^K}[0,T] \to \Rmb$ be given as
\begin{equation*}
	\sbd_{\alpha\beta\gamma}^c(x,\zbd) = \sbd_{\alpha\beta\gamma}(x,\zbd) - \mbd_{\alpha\beta\gamma}(\zbd).
\end{equation*}
Recall $\Dbd$ introduced above Lemma~\ref{lemchap3commonfactor:difference of Y}.
The following lemma gives a useful representation for $J^{N,2}(T)$.

\begin{lemma} \label{lemchap3commonfactor:JN2}
	\begin{equation*}
		J^{N,2}(T) = \sum_{\alpha,\beta,\gamma=1}^K \frac{N_\alpha}{N_\beta N_\gamma} \sum_{(j,k) \in \Smc_{\beta\gamma}} \mbd_{\alpha\beta\gamma}(X^j,X^k,\Dbd) + \sum_{\alpha,\gamma=1}^K \frac{N_\alpha}{N_\gamma^2} \sum_{j \in \Ngamma} \mbd_{\alpha\gamma\gamma}(X^j,X^j,\Dbd) + \Rmc^N_2,
	\end{equation*}
	where $\Rmc^N_2 \to 0$ in probability under $\Pmbbar^N$, as $N \to \infty$.
\end{lemma}

\begin{proof}
	Note that for each $\alpha \in \Kbd$ and $i \in \Nalpha$,
	\begin{align}
		& \| b_\alpha(\Xit,\Rbd^N_t) - b_\alpha(\Xit,\Rbd_t) \|^2 \notag \\
		& \: \: = \Big\| b_{\alpha,(1)}(\Xit,\Rbd_t) (Y^N_t - Y_t) + \sum_{\gamma=1}^K \langle b_{\alpha\gamma,(2)}(\Xit,\Rbd_t,\cdot), (\mu^{\gamma,N}_t - \mu^\gamma_t) \rangle + \theta_{b_\alpha}(\Xit,\Rbd_t,\Rbd^N_t) \Big\|^2 \notag \\
		& \: \: = \| b_{\alpha,(1)}(\Xit,\Rbd_t) (Y^N_t - Y_t) \|^2 + \Big\| \sum_{\gamma=1}^K \langle b_{\alpha\gamma,(2)}(\Xit,\Rbd_t,\cdot), (\mu^{\gamma,N}_t - \mu^\gamma_t) \rangle \Big\|^2 + \|\theta_{b_\alpha}(\Xit,\Rbd_t,\Rbd^N_t) \|^2 \notag \\
		& \: \: \quad + 2 b_{\alpha,(1)}(\Xit,\Rbd_t) (Y^N_t - Y_t) \cdot \sum_{\gamma=1}^K \langle b_{\alpha\gamma,(2)}(\Xit,\Rbd_t,\cdot), (\mu^{\gamma,N}_t - \mu^\gamma_t) \rangle + \Tmc^{N,i}_2(t), \label{eqchap3commonfactor:square of difference of b alpha}
	\end{align}
	where $\Tmc^{N,i}_2(t)$ consists of the remaining two crossproduct terms.
	Using~\eqref{eqchap3commonfactor:assumptions of theta b}, Lemma~\ref{lemchap3commonfactor:order} and~\ref{lemchap3commonfactor:expectation of difference of Y}, as for the proof of~\eqref{eqchap3commonfactor:first term in difference of b alpha}, we see that
	\begin{equation} \label{eqchap3commonfactor:estimates of third term in JN2}
		\EmbPbar \sum_{\alpha=1}^K \sum_{i \in \Nalpha} \int_0^T \|\theta_{b_\alpha}(\Xit,\Rbd_t,\Rbd^N_t) \|^2 \, dt \le \frac{\kappa_1}{N} \to 0 \text{ as } N \to \infty.
	\end{equation}
	Similar estimates together with Cauchy-Schwarz inequality show that
	\begin{equation} \label{eqchap3commonfactor:estimates of fifth term in JN2}
		\EmbPbar \sum_{\alpha=1}^K \sum_{i \in \Nalpha} \int_0^T | \Tmc^{N,i}_2(t) | \, dt \le \frac{\kappa_2}{\sqrt{N}}\to 0 \text{ as } N \to \infty.
	\end{equation}

	Next we study the first term on the right side of~\eqref{eqchap3commonfactor:square of difference of b alpha}.
	Using Lemma~\ref{lemchap3commonfactor:difference of Y}, we have
	\begin{align*}
		& \| b_{\alpha,(1)}(\Xit,\Rbd_t) (Y^N_t - Y_t) \|^2 \\
		& \quad = \Big\| b_{\alpha,(1)}(\Xit,\Rbd_t) \Big( \sum_{\gamma=1}^K \frac{1}{N_\gamma} \sum_{j \in \Ngamma} \sbd_{\gamma,t}(X^j_{[0,t]},\Dbd_{[0,t]}) + \Tmc^N_1(t) \Big) \Big\|^2 \\
		& \quad = \Big\| b_{\alpha,(1)}(\Xit,\Rbd_t) \sum_{\gamma=1}^K \frac{1}{N_\gamma} \sum_{j \in \Ngamma} \sbd_{\gamma,t}(X^j_{[0,t]},\Dbd_{[0,t]}) \Big\|^2 + \| b_{\alpha,(1)}(\Xit,\Rbd_t) \Tmc^N_1(t) \big) \|^2 + \Tmc^{N,i}_3(t),
	\end{align*}
	where $\Tmc^{N,i}_3(t)$ is the corresponding crossproduct term.
	Making use of the boundedness of $b_{\cdot,(1)}$ and Lemma~\ref{lemchap3commonfactor:expectation of Emc Smc and T1}, we can show that
	\begin{gather*}
		\EmbPbar \sum_{\alpha=1}^K \sum_{i \in \Nalpha} \int_0^T \| b_{\alpha,(1)}(\Xit,\Rbd_t) \Tmc^N_1(t) \|^2 \, dt \le \frac{\kappa_3}{N}, \\
		\EmbPbar \sum_{\alpha=1}^K \sum_{i \in \Nalpha} \int_0^T | \Tmc^{N,i}_3(t) | \, dt \le \frac{\kappa_4}{\sqrt{N}}.
	\end{gather*}
	Thus recalling the definition of $\sbd_{\alpha\beta\gamma,1}$, we have
	\begin{align} \label{eqchap3commonfactor:studying first term in JN2}
		\begin{aligned} 
			& \sum_{\alpha=1}^K \sum_{i \in \Nalpha} \int_0^T \| b_{\alpha,(1)}(\Xit,\Rbd_t) (Y^N_t - Y_t) \|^2 \, dt \\
			& \quad = \sum_{\alpha=1}^K \sum_{i \in \Nalpha} \int_0^T \Big\| b_{\alpha,(1)}(\Xit,\Rbd_t) \sum_{\gamma=1}^K \frac{1}{N_\gamma} \sum_{j \in \Ngamma} \sbd_{\gamma,t}(X^j_{[0,t]},\Dbd_{[0,t]}) \Big\|^2 \, dt + \tilde{\Rmc}^N_1 \\
			& \quad = \sum_{\alpha,\beta,\gamma=1}^K \frac{1}{N_\beta N_\gamma} \sum_{i \in \Nalpha, j \in \Nbeta, k \in \Ngamma} \sbd_{\alpha\beta\gamma,1}(X^i,X^j,X^k,\Dbd) + \tilde{\Rmc}^N_1,
		\end{aligned}
	\end{align}
	where $\EmbPbar | \tilde{\Rmc}^N_1 | \to 0$ as $N \to \infty$.

	We now consider the second term on the right side of~\eqref{eqchap3commonfactor:square of difference of b alpha}.
	Recalling the definition of $\sbd_{\alpha\beta\gamma,3}$, we have
	\begin{align} \label{eqchap3commonfactor:studying second term in JN2}
		\begin{aligned} 
			& \sum_{\alpha=1}^K \sum_{i \in \Nalpha} \int_0^T \Big\| \sum_{\gamma=1}^K \langle b_{\alpha\gamma,(2)}(\Xit,\Rbd_t,\cdot), (\mu^{\gamma,N}_t - \mu^\gamma_t) \rangle \Big\|^2 \, dt \\
			& \qquad = \sum_{\alpha,\beta,\gamma=1}^K \frac{1}{N_\beta N_\gamma} \sum_{i \in \Nalpha, j \in \Nbeta, k \in \Ngamma} \sbd_{\alpha\beta\gamma,3}(X^i,X^j,X^k,\Dbd).
		\end{aligned}
	\end{align}
	
	Finally we consider the crossproduct term on the right side of~\eqref{eqchap3commonfactor:square of difference of b alpha}.
	Using Lemma~\ref{lemchap3commonfactor:difference of Y},
	\begin{align*}
		& \sum_{\alpha=1}^K \sum_{i \in \Nalpha} 2 b_{\alpha,(1)}(\Xit,\Rbd_t) (Y^N_t - Y_t) \cdot \sum_{\gamma=1}^K \langle b_{\alpha\gamma,(2)}(\Xit,\Rbd_t,\cdot), (\mu^{\gamma,N}_t - \mu^\gamma_t) \rangle \\
		& \quad = \sum_{\alpha,\beta,\gamma=1}^K \frac{2}{N_\beta N_\gamma} \sum_{i \in \Nalpha, j \in \Nbeta, k \in \Ngamma} b_{\alpha,(1)}(\Xit,\Rbd_t) \sbd_{\beta,t}(X^j_{[0,t]},\Dbd_{[0,t]}) \cdot b_{\alpha\gamma,(2)}^c(\Xit,\Rbd_t,\Xkt) \\
		& \qquad + \sum_{\alpha,\gamma=1}^K \sum_{i \in \Nalpha} \Big( b_{\alpha,(1)}(\Xit,\Rbd_t) \Tmc^N_1(t) \cdot \frac{2}{N_\gamma} \sum_{k \in \Ngamma} b_{\alpha\gamma,(2)}^c(\Xit,\Rbd_t,\Xkt) \Big) \\
		& \quad \equiv \Tmc^N_4(t) + \Tmc^N_5(t).
	\end{align*}
	Using boundedness of $b_{\cdot,(1)}, b_{\cdot\cdot,(2)}$, Lemma~\ref{lemchap3commonfactor:order},~\ref{lemchap3commonfactor:expectation of Emc Smc and T1} and Cauchy-Schwarz inequality, we see that
	\begin{equation*}
		\EmbPbar \int_0^T | \Tmc^N_5(t) | \, dt \le \frac{\kappa_5}{\sqrt{N}} \to 0 \text{ as } N \to \infty.
	\end{equation*}
	For the term $\Tmc^N_4(t)$, recalling the definition of $\sbd_{\alpha\beta\gamma,2}$ and using elementary symmetry properties, we have
	\begin{align*}
		\int_0^T \Tmc^N_4(t) \, dt
		= \sum_{\alpha,\beta,\gamma=1}^K \frac{1}{N_\beta N_\gamma} \sum_{i \in \Nalpha, j \in \Nbeta, k \in \Ngamma} \sbd_{\alpha\beta\gamma,2}(X^i,X^j,X^k,\Dbd).
	\end{align*}
	Thus we have
	\begin{align}
		& \sum_{\alpha=1}^K \sum_{i \in \Nalpha} \int_0^T 2 b_{\alpha,(1)}(\Xit,\Rbd_t) (Y^N_t - Y_t) \cdot \sum_{\gamma=1}^K \langle b_{\alpha\gamma,(2)}(\Xit,\Rbd_t,\cdot), (\mu^{\gamma,N}_t - \mu^\gamma_t) \rangle \, dt \notag \\
		& \qquad = \sum_{\alpha,\beta,\gamma=1}^K \frac{1}{N_\beta N_\gamma} \sum_{i \in \Nalpha, j \in \Nbeta, k \in \Ngamma} \sbd_{\alpha\beta\gamma,2}(X^i,X^j,X^k,\Dbd) + \tilde{\Rmc}^N_2, \label{eqchap3commonfactor:estimates of forth term in JN2}
	\end{align}
	where $\tilde{\Rmc}^N_2 \to 0$ in probability as $N \to \infty$.
			
	Combining~\eqref{eqchap3commonfactor:square of difference of b alpha} -~\eqref{eqchap3commonfactor:estimates of forth term in JN2} and recalling the definition of $\sbd_{\alpha\beta\gamma}$ in~\eqref{eqchap3commonfactor:s alpha beta gamma}, we have
	\begin{align}
		J^{N,2}(T) & = \sum_{\alpha,\beta,\gamma=1}^K \frac{1}{N_\beta N_\gamma} \sum_{i \in \Nalpha, j \in \Nbeta, k \in \Ngamma} \sbd_{\alpha\beta\gamma}(X^i,X^j,X^k,\Dbd) + \tilde{\Rmc}^N_3 \notag \\
		& = \sum_{\alpha,\beta,\gamma=1}^K \frac{1}{N_\beta N_\gamma} \sum_{i \in \Nalpha, j \in \Nbeta, k \in \Ngamma} \sbd^c_{\alpha\beta\gamma}(X^i,X^j,X^k,\Dbd) + \tilde{\Rmc}^N_3 \label{eqchap3commonfactor:JN2 estimate} \\
		& \quad + \sum_{\alpha,\beta,\gamma=1}^K \frac{N_\alpha}{N_\beta N_\gamma} \sum_{(j,k) \in \Smc_{\beta\gamma}} \mbd_{\alpha\beta\gamma}(X^j,X^k,\Dbd) + \sum_{\alpha,\gamma=1}^K \frac{N_\alpha}{N_\gamma^2} \sum_{j \in \Ngamma} \mbd_{\alpha\gamma\gamma} (X^j,X^j,\Dbd), \notag
	\end{align}
	where $\mbd_{\alpha\beta\gamma}$ is as defined in~\eqref{eqchap3commonfactor:m alpha beta gamma} and $\tilde{\Rmc}^N_3 \to 0$ in probability as $N \to \infty$.
	From the boundedness of second moment of $\sbd^c_{\alpha\beta\gamma}$ (which follows from Lemma~\ref{lemchap3commonfactor:expectation of Emc Smc and T1}), conditional independence of $X^i, X^j, X^k$ for distinct indices $i \in \Nalpha, j \in \Nbeta , k \in \Ngamma$, and the fact that for all $(x,\zbd) \in \Cmb_{\Rmb^{3d+2m} \times [\Pmc(\Rmb^d)]^K}[0,T]$
	\begin{equation*}
		\EmbPbar \sbd^c_{\alpha\beta\gamma}(X^i,x^{(1)},x^{(2)},\dbd) = \EmbPbar \sbd^c_{\alpha\beta\gamma}(x,X^j,x^{(2)},\dbd_{[0,t]}) = \EmbPbar \sbd^c_{\alpha\beta\gamma}(x,x^{(1)},X^k,\dbd) = 0,
	\end{equation*}
	it follows that the first term on right side of~\eqref{eqchap3commonfactor:JN2 estimate} converges to $0$ in probability as $N \to \infty$, which completes the proof.
\end{proof}


\subsection{Combining contributions from $J^{N,1}(T)$ and $J^{N,2}(T)$} \label{secchap3commonfactor:combining contributions from JN1 and JN2}

In this section we will combine Lemmas~\ref{lemchap3commonfactor:JN1} and~\ref{lemchap3commonfactor:JN2} to study the asymptotics of the exponent on the left side of~\eqref{eqchap3commonfactor:convergence of characteristic function}.
Recall $\mbd_{\alpha\beta\gamma}$ defined in~\eqref{eqchap3commonfactor:m alpha beta gamma} and canonical maps $X_*$, $Y_*$, $\Wbar_*$, $\mubd_*$ and $\Dbd_*$ defined in Section~\ref{secchap3commonfactor:canonical processes}.
For fixed $\omegabar \in \Omega_m$, define functions $l_{\omegabar}^{\alpha,\beta\gamma} \in L^2(\Omega_d \times \Omega_d, \rho_\alpha(\omegabar,\cdot) \times \rho_\gamma(\omegabar,\cdot))$ as
\begin{equation}
	l_{\omegabar}^{\alpha,\beta\gamma}(\omega,\omega') = \frac{\lambda_\alpha}{\sqrt{\lambda_\beta \lambda_\gamma}} \mbd_{\alpha\beta\gamma}(X_*(\omega),X_*(\omega'),\Dbd_*(\omegabar)), \quad (\omega,\omega') \in \Omega_d \times \Omega_d.
\end{equation}
Let $\lhat_{\omegabar}^{\alpha,\beta\gamma} \in L^2(\Omega_d^K \times \Omega_d^K, \rhohat(\omegabar,\cdot) \times \rhohat(\omegabar,\cdot))$ be {\em lifted} versions of $l_{\omegabar}^{\alpha,\beta\gamma}$, namely
\begin{equation*}
	\lhat_{\omegabar}^{\alpha,\beta\gamma}(\omega,\omega') = \lhat_{\omegabar}^{\alpha,\beta\gamma}(\omega_\beta,\omega_\gamma'), \quad \omega = (\omega_1,\dotsc,\omega_K) \in \Omega_d^K, \omega' = (\omega_1',\dotsc,\omega_K') \in \Omega_d^K,
\end{equation*}
and let $\lhat_{\omegabar} = \sum_{\alpha,\beta,\gamma=1}^K \lhat_{\omegabar}^{\alpha,\beta\gamma}$.
Recall $\rhohat$ and functions $\hhat_{\omegabar}^{\alpha\gamma}$ introduced in Section~\ref{secchap3commonfactor:some random integral operators}.
It follows from~\eqref{eqchap3commonfactor:inner product of f} and~\eqref{eqchap3commonfactor:m alpha beta gamma} that for $\alpha, \alpha', \beta, \gamma \in \Kbd$ and $\omega = (\omega_1,\dotsc,\omega_K)$, $\omega' = (\omega_1',\dotsc,\omega_K') \in \Omega_d^K$,
\begin{align}
	& \int_{\Omega_d^K} \hhat^{\alpha\beta}_\omegabar(\omega'',\omega) \hhat^{\alpha'\gamma}_\omegabar(\omega'',\omega') \, \rhohat(\omegabar,d\omega'') \notag \\
	& \quad = \one_{\{ \alpha = \alpha' \}} \frac{\lambda_\alpha}{\sqrt{\lambda_\beta \lambda_\gamma}} \int_{\Omega_d} \int_0^T \Big( \fbd_{\alpha\beta,t}(X_{*,t}(\omega_\alpha''), X_{*,[0,t]}(\omega_\beta), \Dbd_{*,[0,t]}(\omegabar)) \notag \\
	& \qquad \cdot \fbd_{\alpha\gamma,t}(X_{*,t}(\omega_\alpha''), X_{*,[0,t]}(\omega_\gamma'), \Dbd_{*,[0,t]}(\omegabar)) \Big) \, dt \, \rho_\alpha(\omegabar,d\omega_\alpha'') \notag \\
	& \quad = \one_{\{ \alpha = \alpha' \}} \frac{\lambda_\alpha}{\sqrt{\lambda_\beta \lambda_\gamma}} \mbd_{\alpha\beta\gamma}(X_*(\omega_\beta),X_*(\omega_\gamma'),\Dbd_*(\omegabar)) \notag \\
	& \quad = \one_{\{ \alpha = \alpha' \}} \lhat_{\omegabar}^{\alpha,\beta\gamma}(\omega,\omega'). \label{eqchap3commonfactor:relation between lhat alpha beta gamma and hhat}
\end{align}
Thus, with $\hhat_\omegabar$ as in Section~\ref{secchap3commonfactor:some random integral operators}, 
\begin{equation} \label{eqchap3commonfactor:relation between lhat and hhat}
	\lhat_{\omegabar}(\omega,\omega') = \int_{\Omega_d^K} \hhat_{\omegabar}(\omega'',\omega) \hhat_{\omegabar}(\omega'',\omega') \, \rhohat(\omegabar,d\omega'').
\end{equation}

Recall the integral operators $A_{\omegabar}^{\alpha\gamma}$ defined on $\Hmc_\omegabar$ introduced in Section~\ref{secchap3commonfactor:some random integral operators}. 
Then for $\alpha$, $\alpha'$, $\beta$, $\gamma \in \Kbd$, the operators $A_\omegabar^{\alpha\beta} (A_\omegabar^{\alpha'\gamma})^* : \Hmc_\omegabar \to \Hmc_\omegabar$ are given as follows:
For $g \in \Hmc_\omegabar$ and $\omega \in \Omega_d^K$,
\begin{align*}
	A_\omegabar^{\alpha\beta} (A_\omegabar^{\alpha'\gamma})^* g(\omega) 
	& = \int_{\Omega_d^K} \Big( \int_{\Omega_d^K} \hhat_\omegabar^{\alpha\beta}(\omega',\omega) \hhat_\omegabar^{\alpha'\gamma}(\omega',\omega'') \, \rhohat(\omegabar,d\omega') \Big) g(\omega'') \, \rhohat(\omegabar,d\omega'') \\
	& = \one_{\{\alpha = \alpha'\}} \int_{\Omega_d^K} \lhat_{\omegabar}^{\alpha,\beta\gamma}(\omega,\omega'') g(\omega'') \, \rhohat(\omegabar,d\omega''),
\end{align*}
where the last equality is from~\eqref{eqchap3commonfactor:relation between lhat alpha beta gamma and hhat}.
In particular, we have $A_\omegabar^{\alpha\beta} (A_\omegabar^{\alpha'\gamma})^* = 0$ if $\alpha \ne \alpha'$. 
Moreover, it follows from the display in~\eqref{eqchap3commonfactor:relation between lhat alpha beta gamma and hhat} that for $\alpha, \beta, \gamma \in \Kbd$,
\begin{align}
	& \textnormal{Trace}(A_\omegabar^{\alpha\beta} (A_\omegabar^{\alpha\gamma})^*) \notag \\
	& \quad = \int_{\Omega_d^K \times \Omega_d^K} \hhat_\omegabar^{\alpha\beta}(\omega,\omega') \hhat_\omegabar^{\alpha\gamma}(\omega,\omega') \, \rhohat(\omegabar,d\omega) \, \rhohat(\omegabar,d\omega') \notag \\
	& \quad = \frac{\lambda_\alpha}{\sqrt{\lambda_\beta \lambda_\gamma}} \int_{\Omega_d^K} \mbd_{\alpha\beta\gamma}(X_*(\omega_\beta'),X_*(\omega_\gamma'),\Dbd_*(\omegabar)) \, \rhohat(\omegabar,d\omega') \notag \\
	& \quad = \one_{\{\beta = \gamma\}} \frac{\lambda_\alpha}{\lambda_\gamma} \int_{\Omega_d} \mbd_{\alpha\gamma\gamma}(X_*(\omega_\gamma'),X_*(\omega_\gamma'),\Dbd_*(\omegabar)) \, \rho_\gamma(\omegabar,d\omega_\gamma'), \label{eqchap3commonfactor:Trace of AA*}
\end{align}
where the last equality holds because of the centered terms in the definition of $\sbd_{\gamma,t}$ in~\eqref{eqchap3commonfactor:s gamma t} and the definition of $\sbd_{\alpha\beta\gamma}$ in~\eqref{eqchap3commonfactor:s alpha beta gamma}.
Thus Trace$(A_\omegabar^{\alpha\beta} (A_\omegabar^{\alpha\gamma})^*) = 0$ if $\beta \ne \gamma$.
Define $\tau : \Omega_m \to \Rmb$ as $\tau(\omegabar) =$ Trace$(A_{\omegabar}A_{\omegabar}^*)$, where $A_\omegabar$ is the operator introduced below~\eqref{eqchap3commonfactor:A alpha gamma}.
The following lemma is immediate from the above calculations.

\begin{lemma} \label{lemchap3commonfactor:converging to Trace}
	\begin{equation*}
		\sum_{\alpha,\gamma=1}^K \frac{N_\alpha}{N_\gamma^2} \sum_{j \in \Ngamma} \mbd_{\alpha\gamma\gamma}(X^j,X^j,\Dbd) - \tau(\Vbar) \to 0
	\end{equation*}
	in probability under $\Pmbbar^N$ as $N \to \infty$.
\end{lemma}

\begin{proof}
	Note that for fixed $\omegabar \in \Omega_m$,
	\begin{align*}
		\tau(\omegabar) = \textnormal{Trace}(A_{\omegabar}A_{\omegabar}^*) & = \textnormal{Trace} \Big( (\sum_{\alpha=1}^K \sum_{\gamma=1}^K A_{\omegabar}^{\alpha\gamma})(\sum_{\alpha=1}^K \sum_{\gamma=1}^K A_{\omegabar}^{\alpha\gamma})^* \Big) = \sum_{\alpha=1}^K \sum_{\gamma=1}^K \textnormal{Trace}(A_{\omegabar}^{\alpha\gamma} (A_{\omegabar}^{\alpha\gamma})^*).
	\end{align*}
	It suffices to show for each pair of $\alpha, \gamma \in \Kbd$,
	\begin{equation*}
		\frac{N_\alpha}{N_\gamma^2} \sum_{j \in \Ngamma} \mbd_{\alpha\gamma\gamma}(X^j,X^j,\Dbd) - \textnormal{Trace}(A_{\Vbar}^{\alpha\gamma} (A_{\Vbar}^{\alpha\gamma})^*)
	\end{equation*}
	converges to $0$ in probability as $N \to \infty$.
	However, this property is immediate from~\eqref{eqchap3commonfactor:Trace of AA*} and the law of large numbers, since $\mbd_{\alpha\gamma\gamma}(X^j,X^j,\Dbd)$ is square integrable and conditional on $\Gmc$, $\{ X^j, j \in \Ngamma \}$ are i.i.d. with common distribution $\rho_\gamma(\Vbar,\cdot) \circ X_*^{-1}$.
\end{proof}

We will now use the results from Section~\ref{secchap2multitype:asymptotics of symmetric statistics} with $\Smb = \Omega_d^K$ and $\nu = \rhohat(\omegabar,\cdot), \omegabar \in \Omega_m$.
For each $\omegabar \in \Omega_m, k \ge 1$ and $f \in L^2_{sym}(\rhohat(\omegabar,\cdot)^{\otimes k})$ the MWI $I_k^\omegabar(f)$ is defined as in Section~\ref{secchap2multitype:asymptotics of symmetric statistics}. 
More precisely, let $\Amc^k$ be the collection of all measurable $f : \Omega_m \times (\Omega_d^K)^k \to \Rmb$ such that
\begin{equation*}
	\int_{(\Omega_d^K)^k} | f(\omegabar,\omega_1,\dotsc,\omega_k) |^2 \rhohat(\omegabar,d\omega_1) \dotsc \rhohat(\omegabar,d\omega_k) < \infty, \quad \Pbar \text{ a.e.\ } \omegabar
\end{equation*}
and for every permutation $\pi$ on $\{ 1,\dotsc,k \}$, $f(\omegabar,\omega_1,\dotsc,\omega_k) = f(\omegabar,\omega_{\pi(1)},\dotsc,\omega_{\pi(k)})$, $\Pbar \otimes \rhohat^{\otimes k}$ a.s., where $\Pbar \otimes \rhohat^{\otimes k}(d\omegabar,d\omega_1,\dotsc,d\omega_k) = \Pbar(d\omegabar) \prod_{i=1}^k \rhohat(\omegabar,d\omega_i)$.
Then there is a measurable space $(\Omega^*,\Fmc^*)$ and a regular conditional probability distribution $\lambda^* : \Omega_m \times \Fmc^* \to [0,1]$ such that on the probability space $(\Omega_m \times \Omega^*, \Bmc(\Omega_m) \otimes \Fmc^*, \Pbar \otimes \lambda^*)$, where
\begin{equation*}
	\Pbar \otimes \lambda^* (A \times B) = \int_A \lambda^* (\omegabar,B) \, \Pbar(d\omegabar), \quad A \times B \in \Bmc(\Omega_m) \otimes \Fmc^*,
\end{equation*}
there is a collection of real valued random variables $\{ I_k(f) : f \in \Amc^k, k \ge 1 \}$ with the properties that

$(a)$ For all $f \in \Amc^1$ the conditional distribution of $I_1(f)$ given $\Gmc^* = \Bmc(\Omega_m) \otimes \{ \emptyset, \Omega^* \}$ is Normal with mean $0$ and variance $\int_{\Omega_d^K} f^2(\omegabar,\omega) \, \rhohat(\omegabar,d\omega)$.

$(b)$ $I_k$ is (a.s.) linear map on $\Amc^k$.

$(c)$ For $f \in \Amc^k$ of the form
\begin{equation*}
	f(\omegabar,\omega_1,\dotsc,\omega_k) = \prod_{i=1}^k h(\omegabar,\omega_i), \quad \text{ s.t. } \quad \int_{\Omega_d^K} h^2(\omegabar,\omega) \, \rhohat(\omegabar,d\omega) < \infty, \quad \Pbar \text{ a.e.\ } \omegabar,
\end{equation*}
we have
\begin{equation*}
	I_k(f)(\omegabar,\omega^*) = \sum_{j=0}^{\lfloor k/2 \rfloor} (-1)^j C_{k,j} \Big( \int_{\Omega_d^K} h^2(\omegabar,\omega) \, \rhohat(\omegabar,d\omega) \Big)^j (I_1(h)(\omegabar,\omega^*))^{k-2j}, \quad \Pbar \otimes \lambda^* \text{ a.e.\ } (\omegabar,\omega^*)
\end{equation*}
and
\begin{equation*}
	\int_{\Omega^*} (I_k(f)(\omegabar,\omega^*))^2 \lambda^*(\omegabar,d\omega^*) = k! \Big( \int_{\Omega_d^K} h^2(\omegabar,\omega) \, \rhohat(\omegabar,d\omega) \Big)^k, \quad \Pbar \text{ a.e.\ } \omegabar,
\end{equation*}
where $C_{k,j}$ are as in~\eqref{eqchap2multitype:MWI formula}. 
We write $I_k(f)(\omegabar,\cdot)$ as $I_k^\omegabar(f)$.
With an abuse of notation, we will denote once more by $\Vbar_*$ the canonical process on $\Omega_m \times \Omega^*$, i.e.\ $\Vbar_*(\omegabar,\omega^*) = \omegabar$, for $(\omegabar,\omega^*) \in \Omega_m \times \Omega^*$.

From Lemmas~\ref{lemchap3commonfactor:JN1},~\ref{lemchap3commonfactor:JN2} and~\ref{lemchap3commonfactor:converging to Trace} it follows that
\begin{equation} \label{eqchap3commonfactor:representation of JN}
	J^N(T) = J^{N,1}(T) - \frac{1}{2} J^{N,2}(T) = \Smcbar^N - \frac{1}{2} \tau(\Vbar) + \Rmcbar^N,
\end{equation}
where
\begin{align*}
	\Smcbar^N & = \sum_{\alpha,\gamma=1}^K \frac{1}{N_\gamma} \sum_{(i,j) \in \Smc_{\alpha\gamma}} \int_0^T \fbd_{\alpha\gamma,t}(X^i_t,X^j_{[0,t]},\Dbd_{[0,t]}) \cdot d\Wit \\
	& \quad - \frac{1}{2} \sum_{\alpha,\beta,\gamma=1}^K \frac{N_\alpha}{N_\beta N_\gamma} \sum_{(j,k) \in \Smc_{\beta\gamma}} \mbd_{\alpha\beta\gamma}(X^j,X^k,\Dbd),
\end{align*}
and $\Rmcbar^N \to 0$ in probability as $N \to \infty$ under $\Pmbbar^N$.

Define $F : \Omega_m \times \Omega_d^K \times \Omega_d^K \to \Rmb$ as follows:
For $(\omegabar,\omega,\omega') \in \Omega_m \times \Omega_d^K \times \Omega_d^K$,
\begin{align*}
	F(\omegabar,\omega,\omega') & = \hhat_{\omegabar}(\omega,\omega') + \hhat_{\omegabar}(\omega',\omega) - \lhat_\omegabar(\omega,\omega') \\
	& = \hhat_{\omegabar}(\omega,\omega') + \hhat_{\omegabar}(\omega',\omega) - \int_{\Omega_d^K} \hhat_{\omegabar}(\omega'',\omega) \hhat_{\omegabar}(\omega'',\omega') \, \rhohat(\omegabar,d\omega''),
\end{align*}
where the second equality is from~\eqref{eqchap3commonfactor:relation between lhat and hhat}.
Note that $F \in \Amc^2$ and so $I_2(F)$ is a well defined random variable on $(\Omega_m \times \Omega^*, \Bmc(\Omega_m) \otimes \Fmc^*, \Pbar \otimes \lambda^*)$.
Recall the collection $\Amcbar_\alpha$, $\alpha \in \Kbd$, introduced in Section~\ref{secchap3commonfactor:CLT}.
For $\phi_\alpha \in \Amcbar_\alpha$, $m^\alpha_{\phi_\alpha}$ is defined as in~\eqref{eqchap3commonfactor:m alpha phi alpha}.
For such a $\phi_\alpha \in \Amcbar_\alpha$, $\hat{\Phi}^\alpha_\omegabar(\omega)$ is as defined in~\eqref{eqchap3commonfactor:Phi hat alpha}.
We denote $\Phibar^\alpha : \Omega_m \times \Omega_d^K \to \Rmb$ as $\Phibar^\alpha(\omegabar,\omega) = \hat{\Phi}^\alpha_\omegabar(\omega)$, namely for $\omegabar \in \Omega_m$ and $\omega = (\omega_1,\dotsc,\omega_K) \in \Omega_d^K$,
\begin{equation*}
	\Phibar^\alpha(\omegabar,\omega) = \hat{\Phi}_\omegabar^\alpha(\omega) = \phi_\alpha(X_*(\omega_\alpha)) - m^\alpha_{\phi_\alpha}(\omegabar).
\end{equation*}
Note that $\Phibar^\alpha \in \Amc^1$ and so $I_1(\Phibar^\alpha)$ is well defined.
Let for $\phi_\alpha \in \Amcbar_\alpha$, $\Vmctil^N_\alpha(\phi_\alpha)$ be as in~\eqref{eqchap3commonfactor:Vtil alpha phi alpha}.
From the definition of $\Gmc$ and $\Gmc^*$ it follows that there are maps $\Lbd_N$ and $\Lbd$ from $\Omega_m$ to $\Pmc(\Rmb^{K+1})$ such that
\begin{align*}
	\Lmc \Big( \big( \Vmctil^N_1(\phi_1), \dotsc, \Vmctil^N_K(\phi_K), \Smcbar^N \big) \Big| \Gmc \Big) & = \Lbd_N(\Vbar), \quad \Pmbbar^N \text{ a.s.}, \\
	\Lmc \Big( \big( I_1(\Phibar^1), \dotsc, I_1(\Phibar^K), \frac{1}{2} I_2(F) \big) \Big| \Gmc^* \Big) & = \Lbd(\Vbar_*), \quad \Pbar \otimes \lambda^* \text{ a.s.}
\end{align*}
From conditional independence of $\{ X^i \}$ it follows using Lemma~\ref{lemchap2multitype:unbalanced populations} that 
\begin{equation} \label{eqchap3commonfactor:convergence of conditional distribution}
	\Lbd_N(\omegabar) \to \Lbd(\omegabar) \text{ weakly, for } \Pbar \text{ a.e.\ } \omegabar.
\end{equation}
Define $\taubar : \Omega_m \times \Omega^* \to \Rmb$ as $\taubar(\omegabar,\omega^*) = \textnormal{Trace}(A_\omegabar A_\omegabar^*)$.
The following lemma is the key step.
\begin{lemma} \label{lemchap3commonfactor:convergence of V and S}
	As $N \to \infty$, $i \sum_{\alpha=1}^{K} \Vmctil^N_\alpha(\phi_\alpha) + J^{N,1}(T) - \frac{1}{2} J^{N,2}(T)$ converges in distribution to $i \sum_{\alpha=1}^{K} I_1(\Phibar^\alpha) + \frac{1}{2} I_2(F) - \frac{1}{2} \taubar$.
\end{lemma}

\begin{proof}
	Note that from~\eqref{eqchap3commonfactor:representation of JN},
	\begin{equation*}
		i \sum_{\alpha=1}^{K} \Vmctil^N_\alpha(\phi_\alpha) + J^{N,1}(T) - \frac{1}{2} J^{N,2}(T) = i \sum_{\alpha=1}^{K} \Vmctil^N_\alpha(\phi_\alpha) + \Smcbar^N - \frac{1}{2} \tau(\Vbar) + \Rmcbar^N,
	\end{equation*}
	where $\Rmcbar^N \to 0$ in probability as $N \to \infty$.
	Let $\lbd_N$ and $\lbd$ from $\Omega_m$ to $\Pmc(\Cmb)$ be such that
	\begin{align*}
		\lbd_N(\Vbar) & = \Lmc \Big( i \sum_{\alpha=1}^{K} \Vmctil^N_\alpha(\phi_\alpha) + \Smcbar^N - \frac{1}{2} \tau(\Vbar) \Big| \Gmc \Big), \quad \Pmbbar^N \text{ a.s.}, \\
		\lbd(\Vbar_*) & = \Lmc \Big( i \sum_{\alpha=1}^{K} I_1(\Phibar^\alpha) + \frac{1}{2} I_2(F) - \frac{1}{2} \taubar \Big| \Gmc^* \Big), \quad \Pbar \otimes \lambda^* \text{ a.s.}
	\end{align*}
	It follows from~\eqref{eqchap3commonfactor:convergence of conditional distribution} and definition of $\tau$, $\taubar$ that
	\begin{equation*}
		\lbd_N(\omegabar) \to \lbd(\omegabar) \text{ weakly, for } \Pbar \text{ a.e.\ } \omegabar.
	\end{equation*}
	The desired convergence is now immediate on combining above observations.
\end{proof}


%


\subsection{Completing the proof of Theorem~\ref{thmchap3commonfactor:CLT}} \label{secchap3commonfactor:completing the proof}

It follows from Lemma~$1.2$ of~\cite{ShigaTanaka1985} (cf.\ Lemma~\ref{lemappendixchap2multitype:Shiga and Tanaka} in Appendix~\ref{Appendixchap2multitype:Restating}) and Lemma~\ref{lemchap3commonfactor:Trace} that $\Pbar$ a.s.
\begin{equation*}
	\Emb_{\Pbar \otimes \lambda^*} \Big[ \exp \Big( \frac{1}{2} I_2(F) \Big) \Big| \Gmc^* \Big] = \exp \Big( \frac{1}{2} \textnormal{Trace} (A_{\Vbar_*}(A_{\Vbar_*})^*) \Big).
\end{equation*}
Recalling the definition of $\taubar$ below~\eqref{eqchap3commonfactor:convergence of conditional distribution} it follows that
\begin{equation*}
	\Emb_{\Pbar \otimes \lambda^*} \Big[ \exp \Big( \frac{1}{2} I_2(F) - \frac{1}{2} \taubar \Big) \Big]= 1.
\end{equation*}
Also, recall that
\begin{equation*}
	\EmbPbar \Big[ \exp \Big( J^{N,1}(T) - \frac{1}{2} J^{N,2}(T) \Big) \Big] = 1.
\end{equation*}
Using Lemma~\ref{lemchap3commonfactor:convergence of V and S} along with Scheff\'e's theorem we now have as in Section~\ref{secchap2multitype:completing the proof of CLT} that
\begin{align*}
	& \lim_{N \to \infty} \EmbPbar \Big[ \exp \Big( i \sum_{\alpha=1}^K \Vmctil^N_\alpha(\phi_\alpha) + J^{N,1}(T) - \frac{1}{2} J^{N,2}(T) \Big) \Big] \\
	& = \Emb_{\Pbar \otimes \lambda^*} \Big[ \exp \Big( i \sum_{\alpha=1}^{K} I_1(\Phibar^\alpha) + \frac{1}{2} I_2(F) - \frac{1}{2} \taubar \Big) \Big] \\
	& = \Emb_{\Pbar \otimes \lambda^*} \Big[ \Emb_{\Pbar \otimes \lambda^*} \Big( \exp(i \sum_{\alpha=1}^{K} I_1(\Phibar^\alpha) + \frac{1}{2} I_2(F) - \frac{1}{2} \taubar) \Big| \Gmc_* \Big) \Big] \\
	& = \int_{\Omega_m} \exp \Big( -\frac{1}{2} \| (I - A_\omegabar)^{-1} \sum_{\alpha = 1}^K \hat{\Phi}^\alpha_\omegabar \|^2_{L^2(\Omega_d^K,\rhohat(\omegabar,\cdot))} \Big) \, \Pbar(\omegabar),
\end{align*}
where the last equality is a consequence of Lemma~$1.3$ of~\cite{ShigaTanaka1985} (cf.\ Lemma~\ref{lemappendixchap2multitype:Shiga and Tanaka} in Appendix~\ref{Appendixchap2multitype:Restating}) and Lemma~\ref{lemchap3commonfactor:Trace}.
Thus we have proved~\eqref{eqchap3commonfactor:convergence of characteristic function}, which completes the proof of Theorem~\ref{thmchap3commonfactor:CLT}. \qed


\appendix

\section{A lemma on integral operators} \label{Appendixchap2multitype:Restating}

The following result is taken from Shiga-Tanaka~\cite{ShigaTanaka1985}.
Let $\Smb$ be a Polish space and $\nu \in \Pmc(\Smb)$.
Let $a(\cdot,\cdot) \in L^2(\nu \otimes \nu)$ and denote by $A$ the integral operator on $L^2(\nu)$ associated with $a$: $A \phi(x) = \int_{\Rmb^d} a(x,y) \phi(y) \, \nu(dy)$ for $x \in \Rmb^d$ and $\phi \in L^2(\nu)$.
Then $A$ is a Hilbert-Schmidt operator.
Also, $AA^*$, and for $n \ge 2$, $A^n$, are trace class operators.
The following lemma is taken from~\cite{ShigaTanaka1985}.

\begin{lemma} \label{lemappendixchap2multitype:Shiga and Tanaka}
	Suppose that \textnormal{Trace}$(A^n)=0$ for all $n \ge 2$.
	Then $\Emb [e^{\frac{1}{2} I_2(f)}] = e^{\frac{1}{2} \textnormal{Trace} (AA^*)}$, where $f(x,y) = a(x,y) + a(y,x) - \int_{\Rmb^d} a(x,z) a(y,z) \, \nu(dz)$, and $I_2(\cdot)$ is the MWI defined as in Section~\ref{secchap2multitype:asymptotics of symmetric statistics}.
	Moreover, $I-A$ is invertible and for any $\phi \in L^2(\nu)$, $\Emb [\exp(iI_1(\phi) + \frac{1}{2} I_2(f))] = \exp[-\frac{1}{2} (\| (I-A)^{-1} \phi \|^2_{L^2(\nu)} - \textnormal{Trace} (AA^*))]$, where $I$ is the identity operator on $L^2(\nu)$.
\end{lemma}


\section{Proof of Lemma~\ref{lemchap2multitype:convergence via approximation}} \label{Appendixchap2multitype:Proof of Lemma of convergence via approximation}

Let $f$ be a bounded Lipschitz function from $\Rmb^d$ to $\Rmb$ such that $\|f\|_\infty \le 1$ and $\|f\|_L \le 1$, where $\|\cdot\|_L$ is the Lipschitz norm.
Then 
\begin{align*}
	& |\Emb_{\Pmb_0} f(\xi_n) - \Emb_{\Pmb_0'} f(\eta)| \\
	& \quad \le |\Emb_{\Pmb_0} f(\xi_n) - \Emb_{\Pmb_0} f(\xi_{mn})| + |\Emb_{\Pmb_0} f(\xi_{mn}) - \Emb_{\Pmb_0'} f(\eta_m)| + |\Emb_{\Pmb_0'} f(\eta_m) - \Emb_{\Pmb_0'} f(\eta)| \\
	& \quad \le 2 \sup_{k \ge 1} \Emb_{\Pmb_0} \big(\|\xi_k - \xi_{mk}\| \wedge 1 \big) + |\Emb_{\Pmb_0} f(\xi_{mn}) - \Emb_{\Pmb_0'} f(\eta_m)| + 2 \Emb_{\Pmb_0'} \big( \|\eta_m - \eta\| \wedge 1 \big).
\end{align*}
Letting $n \to \infty$ gives us
\begin{align*}
	& \limsup_{n \to \infty} |\Emb_{\Pmb_0} f(\xi_n) - \Emb_{\Pmb_0'} f(\eta)| \\
	& \quad \le 2 \sup_{k \ge 1} \Emb_{\Pmb_0} \big(\|\xi_k - \xi_{mk}\| \wedge 1 \big) + \limsup_{n \to \infty} |\Emb_{\Pmb_0} f(\xi_{mn}) - \Emb_{\Pmb_0'} f(\eta_m)| + 2 \Emb_{\Pmb_0'} \big( \|\eta_m - \eta\| \wedge 1 \big) \\
	& \quad = 2 \sup_{k \ge 1} \Emb_{\Pmb_0} \big(\|\xi_k - \xi_{mk}\| \wedge 1 \big) + 2 \Emb_{\Pmb_0'} \big( \|\eta_m - \eta\| \wedge 1 \big).
\end{align*}
The result now follows on sending $m \to \infty$. \qed

\bigskip
{\bf Acknowledgements.}  Research supported in part by the National Science Foundation (DMS-1016441, DMS-1305120) and the Army Research Office (W911NF- 14-1-0331).




\bibliographystyle{plain}
\bibliography{reference}

{\sc
\bigskip
\noi
A. Budhiraja\\
Department of Statistics and Operations Research\\
University of North Carolina\\
Chapel Hill, NC 27599, USA\\
email: budhiraj@email.unc.edu
\skp

\noi
R. Wu\\
Department of Statistics and Operations Research\\
University of North Carolina\\
Chapel Hill, NC 27599, USA\\
email: wuruoyu@live.unc.edu

}

\end{document}